\documentclass{birkjour}
\usepackage{graphicx} 
\usepackage{amsthm}
\usepackage{amsmath}
\usepackage{amssymb}
\usepackage{todonotes}
\usepackage{caption}
\usepackage{subcaption}
\usepackage{float} 
\usepackage{hyperref}
\usepackage[page]{appendix}
\usepackage{verbatim}
\usepackage{enumitem}
\usepackage{xcolor}

\hypersetup{
    colorlinks=true,
    linkcolor=blue,
    citecolor=black,
    }

\newtheorem{theorem}{Theorem}[section]
\newtheorem{lemma}[theorem]{Lemma}
\newtheorem{corollary}[theorem]{Corollary}
\newtheorem{proposition}[theorem]{Proposition}

\theoremstyle{definition}
\newtheorem{definition}[theorem]{Definition}
\newtheorem{example}[theorem]{Example}
\newtheorem{remark}[theorem]{Remark}

\newcommand{\Z}{\mathbb{Z}}
\newcommand{\N}{\mathbb{N}}

\newcommand{\C}{\mathbb{C}}
\newcommand{\R}{\mathbb{R}}
\newcommand{\T}{\mathbb{T}}
\newcommand{\I}{\mathcal{I}}
\newcommand{\F}{\mathcal{F}}
\newcommand{\J}{\mathcal{J}}

\renewcommand{\O}{\mathcal{O}}
\renewcommand{\H}{\mathcal{H}}
\newcommand{\dist}{\mathrm{dist}}

\renewcommand{\Re}{\text{Re}\,}
\renewcommand{\Im}{\text{Im}\,}
\DeclareMathOperator{\Arg}{Arg}

\DeclareMathOperator{\ran}{R}
\DeclareMathOperator{\Span}{Span}
\renewcommand{\ker}{\text{N}}

\begin{document}

\title{The Whitney method of fundamental solutions with Lusin wavelets}
\author{Jakob Jonsson$\,^1$ \and Andreas Ros\'en$\,^2$  \and Emil Timlin}

\thanks{$^1\,$Corresponding author. \\$\,^2$Formerly Andreas Axelsson. Supported by the Swedish Research Council (Grant 2022-03996).}

\address{Jakob Jonsson\\Mathematical Sciences, Chalmers University of Technology and University of Gothenburg\\
SE-412 96 G{\"o}teborg, Sweden}
\email{jakob.jonsson@chalmers.se}

\address{Andreas Ros\'en\\Mathematical Sciences, Chalmers University of Technology and University of Gothenburg\\
SE-412 96 G{\"o}teborg, Sweden}
\email{andreas.rosen@chalmers.se}

\address{Emil Timlin\\Mathematical Sciences, Chalmers University of Technology and University of Gothenburg\\
SE-412 96 G{\"o}teborg, Sweden}
\email{timlin@chalmers.se}


\begin{abstract}
We establish the theoretical foundation for a variant of the method of fundamental solutions (MFS), where the source points $\{q_j\}_{j=1}^\infty$ accumulate towards the domain in a Whitney fashion, meaning that their separation is proportional to the distance to the domain.
We prove that the normalized Lusin wavelets $\psi_j(w) = b_j(w-q_j)^{-2}$ constitute a generalized basis, known as a frame, for the Hardy subspace of $L_2$-traces of holomorphic functions on the domain.
Consequently, our method, where $\psi_j$ are used as basis functions in the MFS, enables a numerically stable approximation of solutions to Laplace boundary value problems, even when the solutions lack analytic continuation across the boundary.
Despite the source points accumulating towards the domain, our computations achieve at least 12 digits of accuracy uniformly up to the boundary, including cases when the solution lacks analytic continuation or when the boundary has corners.
\end{abstract}
\maketitle

\section{Introduction}
\label{sec:intro}
In this paper, we provide the theory behind a version of the \emph{method of fundamental solutions}, in which source points $\{q_j\}_{j=1}^\infty$ are arranged in a Whitney fashion; that is, where the balls $B_j$, which are centered at $q_j$ and have radii proportional to the distance between $q_j$ and the boundary of the domain, constitute a uniformly locally finite cover of the complement of the domain.
This method provides a new constructive way to solve elliptic boundary value problems.
Additionally, we present numerical results demonstrating that this method outperforms the classical method of fundamental solutions for non-smooth data on Lipschitz domains.
To establish a basis for comparison in this introduction, we first survey the boundary integral equation method alongside the classical method of fundamental solutions.

The methods described here apply, suitably adjusted, to general elliptic boundary value problems on bounded domains in $\R^n$ for partial differential equations with constant coefficients.
As a model problem, we consider in this paper the interior Neumann problem for the Laplace equation in $\R^2 = \C$.
Given a bounded open domain $\Omega^+$ with boundary $\gamma$, and a real-valued function $g$ on $\gamma$, the interior Neumann problem is to find $u : \Omega^+ \to \R$ such that
\begin{equation}
    \label{eq:neumann_problem}
    \begin{cases}
        \Delta u = 0  \quad & \text{in } \Omega^+,\\
         \partial_\nu u = g & \text{on } \gamma,
    \end{cases}
\end{equation}
where $\partial_\nu u = \nu \cdot \nabla u$ is the normal derivative of $u$.
Below in~\eqref{eq:intro_cplx_neumann}, we shall reformulate~\eqref{eq:neumann_problem} as a boundary value problem for the (conjugate) gradient $f$ of $u$ in $\Omega^+$, which is in one-to-one correspondence with $u$, modulo constants. 
Our results, including numerical results, will be phrased in terms of $f$. 
However, as demonstrated in Examples~\ref{ex:smooth_data_u} and~\ref{ex:smooth_data_dirichlet}, our algorithm can easily be adapted to compute the potential $u$ itself.

One classical way to solve~\eqref{eq:neumann_problem} is the \emph{boundary integral equation} (BIE) method (see Kellogg~\cite[Ch.~XI]{Kellogg_1929}), using the \emph{single layer potential}
\begin{equation*}
    \widetilde Sh(z) = \frac{1}{2\pi}\int_\gamma \log|z-w| h(w)\,|dw|, \quad \quad z \in \Omega^+,
\end{equation*}
where $|dw|$ is the arc-length measure on $\gamma$, and $h$ is a real-valued function on $\gamma$.
By differentiating under the integral sign, it follows that $\widetilde Sh$ is harmonic in $\Omega^+$, so with the ansatz $ u = \widetilde Sh$, \eqref{eq:neumann_problem} reduces to finding $h : \gamma \to \R$ so that $\partial_\nu \widetilde Sh = g$ on $\gamma$. 
As the Neumann trace of the single layer potential is given by $\partial_\nu \widetilde S h|_\gamma = -\frac{1}{2}h + K^*h$,
where $K^*$ is the $L_2(\gamma)$-adjoint of the \emph{Neumann-Poincaré} operator
\begin{equation*}
    Kh(z) = \frac{1}{2\pi}\int_\gamma \nu(w)\cdot \nabla_w \log|z-w| h(w)\,|dw|, \quad \quad z \in \gamma,
\end{equation*}
it follows that $u = \widetilde Sh$ solves~\eqref{eq:neumann_problem} if $g = -\frac{1}{2}h+K^*h$.
A common numerical strategy for solving this boundary integral equation uses a Nyström discretization of $\gamma$, where integrals are approximated by sampling $\{w_m\}_{m=1}^M$ on $\gamma$, so that for $z \in \gamma \setminus \{w_m\}_{m=1}^M$,
\[\int_\gamma \nu(w)\cdot \nabla_w  \log|z-w| h(w)\,|dw| \approx  \sum_{m=1}^M  \nu(w_m)\cdot \nabla_w  \log|z-w_m| h(w_m) \alpha_m,\]
for some quadrature weights $\{\alpha_m\}_{m=1}^M$.
Using such a scheme, it is feasible to accurately find $h$ solving $g = -\frac{1}{2}h+K^*h$; see Atkinson~\cite[Ch.~4]{Atkinson_1997} for precise results.
Since $\widetilde S$ is bounded in the uniform topology,
computing $u = \widetilde Sh$ results in a high accuracy when $z$ is not too close the boundary $\gamma$, since then the interior error of $u$ is dominated by the error of $h$.
However, because the kernel of $\widetilde S$ has a singularity at the diagonal $w=z$, the numerical solution suffers from significant errors for points $z$ near the boundary.
This loss of accuracy near the boundary is the primary drawback of the BIE method.
While Helsing, Ojala~\cite{HelsingOjala_2008}, and Barnett~\cite{Barnett_2014}, among others, have proposed modifications to remedy this issue, it remains a problem in the basic algorithm.

To get around this problem, one idea is to push the singularities in the formula $u = \widetilde S h$ away from $\gamma$ and into $(\overline{\Omega^+})^c =: \Omega^-$.
A discrete version of such a representation is the basic idea of the \emph{method of fundamental solutions} (MFS);
see Cheng, Hong~\cite{ChengHong_2020} for an overview of the MFS.
While this method can be applied directly to~\eqref{eq:neumann_problem}, we first reformulate the second-order boundary value problem as a first-order problem to be compatible with the rest of the paper.
As $u$ is a harmonic function,  $f = (\partial_x  - i\partial_y) u$ is an analytic function whose boundary condition translates to $\Re(\nu f) = g$, where $\nu(w)$ is the complex outward pointing normal vector at $w \in \gamma$.
Thus~\eqref{eq:neumann_problem} is (up to a constant) equivalent to
\begin{equation}
    \label{eq:intro_cplx_neumann}
    \begin{cases}
    (\partial_x + i\partial_y)f = 0 & \text{in }\Omega^+, \\ \Re(\nu f) = g & \text{on }\gamma.
\end{cases}
\end{equation}
For the MFS, the idea is to approximate $f$ by a linear combination of fundamental solutions of the differential operator $\partial_x + i\partial_y$, that is,
\begin{equation}
    \label{eq:intro_mfs_approx}
    f(w) \approx \sum_{j=1}^n c_j \frac{1}{w-q_j} =:f_a(w),
\end{equation}
for some set of \emph{source points} $\{q_j\}_{j=1}^n \subset \Omega^-$.
The source points can be placed anywhere in $\Omega^-$, either adapted to $g$ or not, but it is common to place them uniformly along a curve $\gamma' \subset \Omega^-$ at a fixed, positive distance to $\gamma$.
As the approximation $f_a$ of $f$ is analytic in $\Omega^+$, it remains to choose the coefficients $(c_j)_{j=1}^n$ such that $\Re(\nu f_a) \approx g$.
A simple way to implement this is to enforce $\Re(\nu(w_m)f_a(w_m)) = g(w_m)$ at a number of \emph{collocation points} $\{w_m\}_{m=1}^M \subset \gamma$ via the system of equations
\begin{equation}
    \label{eq:intro_mfs_collocation}
    \sum_{j=1}^n \Re \Bigl(c_j \frac{\nu(w_m)}{w_m-q_j}\Bigr) = \Re(\nu(w_m)g(w_m)), \quad \quad 1 \le m \le M,
\end{equation}
where $M$ usually is selected so that the ratio $n/M$ is fixed.
One then solves for the complex coefficients $(c_j)_{j=1}^n$ in the real system~\eqref{eq:intro_mfs_collocation}. 
Unlike the BIE method, there are no singularities of this representation of $f$ close to the boundary, enabling simple and accurate computation of the solution at the boundary.
As the error function $f - f_a$ is analytic on $\Omega^+$, it follows from the maximum principle that the error in the domain is controlled by the maximum boundary error.
For domains with analytic boundaries and $f$ with an analytic continuation to a neighborhood of $\overline{\Omega^+}$, this boundary error decreases exponentially in the Sobolev $H^s(\gamma)$ norm, as proved 
by Katsurada~\cite[Thm.~4.1]{Katsurada_1990}.
However, if $f$ lacks such an analytic continuation, solving~\eqref{eq:intro_mfs_collocation} becomes problematic, as then the $\ell_2$ norm of the coefficients $(c_j)_{j=1}^n$ rapidly increases as $n \to \infty$; see Barnett and Betcke~\cite[Thm.~7]{Barnett_2008}.
This leads to large errors from floating point arithmetic when solving~\eqref{eq:intro_mfs_collocation}, which prevents errors from approaching machine epsilon.
The theoretical reason for this numerical instability is that the approximation in~\eqref{eq:intro_mfs_approx} is ill-posed in the sense that the matrix implicit in~\eqref{eq:intro_mfs_collocation} is ill-conditioned.
In fact, the condition number grows exponentially as $n \to \infty$, as was shown in, for instance, Kitagawa~\cite[Thm.~1.1(b)]{Kitagawa_1991}.

In this paper, we propose a variant of the MFS that we refer to as the \emph{Whitney method of fundamental solutions} (WMFS), where the source points $\{q_j\}_{j=1}^\infty$ constitute a \emph{Whitney set} in $\Omega^-$.
This means that the set $\{B(q_j, \varepsilon \dist(q_j, \gamma))\}_{j=1}^\infty$ is a uniformly locally finite cover of $\Omega^-_\delta :=  \{z \in \Omega^-_\delta\:; \dist(z, \gamma) < \delta\}$ for some $\delta > 0$, in the sense that every point $z \in \Omega^-_\delta$ intersects at most $N$ balls in the cover, where $\varepsilon \in (0, 1/2]$ and $N \in \Z_+$ are parameters independent of $z$. 
For an introduction to the closely related concept of a \emph{Whitney partition}, together with an application of the notion, we refer the reader to the book by Stein~\cite[Ch.~VI]{Stein_1971}.
In Example~\ref{ex:explicit_whitney_partition}, we give explicit examples of Whitney sets of source points, which, in particular, is depicted in Figure~\ref{fig:Ca}, although only a finite number of the source points are shown in the plot. 

A key property of Whitney sets of source points, and a main result of our paper,  is that using this allocation eliminates the problem of the ill-conditioned system in~\eqref{eq:intro_mfs_collocation}, in the following precise sense: 
given that $\varepsilon > 0$ is sufficiently small, every $f \in L_2^+(\gamma)$ (the complex Hardy subspace of $L_2(\gamma)$, consisting of traces of analytic functions) has an unconditionally $L_2(\gamma)$ convergent series expansion
\begin{equation}
    \label{eq:intro_lusin_whitney_expansion}
    f(w) = \sum_{j=1}^\infty c_j \frac{b_j}{(w-q_j)^{2}},
\end{equation}
with coefficients satisfying $\sum_{j=1}^\infty |c_j|^2 \simeq \|f\|^2$ and constants $b_j$ chosen so that $\psi_j(w)= b_j (w-q_j)^{-2}$ are normalized functions in $L_2(\gamma)$.
Following Lusin (see Meyer~\cite[Ch.~1]{Meyer_1993}), we use functions $(w - q_j)^{-2}$ in~\eqref{eq:intro_lusin_whitney_expansion} instead of $(w - q_j)^{-1}$ as in~\eqref{eq:intro_mfs_approx} to improve the localization of our "basis functions" $\psi_j$.
Thus, strictly speaking, $\psi_j$ are not fundamental solutions, but by abuse of terminology, we shall speak of the WMFS as a method of fundamental solutions.
Now, although the series in~\eqref{eq:intro_lusin_whitney_expansion} is reminiscent of a basis expansion, we show in Proposition~\ref{prop:no_lusin_basis} that $\{\psi_j\}_{j=1}^\infty$ cannot be a Schauder basis for $L_2^+(\gamma)$, regardless of what set of source points is used.
Therefore, the coefficients in~\eqref{eq:intro_lusin_whitney_expansion} are never unique. 
Despite this, the vectors $\{\psi_j\}_{j=1}^\infty$ retain many properties of bases.
For instance, by truncating the series in~\eqref{eq:intro_lusin_whitney_expansion}, it follows that \emph{every} $f \in L_2^+(\gamma)$ may be approximated by linear combinations of $\psi_j$, in such a way so that the coefficients $\{c_j\}_{j=1}^J$ are $\ell_2$ bounded as $J \to \infty$.
This contrasts the case of the classical MFS, where such an approximation with $\ell_2$ bounded coefficients is only possible for $f$ with an analytic continuation to a neighborhood of $\overline{\Omega^+}$ (see Barnett and Betcke~\cite[Thm.~7]{Barnett_2008}), leading to increased errors in floating point arithmetic for functions lacking such analytic continuations.
Thus, the WMFS provides a simple solution to the primary drawback of the MFS---numerical instability.

Somewhat surprisingly given the estimate $\sum_{j=1}^\infty |c_j|^2 \simeq \|f\|^2$ mentioned above, it turns out that truncating the series~\eqref{eq:intro_lusin_whitney_expansion} at the $M$:th term and sampling the Lusin wavelets at $M$ collocation points, and then enforcing the boundary condition as in~\eqref{eq:intro_mfs_collocation}, results in a system with a large condition number of the matrix; 
see Remark~\ref{rmk:CC-method} for a possible remedy to this issue. 
% Although our implementation of the WMFS in Section~\ref{sec:numerical_results} does not use collocation points to generate a matrix equation for the coefficients---rather, we use a variational approach by testing $\Re(\nu f) = g$ against suitable test functions---a similarly ill-conditioned system is generated. 
However, as has been observed by Barnett and Betcke~\cite{Barnett_2008}, the numerical stability of the MFS is governed not by the condition number, but by the coefficient norm $ |c|^2 = \sum_{j=1}^M |c_j|^2$, which must remain $\O(1)$ in $M$.
The minimum achievable error of the MFS is directly related to $|c|$ for $M$ large, since evaluating $f_a$ induces round-off errors of order $|c|\times (\text{machine epsilon})$. 
The condition number only implies an upper bound for $|c|$, which in many instances is a very crude bound. 
In Section~\ref{sec:numerical_results}, we are able in each example to choose source points from the Whitney covering in such a way that $|c|$ stays small and uniformly bounded as the number of source points increases, even when the solution lacks analytical continuation to a neighborhood of $\overline{\Omega^+}$ and when the domain has corners. 
Although we have not included any graphs illustrating this, the boundedness of $|c|$ can be inferred implicitly from the accuracy of the solutions.

The motivation for using a Whitney set as source points and using the derivatives $(w-q_j)^{-(k+1)}$ of the Cauchy kernel comes from wavelet theory.
To explain the connection, we briefly consider the domain $\C^+ = \{x+iy\:; y >0\}$ of the upper-half complex plane, where wavelets of the form $(z-q_j)^{-(k+1)}$, $k \ge 1$, were first studied by Lusin;  see Meyer~\cite[Ch.~1]{Meyer_1993} and references therein.
The idea behind wavelets is to construct a basis for a 
space, typically $L_2(\R^n)$,
such that the base consists of translates and scalings of a single \emph{mother wavelet} $\psi \in L_2(\R^n)$.
In our case, we construct a generalized basis for the Hardy subspace $L_2^+(\R)$ (the complex subspace of traces of analytic functions on $\C^+$), using the Lusin mother wavelet $\psi(x) = 1/(x+i)^{k+1}$, where $k \ge 1$ is necessary to ensure that $\psi$ has sufficient decay. 
Given the mother wavelet and fixed parameters $a>1, b>0$, one defines
\[\psi_{j,l}(x) = a^{-l/2}\psi(a^{-l}x-jb), \quad \quad j, l \in \Z,\]
which for $\psi(x) = 1/(x+i)^{k+1}$ yields \emph{Lusin wavelets}
\[\psi_{j,l}(x) = \frac{a^{lk+l/2}}{(x-q_{j,l})^{k+1}},\]
with source points $q_{j,l}= ja^lb - ia^l$.
Geometrically, $\{q_{j,l}\}_{j, l \in \Z}$ is a Whitney set for the lower-half complex plane $\C^-$ comprised of horizontal layers indexed by $l$, where each layer is a row of uniformly spaced source points, similar to the configuration in Figure~\ref{fig:Ca}.
As $a \to 1$ and $b \to 0$, the "density" of the source points increases.
Given that $(a,b)$ is sufficiently close to $(1,0)$, the collection $(\psi_{j,l})_{j,l \in \Z}$ forms a generalized basis for $L_2^+(\R)$, a \emph{frame}, in the sense that every $f \in L_2^+(\R)$ satisfies 
\begin{equation}
\label{eq:intro_frame_property}
    A\|f\|^2 \le \sum_{j=-\infty}^\infty \sum_{l=-\infty}^\infty |(f,\psi_{j,l})|^2 \le B\|f\|^2
\end{equation}
for some constants $0 < A, B < \infty$.
From~\eqref{eq:intro_frame_property}, it follows that $f$ has a series expansion
\begin{equation}
    \label{eq:intro_wavelet_expansion}
    f = \sum_{j=-\infty}^\infty \sum_{l=-\infty}^\infty c_{j,l}\psi_{j,l}
\end{equation}
with (non-unique) coefficients satisfying $\sum_{j, l \in \Z}|c_{j,l}|^2 \simeq \|f\|^2$ (see Daubechies~\cite[Prop.~3.3.2]{Daubechies_1992}).
One way to construct such a wavelet frame is to discretize its continuous counterpart, given by \emph{Calderón's formula} (see Meyer~\cite[Ch.~1, Eq.~(5.9)]{Meyer_1993})
\begin{equation}
    \label{eq:intro_calderons_formula}
    f(x) = C_{\psi}\int_0^\infty \int_{-\infty}^\infty (f, \psi_{(u,t)})\psi_{(u,t)}(x)\,du\frac{dt}{t},
\end{equation}
where $\psi_{(u,t)}(x) = t^{-1}\psi(t^{-1}(x-u))$.
The "frame coefficients" $(f, \psi_{(u,t)}) = \int_{-\infty}^\infty f(x)\psi_{(u,t)}(x)\,dx$ satisfy the isometry property
\begin{equation}
    \label{eq:intro_calderon_isometry}
    \|f\|^2 = C_\psi\int_0^\infty \int_{-\infty}^\infty |(f, \psi_{(u,t)})|^2\,du\frac{dt}{t},
\end{equation}
where $C_\psi = \int_0^\infty |\widehat \psi(\xi)|^2 \,\frac{d\xi}{\xi} < \infty$ if $k \ge 1$.
When $k= 0$, $C_\psi = \infty$, which is the reason we only consider $k \ge 1$.
The main theoretical result of this paper is Theorem~\ref{thm:lusin_frame}, which gives a generalization of~\eqref{eq:intro_frame_property} to Lipschitz domains with boundary $\gamma$.
The resulting set of Lusin wavelets, with source points placed according to a Whitney partition, then forms a frame for $L_2^+(\gamma)$, which implies~\eqref{eq:intro_lusin_whitney_expansion}.
We prove~\eqref{eq:intro_calderon_isometry} by reducing the statement to known square function estimates for the Cauchy integral operator on Lipschitz curves.
For Lipschitz domains, the equality (up to a multiplicative constant) in~\eqref{eq:intro_calderon_isometry} does not hold, so the exact statement in Theorem~\ref{thm:square_function_estimates} becomes (for $k=1$)
\begin{equation}
    \label{eq:intro_cauchy_sfe}
    \|f\|^2 \simeq \int_{\Omega_1^+}|F'(z)|^2d(z)\,dA(z) + \int_{\Omega_1^-}|F'(z)|^2d(z)\,dA(z) + |F(0)|^2,
\end{equation}
where $F(z) = \frac{1}{2\pi i}\int_\gamma\frac{f(w)}{w-z}\,dw$ is the Cauchy integral of $f$, $d(z) = \dist(z, \gamma)$, $\Omega_1^\pm = \{z \in \Omega^\pm\:;d(z)<1\}$, and $dA$ is the area measure on $\C$.
To prove~\eqref{eq:intro_frame_property}, we discretize~\eqref{eq:intro_cauchy_sfe} with an argument involving the Poincaré-Wirtinger inequality.
The discretization procedure leads to Theorem~\ref{thm:lusin_frame}, which states that
\[\|f\|^2 \simeq \sum_{j=1}^\infty |(f, \psi_j^+)|^2 + \sum_{j=1}^\infty |(f, \psi_j^-)|^2 + |(f, w^{-1})|^2,\]
where $\psi_j^\pm(w) = b_j^\pm(w-q_j^\mp)^{-2}$ are normalized Lusin wavelets with corresponding source points $\{q_j^\mp\}_{j=1}^\infty$ constituting Whitney sets in $\Omega_1^\mp$, respectively.
It follows that $\{\psi_j^+\}_{j=1}^\infty$ is a frame for $L_2^+(\gamma)$.

The paper is organized as follows. 
In Section~\ref{sec:sfe}, we survey the square function estimates needed for constructing a Lusin wavelet frame of $L_2^+(\gamma)$ and discuss the Hardy splitting of $L_2(\gamma)$ on Lipschitz curves $\gamma$. 
As this section is only meant as an outline, all the proofs of Section~\ref{sec:sfe} are relegated to Appendices~\ref{appendix:hardy_splitting} and~\ref{appendix:sfe}, where we derive the theorems from known results in the harmonic analysis literature.
In Section~\ref{sec:discrete_frames}, we discretize the square function estimates of Section~\ref{sec:sfe} to construct a discrete Lusin frame for $L_2(\gamma)$.
In Section~\ref{sec:bvp_theory}, the $L_2(\gamma)$-frame is readily converted to a frame for $L_2^+(\gamma)$.
We apply this frame to solve boundary value problems such as~\eqref{eq:intro_cplx_neumann} and show that this gives rise to a method of fundamental solutions as described above. 
Finally, we implement and give numerical results for the method in Section~\ref{sec:numerical_results}.
On the one hand, we find in particular that the WMFS gives high accuracy in the computed solution near the boundary, unlike the BIE method; see Figure~\ref{fig:Aa}.
On the other hand, we find that the WMFS achieves high accuracy also when the solution cannot be analytically continued across $\gamma$, unlike the classical MFS; see Figure~\ref{fig:square_conewmfs_vs_mfs}.

We have limited the numerical examples in Section~\ref{sec:numerical_results} to interior Neumann and Dirichlet problems.
However, exterior Neumann and Dirichlet problems can similarly be solved with the WMFS using exterior Lusin wavelets $\psi_j^-$, and more general transmission problems can be solved using both $\psi_j^+$ and $\psi_j^-$.
% Moreover, besides the Neumann problem, the WMFS can be used to solve other boundary value problems like the Dirichlet problem.
Beyond analytic functions in $\C$, we expect the WMFS to extend to both Helmholtz equations and higher dimensions. 
Further publications on this are planned.

\section{Square function estimates}
\label{sec:sfe}
In this section, we discuss square function estimates for the $L_2$ space of functions on a closed Lipschitz curve $\gamma$. 
Before stating the main theorem, we fix some notation. 
First, we let $\gamma = \{r_\gamma(\theta)e^{i\theta}\:;0 \le \theta \le 2\pi\}$ be a positively oriented curve, where $r_\gamma : [0,2\pi] \to (0, \infty)$ is an arbitrary periodic Lipschitz function. 
The bounded interior open domain enclosed by $\gamma$ is denoted $\Omega^+$, while $\Omega^- = (\overline{\Omega^+})^c$ is the exterior open domain; see Figure~\ref{fig:lipschitz_domain_notation}.
Second, we let $k \ge 1$ be an integer and fix $\delta \in (0,\infty]$.
Then, with $d(z) = \dist(z,\gamma)$, we define $\Omega^\pm_\delta = \{z \in \Omega^\pm \:;d(z) < \delta\}$. Occasionally, we also write $d(x,A) = \text{dist}(x,A)$ for an arbitrary set $A\subset \R^n$.
The area measure on $\C$ is denoted by $dA$ while $dw$ and $|dw|$ are the complex and arc-length measures on $\gamma$, respectively.
Finally, $(f, g) = \int_\gamma f(w)\overline{g(w)}\,|dw|$ is the canonical inner-product on $L_2(\gamma)$ with corresponding norm $\|\cdot\|=\|\cdot\|_2$. 
For a function $f \in L_2(\gamma)$, we write  $F(z) = \frac{1}{2\pi i}\int_\gamma \frac{f(w)}{w-z}\,dw$ for the Cauchy integral of $f$, where $z \in \C \setminus \gamma$.

\begin{theorem}[Square function estimates]
    \label{thm:square_function_estimates}
     Fix $\delta \in (0,\infty]$ and an integer $k \geq 1$. 
     Then, for $f \in L_2(\gamma)$, we have the estimate
\[\|f\|^2 \simeq \int_{\Omega^+_\delta \cup \Omega^-_\delta}|F^{(k)}(z)|^2\,d(z)^{2k-1}dA(z)  + \sum_{j=0}^{k-1}|F^{(j)}(0)|^2. \]
%where the constants implicit in the estimate do not depend on $f$.
\end{theorem}
In connection to the proof of Theorem~\ref{thm:square_function_estimates}, we find the classical result that $L_2(\gamma)$ has a topological splitting in terms of subspaces consisting of boundary values of holomorphic functions on $\Omega^+$ and $\Omega^-$. 
\begin{theorem}
    \label{thm:hardy_splitting}
    For $f \in L_2(\gamma)$ and $t \neq 0$, define $F_t(w) = F(e^{-t}w)$, $w \in \gamma$. 
    Then we have the following convergences in the $L_2(\gamma)$ norm:
    \begin{align*}
        F_t &\to f^+, \quad \quad t \to 0^+, \\
        -F_t &\to  f^-, \quad \quad t \to 0^-,
    \end{align*}
    and further that $P^\pm : L_2(\gamma) \to L_2(\gamma)$ defined by $P^\pm f = f^\pm$ are bounded, complementary (oblique) projections.
\end{theorem}
As a consequence of Theorem~\ref{thm:hardy_splitting}, we get that $L_2(\gamma)$ decomposes into so-called Hardy subspaces.
We write $\ran (T)$ and  $\ker (T)$ to denote the range and nullspace of an operator $T$, respectively.
\begin{definition}[Hardy subspaces]\label{def:hardy_subspaces}
    The Hardy subspaces of $L_2(\gamma)$ for $\Omega^\pm$ are $L_2^\pm(\gamma) = R(P^\pm)$, where $P^\pm$ are the projections from Theorem~\ref{thm:hardy_splitting}. 
\end{definition}
So, by Theorem~\ref{thm:hardy_splitting}, we have a topological splitting
\begin{equation}
    \label{eq:hardy_splitting}
    L_2(\gamma) = L_2^+(\gamma) \oplus L_2^-(\gamma).
\end{equation}

\begin{figure}
     \centering
         \includegraphics[width=\textwidth]{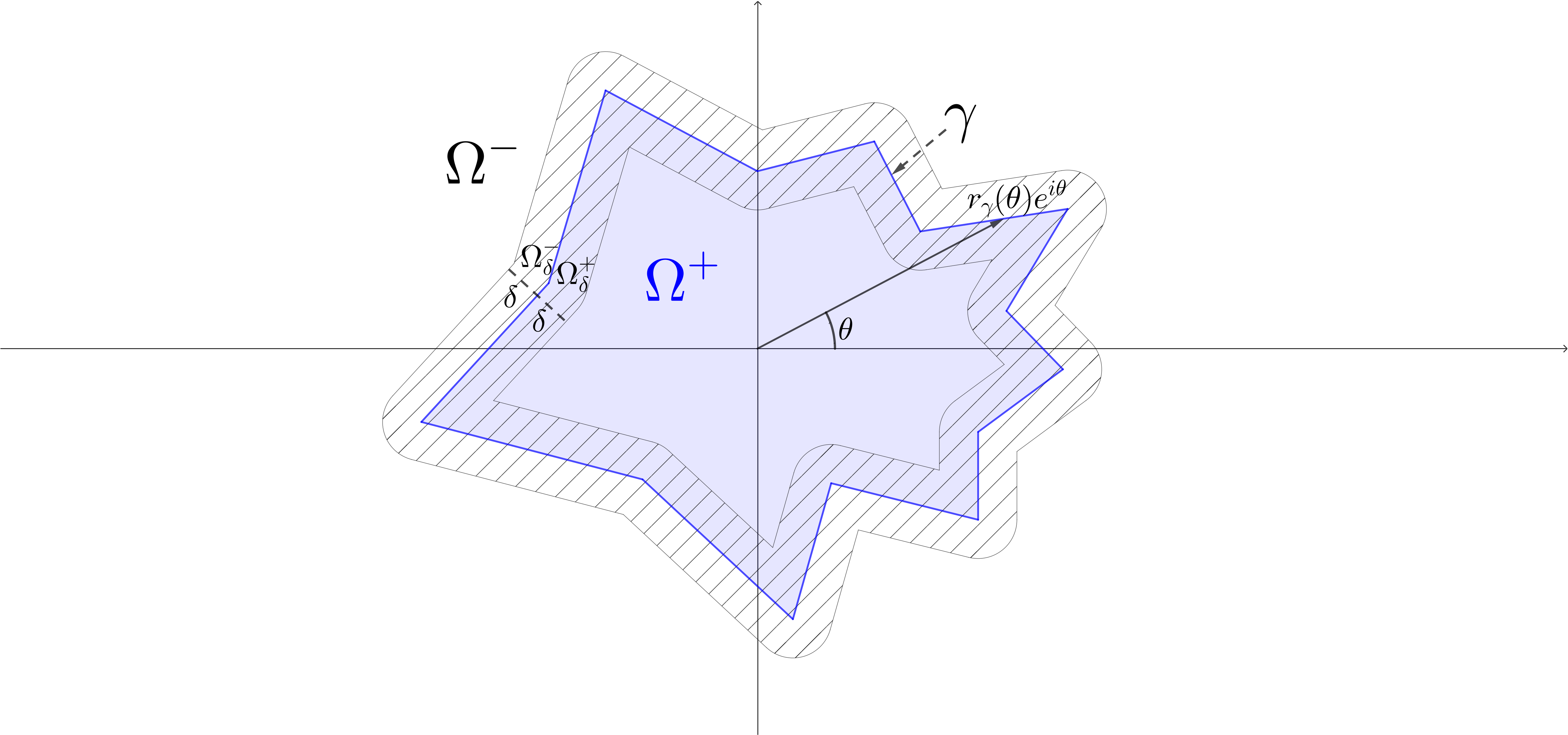}
         \caption{The Lipschitz domain $\Omega^+$.}
         \label{fig:lipschitz_domain_notation}
\end{figure}
The idea of the proof of Theorems~\ref{thm:square_function_estimates} and~\ref{thm:hardy_splitting} is to map the problem to the unit disk and to reduce it to known square function estimates and functional calculus of an infinitesimal generator $BD$ for the Cauchy-Riemann system.
For a function $f$ holomorphic on $\Omega^+$ or $\Omega^-$, we have that $f$ satisfies Cauchy-Riemann's equations
\[(\partial_x + i \partial_y)f(x,y) = 0.\]
Changing variables from $z = x+iy$ to $z = r_\gamma(\theta)e^{-t}e^{i \theta}$, we get by the chain rule that
\begin{align*}
    \partial_t &= -r_\gamma(\theta)e^{-t}(\cos \theta \,\partial_x + \sin \theta \,\partial_y), \\
    \partial_\theta &= e^{-t}((r_\gamma(\theta)\cos \theta)' \,\partial_x + (r_\gamma(\theta)\sin \theta)' \,\partial_y).
\end{align*}
Thus,  solving for $(\partial_x, \partial_y)$, we obtain that
\[\begin{pmatrix}
    \partial_x \\ \partial_y
\end{pmatrix} = A(t, \theta) \begin{pmatrix}
    (r_\gamma(\theta)\sin \theta)' & r_\gamma(\theta) \sin \theta \\ -(r_\gamma(\theta) \cos \theta)' & -r_\gamma(\theta)\cos \theta 
\end{pmatrix} \begin{pmatrix}
    \partial_t \\ \partial_\theta
\end{pmatrix}\]
for some non-zero coefficient $A(t, \theta)$. 
Therefore, in $(t, \theta)$-coordinates, Cauchy-Riemann's equations reads
\begin{equation}\label{eq:cauchy_riemann_t_theta}\biggl(\partial_t + \frac{r_\gamma(\theta)e^{i\theta}}{(r_\gamma(\theta)e^{i\theta})'}\partial_\theta \biggr)f(t,\theta) = 0.\end{equation}
Setting $f_t(\theta) = f(t,\theta)$, we get that $f_t \in L_2(\T)$, where $\T$ is the unit circle. 
Define the differential operator $D : f \mapsto -i \partial_\theta f$ and the multiplication operator $B : f \mapsto bf$, where $b(\theta) = i\frac{r_\gamma(\theta)e^{i\theta}}{(r_\gamma(\theta)e^{i\theta})'}$ is the coefficient from~\eqref{eq:cauchy_riemann_t_theta} times $i$. 
Then~\eqref{eq:cauchy_riemann_t_theta} is equivalent to
\begin{equation}\label{eq:t_theta_cauchy_riemann}
    (\partial_t + BD)f_t = 0.
\end{equation}
Let us discuss a few facts about the operator $BD$. 
\begin{proposition}
    \label{prop:bd_facts}
   The operator $BD : L_2(\T) \to L_2(\T)$ with the Sobolev space $H^1(\T)$ as domain is a densely defined, closed, unbounded operator,  and is \emph{bisectorial} in the sense that 
    \[\sigma(BD) \subset S_\omega := \{\lambda \in \C\:; |\Arg \lambda | \le \omega \text{ or }|\Arg (-\lambda)| \le \omega\}\cup \{0\}\]
    for $\omega = \sup_\theta |\Arg b(\theta)| \in  [0, \pi/2)$. 
    Further, $BD$ satisfies the resolvent estimates
    \begin{equation}
        \label{eq:resolvent_estimates}
        \bigl\|(\lambda - BD)^{-1} \bigr\| \lesssim \frac{1}{\dist(\lambda, S_\omega)}.
    \end{equation}
    The nullspace $\ker(BD)$ consists of the constant functions, and we have the topological direct sum
     \begin{equation}
        \label{eq:bd_range_kernel_splitting}
        L_2(\T) = \ker (BD) \oplus \ran (BD),
    \end{equation}
    where $\ran (BD)$ is closed.
\end{proposition}
These are well-known results for self-adjoint operators $D$ and accretive operators $B$; see for instance Auscher, Axelsson, and McIntosh~\cite[Prop.~3.3]{AAMc_2008}.
To get~\eqref{eq:resolvent_estimates}, consider $(\lambda-BD)^{-1}f=u$.
Then write $(B^{-1}(\lambda - BD)u,u) = (B^{-1}f,u)$ and evaluate the imaginary part using that $\inf_\theta \Re b(\theta) > 0$.
To get~\eqref{eq:bd_range_kernel_splitting}, note that $L_2(\T) = \ker (D) \oplus \ran (D)$ since $D$ is self-adjoint.
Now,~\eqref{eq:bd_range_kernel_splitting} follows since $N(BD) = N(D)$ and the angle between $\ran (D)$ and $\ran (BD) = B \ran (D)$ is less than $\pi/2$.

When $\Omega^+$ is the unit disk, corresponding to the case where $B=I$, then $BD = D = -i \partial_\theta$ has spectrum $\sigma(D) = \Z$. 
Using Fourier series and Parseval's identity, we have an isometry $\mathcal{F} : L_2(\T) \to \ell_2(\Z)$ given by the Fourier transform $\mathcal{F}$ which diagonalizes $D$, so that $\mathcal{F}D\mathcal{F}^{-1}$ is the multiplication operator $(\widehat f(k))_{k \in \Z} \mapsto (k \widehat f(k))_{k \in \Z}$ on $\ell_2(\Z)$.
Here, the Cauchy integral $f \mapsto F_t^\pm$ is the Fourier multiplier
\begin{equation}
    \label{eq:fourier_multiplier}
    \widehat f(k) \mapsto e^{-tk} \chi^\pm(k) \widehat f(k),
\end{equation}
where $\chi^+(k) = \chi_{\{k \ge 0\}}(k)$ and $\chi^-(k) = \chi_{\{k < 0\}}(k)$.
Theorems~\ref{thm:square_function_estimates} and~\ref{thm:hardy_splitting} can be proved using Fourier series in this case, as done in Jonsson and Timlin~\cite[Thm.~3.1.1, Thm.~2.2.5]{JonssonTimlin_2024}.

When $B \neq I$, it is well known in harmonic analysis that the generalization of Parseval's identity to non-self-adjoint operators $BD$ is a Littlewood-Paley type splitting into spectral subspaces associated with the dyadic parts
\[\{\lambda \in S_\omega\:; 2^j \le | \Re\lambda| < 2^{j+1}\}, \quad \quad j \in \Z,\]
of the bisector $S_\omega$.
The precise statement of such a Littlewood-Paley splitting is the square function estimates in Theorem~\ref{thm:square_function_estimates}.
From this one can construct bounded operators $e^{-tBD}\chi^\pm(BD)$ by Dunford functional calculus, which generalizes~\eqref{eq:fourier_multiplier}.
Using known estimates for the operator $BD$, we in this way prove Theorems~\ref{thm:square_function_estimates} and~\ref{thm:hardy_splitting} in Appendices~\ref{appendix:sfe} and~\ref{appendix:hardy_splitting}, respectively.

\section{Discrete Lusin frames}\label{sec:discrete_frames}
The goal of this section is to construct certain generalized bases for $L_2(\gamma)$ referred to as \emph{frames}.
% For this section, we will, unless otherwise stated, use the notation as described in the beginning of Section~\ref{sec:sfe}.
\begin{definition}\label{def:frame}
    A set $\{f_j\}_{j=1}^\infty \subset H$ of vectors in some Hilbert space $H$ is a \emph{frame} if there exist constants $0 < A \le B < \infty$ so that
    \begin{equation}
        \label{eq:frame}
        A\|f\|_H^2 \le \sum_{j=1}^\infty|(f, f_j)_H|^2 \le B\|f\|_H^2, \quad \quad f \in H.
    \end{equation}
\end{definition}
Frames can be regarded as generalized bases that allow for linearly dependent basis vectors.
Despite the linear dependence, frames retain many of the useful properties of bases.
For instance, given a frame $\{f_j\}_{j=1}^\infty$, each $f \in H$ has an unconditionally convergent expansion
\begin{equation*}
    f = \sum_{j=1}^\infty c_j f_j.
\end{equation*}
In general, however, the coefficients are not unique. For a comprehensive introduction to frames, we refer the reader to the book by Christensen~\cite{Christensen_2016}.
In Sections~\ref{sec:bvp_theory} and~\ref{sec:numerical_results}, we use the frame constructed below in Theorem~\ref{thm:lusin_frame} to solve boundary value problems on the domain $\Omega^+$. 
These frame vectors, which we refer to as \emph{Lusin wavelets}, are normalized rational functions $w \mapsto b_j(w-q_j)^{-(k+1)}$, where $k\ge 1$ is an integer and $\{q_j\}_{j=1}^\infty \subset \C \setminus \gamma$ is a so-called \textit{Whitney set}.
\begin{definition}    
\label{def:whitney_set}
Let $A \subset \R^n$ be a closed set, $\O \subset A^c$, and $\varepsilon \leq 1/2$. We say that $\{x_j\}_{j=1}^\infty \subset A^c$ is an \textit{$\varepsilon$-fine Whitney set for $\O$ with respect to $A$}, if the following two properties hold:
\begin{enumerate}[font = \rmfamily, label = (\roman*)]
    \item $ \O \subset \bigcup_{j = 1}^\infty B_j$, where $B_j = B(x_j,\varepsilon d(x_j,A))$, \label{def:whitney_set_cond_1}
    \item $N := \sup_{x \in \R^n}\sum_{j = 1}^\infty \chi_{B_j}(x) < \infty$. \label{item:whitney_set_cond_2}
\end{enumerate}
We call $N$ the \textit{covering constant for $\{x_j\}_{j=1}^\infty$}. If we do not wish to emphasize the parameter $\varepsilon$, we omit the word "$\varepsilon$-fine".
\end{definition}
\begin{remark}
    As we show in Appendix~\ref{appendix:whitney_partition}, for any $A$ and $\O$ as in Definition~\ref{def:whitney_set}, there exists an $\varepsilon$-fine Whitney set for $\O$ with respect to $A$ for any $\varepsilon \in (0,1/2]$. Moreover, such a Whitney set can always be chosen so that the covering constant is controlled by a constant $C$ depending only on the dimension $n$. When $\O$ is bounded, this is an immediate consequence of Besicovitch's covering theorem; see Evans and Gariepy~\cite[Thm.~1.27]{EvansGariepy_2015}.
\end{remark}
In Section~\ref{sec:numerical_results}, we need explicit Whitney sets for $\Omega^\pm_\delta$ with respect to $\gamma$.
\begin{example}
    \label{ex:explicit_whitney_partition}
    Define $\rho : \C \to \C$ by $\rho(re^{i\theta}) = rr_\gamma(\theta)e^{i\theta}$, and let $\mathcal{L} = \max(\textup{Lip}(\rho),\textup{Lip}(\rho^{-1}))$. Fix $\varepsilon > 0$ such that $\varepsilon' := 3\mathcal{L}^2\varepsilon/2 \leq 1/4$, and let $n_l = \big\lceil 2\pi \varepsilon^{-1} (1+(1+\varepsilon)^l) \big\rceil$ for $l \geq 0$. For $l \geq 0$ and $1 \leq j \leq n_l$, let 
    \begin{equation*}
        r_l = 1 + \frac{1}{(1+\varepsilon)^{l}}, \quad 
        \theta_{j,l} = -\pi +  \frac{2\pi}{n_l}j,\quad \text{and} \quad q_{j,l} = r_l r_\gamma(\theta_{j,l})e^{i\theta_{j,l}}.
    \end{equation*}
    Then $\cup_{l = 0}^\infty \{q_{j,l}\}_{j=1}^{n_l}$ is an $\varepsilon'$-fine Whitney set for $\Omega^-_{1/\mathcal{L}}$ with respect to $\gamma$. This claim is verified in Appendix~\ref{appendix:whitney_partition}. Given any $\delta > 0$, one can remove a finite number $N_0 -1$ of layers $l$, so that $l \geq N_0$, to obtain $\varepsilon'$-fine Whitney sets for $\Omega_\delta^-$ whose source points lie at a distance to $\gamma$ that is $\lesssim$ $\delta$. Analogous statements hold for $\Omega^+_\delta$ if one instead defines $r_l = 1 - (1+\varepsilon)^{-l}$, $l \geq 1$. See Figures~\ref{fig:Aa},~\ref{fig:Ca}, and~\ref{fig:Cb} for examples of Whitney sets of this form.
\end{example}
Using Theorem~\ref{thm:square_function_estimates}, we show that when the source points $q_j^\pm$ for the Lusin wavelets are chosen so as to form a sufficiently fine Whitney set for $\Omega_\delta^\pm$ with respect to $\gamma$, the collection $(\psi_j^+)_{j=1}^\infty \cup (\psi_j^-)_{j=1}^\infty\cup (w^{-j})_{j=1}^k$ is a frame for $L_2(\gamma)$. 
\begin{theorem}[Lusin frame]
    \label{thm:lusin_frame}
    Let $k \geq 1$ and let $\delta \in (0,\infty]$. Consider $\varepsilon$-fine Whitney sets $\{q_j^\pm\}_{j = 1}^\infty \subset \Omega^\pm$ for $\Omega^\pm_\delta$ with respect to $\gamma$. Let $\psi_j^\pm (w)= b^\pm_j(w - q_j^\mp)^{-(k+1)}$ be the corresponding normalized Lusin wavelets. There exists a universal constant $\varepsilon_0>0$ such that if $\varepsilon \leq \varepsilon_0$, then  
    \begin{align*}
        \|f\|^2 \simeq \sum_{j=1}^\infty |(f, \psi_j^+)|^2 + \sum_{j=1}^\infty |(f, \psi_j^-)|^2 + \sum_{j=1}^{k}|(f, w^{-j})|^2,
    \end{align*}
    for all $f \in L_2(\gamma)$. 
\end{theorem}
The proof of Theorem~\ref{thm:lusin_frame} will be carried out in two stages. After proving Lemma~\ref{lem:k+1_norm_less_than_k_norm} below, we give a proof of the special case when $\delta = \infty$. We then introduce two auxiliary results related to truncating a collection of Lusin wavelets, enabling us to reduce the case $\delta < \infty$ to $\delta = \infty$. 
\begin{lemma}\label{lem:k+1_norm_less_than_k_norm}
    Let $\O \subset \C$ be any open set, let $k \geq 1$, and let $G$ be analytic in $\O$. There exists a universal constant $C$ such that 
    \begin{align*}
        \int_{\O}|G^{(k+1)}(z)|^2d(z)^{2k+1}\,dA(z) \leq C \int_{\O}|G^{(k)}(z)|^2d(z)^{2k-1}\,dA(z),
    \end{align*}
    where $d(z) = \textup{dist}(z,\partial\O)$.
\end{lemma}
\begin{proof}
We apply Lemma~\ref{lem:existence_of_whitney_sets} with $\varepsilon = 1/8$ to obtain a $1/8$-fine Whitney set $\{z_j\}_{j = 1}^\infty \subset \O$ for $\O$ with respect $\partial \O$, such that $\sum_{j = 1}^\infty\chi_{B_j'}(z) \leq 4800$ for all $z \in \C$, where $B_j' = B(z_j,\frac{1}{4}d(z_j))$. Let also $B_j = B(z_j,\frac{1}{8}d(z_j))$. Note that whenever $z \in B_j'$, we have
\begin{equation*}
    \frac{3}{4}d(z_j) < d(z_j) - |z-z_j| \leq d(z) \leq d(z_j) + |z-z_j| < \frac{5}{4}d(z_j).
\end{equation*}
Using Caccioppoli's inequality for analytic functions applied to $G^{(k)}$ and the balls $B_j$, $B_j'$ (see Giaquinta and Martinazzi~\cite[Thm~4.1]{Giaquinta_2012}), we estimate
\begin{multline*}
    \int_{\O}|G^{(k+1)}(z)|^2d(z)^{2k+1}\,dA(z) 
    \leq \sum_{j = 1}^\infty \int_{B_j}|G^{(k+1)}(z)|^2d(z)^{2k+1}\,dA(z) \\
     \simeq \sum_{j = 1}^\infty d(z_j)^{2k+1}\int_{B_j}|G^{(k+1)}(z)|^2\,dA(z) \\
    \lesssim \sum_{j = 1}^\infty d(z_j)^{2k+1} \bigg(\frac{1}{4}d(z_j) - \frac{1}{8}d(z_j)\bigg)^{-2} \int_{B'_j}|G^{(k)}(z)|^2\,dA(z) \\
     \simeq \sum_{j = 1}^\infty \int_{B'_j}|G^{(k)}(z)|^2d(z)^{2k-1}\,dA(z) 
     \lesssim \int_{\O}|G^{(k)}(z)|^2d(z)^{2k-1}\,dA(z). 
\end{multline*}
The constant implicit in each estimate is universal.
\end{proof}
\begin{proof}[Proof of Theorem~\ref{thm:lusin_frame}, case $\delta = \infty$]
    The idea of the proof is to discretize the integrals in Theorem~\ref{thm:square_function_estimates}, so that for all $f \in L_2(\gamma)$, we obtain
    \begin{align}
    \sum_{j=1}^\infty |b_j^\mp F^{(k)}(q^\pm_j)|^2 \simeq \int_{\Omega^\pm}|F^{(k)}(z)|^2\,d(z)^{2k-1}dA(z),\label{eq:discretization_estimate}
    \end{align}
    where $F$ is the Cauchy integral of $f$, and $b_j^\mp$ are the normalizing constants for the (unnormalized) Lusin wavelets $(w-q^\pm_j)^{-(k+1)}$. Using the fact that
    \[b_j^\mp F^{(k)}(q^\pm_j) = b_j^\mp\frac{k!}{2\pi i}\int_\gamma f(w)(w-q_j^\pm)^{-(k+1)}\,dw = \frac{k!}{2\pi i}(f, \overline{i \nu \psi_j^\mp}),\]
    where $\nu$ is the unit normal vector of $\gamma$ pointing into $\Omega^-$, we then combine Theorem~\ref{thm:square_function_estimates} and~\eqref{eq:discretization_estimate} to conclude that for every $f \in L_2(\gamma)$,
    \begin{align*}
        \|f\|^2 &\simeq \sum_{j=1}^\infty |(f, \overline{i \nu \psi_j^+})|^2 + \sum_{j=1}^\infty |(f, \overline{i \nu \psi_j^-})|^2 + \sum_{j=1}^{k}|(f, \overline{i \nu w^{-j}})|^2\\
        &= \sum_{j=1}^\infty |(\overline{i \nu f},  \psi_j^+)|^2 + \sum_{j=1}^\infty |(\overline{i \nu f},  \psi_j^+)|^2 + \sum_{j=1}^{k}|(\overline{i \nu f},  w^{-j})|^2.
    \end{align*}
    Replacing $f$ by $\overline{i \nu f}$ then implies that
    \[\|f\|^2 \simeq \sum_{j=1}^\infty |(f, \psi_j^+)|^2 + \sum_{j=1}^\infty |(f, \psi_j^-)|^2 + \sum_{j = 1}^{k}|(f, w^{-j})|^2, \quad \quad f \in L_2(\gamma),\]
    since $\|\overline{i \nu f}\| = \|f\|$. The proof thus reduces to showing~\eqref{eq:discretization_estimate}. We prove this estimate for the case of $\sum_{j=1}^\infty |b_j^- F^{(k)}(q^+_j)|^2$; the other case, where $+$ is replaced by $-$, and vice versa, is analogous.
    
    We start by showing the inequality $\lesssim$ of~\eqref{eq:discretization_estimate}, which is easier. Let $d(z)= \dist(z,\gamma)$ and set $B_j  = B(q^+_j, \varepsilon d(q^+_j))$. 
    Note that
    \begin{align*}\label{eq:comparable_coefficients}
        |b_j^-|^2= \bigg(\int_\gamma |w - q_j^+|^{-(2k+2)}\,|dw|\bigg)^{-1} \simeq d(q_j^+)^{2k+1} \simeq \frac{1}{\varepsilon^2}d(q_j^+)^{2k-1}m(B_j),
    \end{align*}
    where $m$ denotes Lebesgue measure. Using this observation, $d(z) \simeq d(q^+_j)$ for $z \in B_j$, the mean value properties for holomorphic functions, and Condition~\ref{item:whitney_set_cond_2} of Definition~\ref{def:whitney_set}, we calculate
    \begin{equation}\label{eq:calculation_for_upper_frame_bound}
    \begin{aligned}
        \sum_{j = 1}^\infty |b_j^-F^{(k)}(q_j^+)|^2 & \simeq \frac{1}{\varepsilon^2} \sum_{j = 1}^\infty d(q_j^+)^{2k-1}m(B_j)|F^{(k)}(q_j^+)|^2 \\
        & = \frac{1}{\varepsilon^2} \sum_{j = 1}^\infty d(q_j^+)^{2k - 1}m(B_j)^{-1}\bigg|\int_{B_j}F^{(k)}(z)\,dA(z)\bigg|^2 \\
        & \leq \frac{1}{\varepsilon^2} \sum_{j = 1}^\infty d(q_j^+)^{2k - 1}\int_{B_j}|F^{(k)}(z)|^2\,dA(z) \\
        & \simeq \frac{1}{\varepsilon^2} \int_{\Omega^+}\sum_{j = 1}^\infty\chi_{B_j}(z)|F^{(k)}(z)|^2d(z)^{2k - 1}\,dA(z) \\
        & \lesssim \frac{1}{\varepsilon^2} \int_{\Omega^+}|F^{(k)}(z)|^2d(z)^{2k - 1}\,dA(z).
    \end{aligned}
    \end{equation}
    This proves the inequality $\lesssim$ of~\eqref{eq:discretization_estimate}. Note that this estimate holds for \textit{any} fixed $\varepsilon$.

    We now consider the other direction $\gtrsim$ of~\eqref{eq:discretization_estimate}. 
    Observe that it is enough to show that this estimate holds with $\sum_{j=1}^\infty |b_j^-F^{(k)}(q^+_j)|^2$ replaced by $\sum_{j \in \mathcal{I}} |b_{j}^-F^{(k)}(q^+_{j})|^2$, for any $\mathcal{I} \subset \N$. We now assume that $\varepsilon_0$, which is to be chosen sufficiently small, is less than $1/20$. If $\varepsilon \leq \varepsilon_0$, then according to Lemma~\ref{lem:extracting_sub_whitney_sets}, there exists $\mathcal{I} \subset \N$ such that $\{q_j^+\}_{j \in \mathcal{I}}$ is a $5\varepsilon$-fine Whitney set for $\Omega^+$ with respect to $\gamma$, with covering constant $N \leq 1200$.  
    The point of considering $\{q_j^+\}_{j \in \mathcal{I}}$ instead of $\{q_j^+\}_{j = 1}^\infty$, is that we have the absolute bound 1200 on the covering constant $N$; this ensures that we can choose $\varepsilon_0$ independently of the covering constant for $\{q_j^+\}_{j =1}^\infty$.
    
    Define the norms $\| \cdot \|_\Omega^2 = \int_\Omega |\cdot|^2\,dA(z)$ and $\| \cdot\|_{\Omega,k}^2 = \int_\Omega |\cdot|^2 \,d(z)^{2k-1}dA(z)$.   
    Let $B_j' = B(q_j^+,5\varepsilon)$ and let $a_j = \|1\|_{B'_j,k}$. By the conjugate rule and Cauchy-Schwarz, 
    \begin{equation}\label{eq:Lusin_frame_3.5}
    \begin{aligned}
        & \bigg|\sum_{j \in \mathcal{I}}\|F^{(k)}\|^2_{B_j',k} - \sum_{j \in \mathcal{I}}|a_j F^{(k)}(q^+_j)|^2 \bigg| \\
        & \leq \bigg(\sum_{j \in \mathcal{I}}\big|\|F^{(k)}\|_{B_j',k} - |a_j F^{(k)}(q^+_j)|\big|^2 \bigg)^{1/2}\bigg(\sum_{j \in \mathcal{I}}\big|\|F^{(k)}\|_{B_j',k} + |a_j F^{(k)}(q^+_j)|\big|^2 \bigg)^{1/2} \\
        & = \text{I} \cdot \text{II}.
    \end{aligned}
    \end{equation}
    The triangle inequality together with a similar calculation as in~\eqref{eq:calculation_for_upper_frame_bound}, with $b_j^-$ replaced by $a_j$, shows that
    \begin{align}\label{eq:estimate_of_II}
        \text{II} \lesssim \bigg(\sum_{j \in \mathcal{I}}\|F^{(k)}\|^2_{B'_j,k} \bigg)^{1/2},
    \end{align}
    where the constant implicit in the estimate is universal. To estimate I, we use the mean value property of holomorphic functions in conjunction with the Poincaré-Wirtinger inequality for balls to get
    \begin{equation}\label{eq:poincare_wirtinger}
    \begin{aligned}
        \big|\|F^{(k)}\|_{B'_j,k} - |a_j F^{(k)}(q^+_j)|\big| & = \big|\|F^{(k)}\|_{B_j',k} - \|F^{(k)}(q^+_j)\|_{B_j',k} \big| \\
        & \leq \|F^{(k)} - F^{(k)}(q^+_j)\|_{B_j',k} \\
        & \simeq d(q_j^+)^{k-1/2}\|F^{(k)} - F^{(k)}(q^+_j)\|_{B_j'} \\
        & \lesssim d(q_j^+)^{k+1/2} \varepsilon \|F^{(k+1)}\|_{B_j'} \\
        & \simeq \varepsilon \|F^{(k+1)}\|_{B_j',k+1}.
    \end{aligned}    
    \end{equation}
    By Lemma~\ref{lem:k+1_norm_less_than_k_norm}, and Conditions~\ref{def:whitney_set_cond_1} and~\ref{item:whitney_set_cond_2} of Definition~\ref{def:whitney_set}, we have
    \begin{align}\label{eq:Lusin_frame_3}
        \sum_{j \in \mathcal{I}}\|F^{(k+1)}\|^2_{B'_j,k+1} \leq 400\|F^{(k+1)}\|^2_{\Omega^+,k+1} \lesssim \|F^{(k)}\|^2_{\Omega^+,k} \leq \sum_{j \in \mathcal{I}}\|F^{(k)}\|^2_{B'_j,k},
    \end{align}
    where the constants implicit in the estimates are universal. Combining~\eqref{eq:poincare_wirtinger} and~\eqref{eq:Lusin_frame_3} then shows that
    \begin{align}\label{eq:estimate_of_I}
        \text{I} \lesssim \varepsilon\bigg(\sum_{j \in \mathcal{I}}\|F^{(k+1)}\|^2_{B'_j,k+1}\bigg)^{1/2} \lesssim \varepsilon \bigg(\sum_{j \in \mathcal{I}}\|F^{(k)}\|^2_{B'_j,k}\bigg)^{1/2}.
    \end{align}
    Therefore, \eqref{eq:Lusin_frame_3.5} together with \eqref{eq:estimate_of_II} and \eqref{eq:estimate_of_I} gives
    \begin{align*}
        \bigg|\sum_{j \in \mathcal{I}}\|F^{(k)}\|^2_{B_j',k} - \sum_{j \in \mathcal{I}}|a_j F^{(k)}(q^+_j)|^2 \bigg| \leq C\varepsilon \sum_{j \in \mathcal{I}}\|F^{(k)}\|^2_{B_j',k}
    \end{align*}
    for some universal constant $C$. By choosing $\varepsilon_0 = \frac{1}{20}\min(1/C, 1)$, and using $|a_j| \simeq \varepsilon|b_j^-|$, it follows that whenever $\varepsilon \leq \varepsilon_0$,
    \begin{multline*}
         (1-C\varepsilon)\|F^{(k)}\|^2_{\Omega^+,k} \leq (1-C\varepsilon)\sum_{j \in \mathcal{I}}\|F^{(k)}\|^2_{B_j',k} \\ 
         \leq \sum_{j \in \mathcal{I}}|a_j F^{(k)}(q^+_j)|^2 \lesssim \varepsilon\sum_{j = 1}^\infty|b_j^- F^{(k)}(q^+_j)|^2,
    \end{multline*}
    which is precisely the inequality $\gtrsim$ of~\eqref{eq:discretization_estimate}. We conclude that $\{\psi_j^+\}_{j = 1}^\infty \cup \{\psi_j^-\}_{j = 1}^\infty \cup \{w^{-j}\}_{j = 0}^{k-1}$ is a frame for $L_2(\gamma)$.
\end{proof}
We now give a few auxiliary results in preparation for the case $\delta < \infty$, the first of which is interesting in its own right. 
\begin{proposition}
    \label{prop:no_lusin_basis}
    Let $\{q_j^\pm\}_{j = 0}^\infty \subset \Omega^\pm$ and write $\psi_j^\pm(w) = b_j^\pm(w-q_j^\mp)^{-(k+1)}$ for the corresponding Lusin wavelets, where $k \ge 0$ is fixed. 
    Then 
    \begin{enumerate}[font = \rmfamily, label = \textup{(\roman*)}]
        \item      $\Span\{\psi_j^+\}_{j = 0}^\infty$ is dense in $L_2^+(\gamma)$ if and only if $\Span\{\psi_j^+\}_{j = 1}^\infty$ is dense in $L_2^+(\gamma).$ \label{item:no_lusin_basis_1}
        \item     $\Span(\{\psi_j^-\}_{j = 0}^\infty \cup \{w^{-j}\}_{j=1}^{k})$ is dense in $L_2^-(\gamma)$ if and only if $\Span(\{\psi_j^-\}_{j = 1}^\infty \cup \{w^{-j}\}_{j=1}^{k})$ is dense in $L_2^-(\gamma).$ \label{item:no_lusin_basis_2}
    \end{enumerate}
\end{proposition}
Note carefully that $j$ starts at 0 in the first parts of the statements and at 1 in the second parts of the statements.
It follows from the proposition that no set of Lusin wavelets (possibly together with a finite number of $w^{-j}$-vectors) is a Schauder basis for $L_2^+(\gamma)$ or $L_2^-(\gamma)$.
To see this, assume that $\{f_j\}_{j=0}^\infty$ is a Schauder basis for a Hilbert space $H$.
Then, by the closed graph theorem, the mapping $f = \sum_{j=0}^\infty \alpha_{j} f_j \mapsto \alpha_0 f_0$ is a bounded, complemented, and oblique projection.
Hence,
\[H = \Span\{f_0\} \oplus \overline{\Span\{f_j\}_{j=1}^\infty}\]
is a topological splitting, which in particular implies that $\Span\{f_j\}_{j=1}^\infty$ is not dense in $H$. To prove Proposition~\ref{prop:no_lusin_basis}, we use a duality between $L_2^+(\gamma)$ and $L_2^-(\gamma)$, which follows abstractly from Theorem~\ref{thm:hardy_splitting}.
\begin{lemma}
\label{lem:hardy_duality}
    The Hardy subspaces $L_2^+(\gamma)$ and $L_2^-(\gamma)$ form a dual pair under the bilinear form
    \[\langle f, g\rangle  = \int_\gamma f(w)g(w)\,dw, \quad \quad f \in L_2^+(\gamma), g \in L_2^-(\gamma),\]
    in the sense that
    \begin{align}
        \|f\| &\simeq \sup_{0 \neq g \in L_2^-(\gamma)} \frac{|\langle f, g \rangle|}{\|g\|}, \quad \quad f \in L_2^+(\gamma), \label{eq:hardy_duality_plus}\\
                \|g\| &\simeq \sup_{0 \neq f \in L_2^+(\gamma)} \frac{|\langle f, g \rangle|}{\|f\|}, \quad \quad g \in L_2^-(\gamma). \label{eq:hardy_duality_minus}
    \end{align}
    In particular, a subset $U^\pm$ is dense in $L_2^\pm(\gamma)$ if and only if there is no non-zero $g^\mp \in L_2^\mp(\gamma)$ with $\langle u^\pm, g^\mp \rangle = 0$ for all $u^\pm \in U^\pm$. 
\end{lemma}
\begin{proof}
    Recall that $(\cdot, \cdot)$ denotes the canonical sequilinear inner-product on $L_2(\gamma)$.
    Let $f \in L_2^+(\gamma)$ be non-zero, and write $\overline{i \nu f} = (\overline{i \nu f})^+ + (\overline{i \nu f})^-$ according to~\eqref{eq:hardy_splitting}, so that $(\overline{i \nu f})^\pm \in L_2^\pm(\gamma)$.
    It then follows from Cauchy's integral theorem and Theorem~\ref{thm:hardy_splitting} that
    \begin{multline*}
        \|f\| 
        = \frac{(f, f)}{\|f\|} 
        = \frac{\langle f, \overline{i \nu f}\rangle}{\|f\|} 
        = \frac{\langle f, (\overline{i \nu f})^+ + (\overline{i \nu f})^- \rangle}{\|f\|} 
        =  \frac{\langle f, (\overline{i \nu f})^- \rangle}{\|f\|} 
        \\
        =  \frac{\langle f, (\overline{i \nu f})^- \rangle}{\|(\overline{i \nu f})^-\|} \frac{\|(\overline{i \nu f})^-\|}{\|f\|} 
        \lesssim    \frac{\langle f, (\overline{i \nu f})^- \rangle}{\|(\overline{i \nu f})^-\|} 
        \le \sup_{0 \neq g \in L_2^-(\gamma)} \frac{|\langle f, g \rangle|}{\|g\|}.
    \end{multline*}
    Applying the Cauchy-Schwarz inequality, we deduce that
    \[\|f\| \simeq \sup_{0 \neq g \in L_2^-(\gamma)} \frac{|\langle f, g \rangle|}{\|g\|}.\]
    The estimate for $L_2^-(\gamma)$ is proved analogously.
    To see that the density claim holds, assume that $U^+$ is a dense subset of $L_2^+(\gamma)$.
    Then, if $g^- \in L_2^-(\gamma)$ is such that $\langle u^+, g^- \rangle = 0$ for all $u^+ \in U^+$, it follows that
    \[\|g^-\| \simeq \sup_{0 \neq f \in L_2^+(\gamma) }\frac{|\langle f, g^-\rangle|}{\|f\|} = \sup_{0 \neq u^+ \in U^+ }\frac{|\langle u^+, g^-\rangle|}{\|u^+\|} = 0.\]
    For the other direction, assume that $U^+ \subset L_2^+(\gamma)$ is not dense.
    Then there is some non-zero $f^+ \in L_2^+(\gamma)$ so that $0 = (u^+, f^+) = \langle u^+, (\overline{i \nu f^+})^- \rangle$ for all $u^+ \in U^+$.
    From~\eqref{eq:hardy_duality_minus} we deduce that
    \[\|(\overline{i \nu f^+})^-\| \gtrsim \frac{|\langle f^+, (\overline{i \nu f^+})^- \rangle |}{\|f^+\|} = \frac{|\langle f^+, \overline{i \nu f^+} \rangle |}{\|f^+\|} = \|f^+\| > 0.\]
    So $(\overline{i \nu f^+})^-$ is a non-zero function in $L_2^-(\gamma)$ with $\langle u^+, (\overline{i \nu f^+})^- \rangle = 0$ for all $u^+ \in U^+$.
    The characterization of dense subsets of $L_2^-(\gamma)$ is proved similarly.
\end{proof}

\begin{proof}[Proof of Proposition~\ref{prop:no_lusin_basis}]
    We begin by proving~\ref{item:no_lusin_basis_2}  since this argument is slightly simpler.
    By contraposition, we assume that $\Span(\{\psi_j^-\}_{j = 1}^\infty \cup \{w^{-j}\}_{j=1}^{k})$ is not dense in $L_2^-(\gamma)$, and aim to prove that $\Span(\{\psi_j^-\}_{j = 0}^\infty \cup \{w^{-j}\}_{j=1}^{k})$ is not dense in $L_2^-(\gamma)$.
    By Lemma~\ref{lem:hardy_duality} and Cauchy's integral formula, there is some non-zero $f \in L_2^+(\gamma)$ such that $F^{(k)}(q_j) = 0$ for all $j \ge 1$ and $F^{(j)}(0) = 0$ for $j = 0, \dots, k-1$.
    It then suffices to show that there is some non-zero $g \in L_2^+(\gamma)$ with $G^{(k)}(q_j) = 0$ for all $j \ge 0$ and $G^{(j)}(0) = 0$ for $j = 0, \dots, k-1$.
    The main goal of the proof is to show that there is some $g \in L_2^+(\gamma)$ so that
    \[G^{(k)}(z) = (z-q_0)F^{(k)}(z), \quad \quad z \in \Omega^+.\]
    Fix $a \in \Omega^-$ and define
    \[g(w) = (w-q_0)f(w) - k \int_{0}^w F(z)\,dz - \sum_{j=0}^{k-1}c_j w^j, \quad \quad w \in \gamma,\]
     where $(c_j)_{j=0}^{k-1}$ are constants chosen so that $G^{(j)}(0) = 0$ for $j = 0, \dots, k-1$. 
    Then $g$ is a well-defined function in $L_2^+(\gamma)$ since $f \mapsto \int_0^w F(z)\,dz$ is a $L_2^+(\gamma) \to L_2^+(\gamma)$ Hilbert-Schmidt (hence bounded) operator.
    Indeed, it is an integral operator with kernel $K(z,w) =  \frac{1}{2\pi i}(\log(-z)-\log(w-z))$, which is a function in $L_2(\gamma \times \gamma)$.
    Further, by Leibniz' rule, we have that $G^{(k)}(z) = (z-q_0)F^{(k)}(z)$, so $G^{(k)}(q_j) = 0$ for all $j \ge 0$.
    The function $g$ is non-zero, as otherwise we would have $G^{(k)} = 0$, implying that $F^{(k)} = 0$ and hence $f=0$ by Theorem~\ref{thm:square_function_estimates}, a contradiction.

    For the statement~\ref{item:no_lusin_basis_1}, assume that there is some non-zero $f \in L_2^-(\gamma)$ so that $F^{(k)}(q_j) = 0$ for all $j \ge 1$.
    We want to show that there exists $g \in L_2^-(\gamma)$ such that
    \[G^{(k)}(z) = (1-q_0z^{-1})F^{(k)}(z), \quad \quad z \in \Omega^-.\]
    Define 
    \[g(w)=(1-q_0w^{-1})f(w) -kq_0w^{k-1} \int_{a}^w z^{-k-1}F(z)\,dz, \quad \quad w \in \gamma.\]
    % where $c$ is a constant selected so that $\int_\gamma \frac{g(w)}{w^k}\,dw = 0$.
    The integral $\int_a^w z^{-k-1}F(z)\,dz$ is well-defined (does not depend on the path connecting $a$ to $w$), since for a contour $\gamma' \subset \Omega^-$ with $\dist(\gamma, \gamma')$ sufficiently large, we have that 
    \[\int_{\gamma'} z^{-k-1}F(z)\,dz = \sum_{n=1}^\infty \alpha_n \int_{\gamma'} z^{-n-k-1}\,dz=0,\] 
    where $\sum_{n=1}^\infty \alpha_n z^{-n}$ is the Laurent series of $F(z)$.
    It follows that $g \in L_2(\gamma)$ since $f \mapsto \int_a^w z^{-k-1}F(z)\,dz$ is Hilbert-Schmidt as an operator $L_2^-(\gamma) \to L_2(\gamma)$.
    If $g$ is not in $L_2^-(\gamma)$, we may replace $g$ by $g^-$, the part of $g$ in $L_2^-(\gamma)$, as the Cauchy integral of $g$ and $g^{-}$ are the same for $z \in \Omega^-$. 
    To see that $G^{(k)}(z) = (1-q_0z^{-1})F^{(k)}(z)$, we compute the Laurent series of $G$ for $|z|$ sufficiently large:
    \begin{multline*}
        G(z) = F(z) -q_0\sum_{n=1}^\infty \alpha_n z^{-n-1} + kq_0z^{k-1}\sum_{n=1}^\infty \frac{\alpha_n}{n+k} z^{-n-k} \\
        = F(z) - q_0 \sum_{n=1}^\infty\alpha_n \frac{n}{n+k}z^{-n-1},
    \end{multline*}
    where the term $kq_0z^{k-1}\sum_{n=1}^\infty \frac{\alpha_n}{-n-k} a^{-n-k}$ arising from the integral given by $kq_0z^{k-1} \int_{a}^w \zeta^{-k-1}F(\zeta)\,d\zeta$ disappears as it is in $L_2^+(\gamma)$, and hence is annihilated when computing the Cauchy integral of $g$.
    Then $G^{(k)}(z) = (1-q_0z^{-1})F^{(k)}(z)$, for $|z|$ sufficiently large, which by analytic continuation extends to all of $\Omega^-$.
    Consequently, $G^{(k)}(q_j) = 0$ for all $j \ge 0$.
    Finally, $g \neq 0$, as otherwise $G^{(k)}(z) = (1-q_0z^{-1})F^{(k)}(z) = 0$, which implies that $f = 0$ by Theorem~\ref{thm:square_function_estimates}.
    % The remark in the statement of Proposition~\ref{prop:no_lusin_basis} after $\rm (i)$ and $\rm (ii)$ is an immediate consequence of Theorem~2.1.6 in Heil and Walnut~\cite{HeilWalnut_1989}.
\end{proof}
Next, we have another result related to the removal of Lusin wavelets from a collection of such functions. More precisely, given a frame of Lusin wavelets for $L_2(\gamma)$ with source points in $\Omega^\pm$ (together with the functions $\{w^{-j}\}_{j = 1}^k$), where the source points in $\Omega^-$ may form an unbounded set, we can remove all Lusin wavelets whose source points lie outside a bounded neighborhood of $\gamma$ without disturbing the frame property. This is possible because $\gamma$ is a compact curve.
\begin{lemma}\label{lem:removal_of_lusin_wavelets_at_a_positive_distance_from_gamma}
    Consider Whitney sets $\{q_j^\pm\}_{j = 1}^\infty \subset \Omega^\pm$ for $\Omega^\pm$ with respect to $\gamma$. Let $\psi_{j}^\pm$ be the normalized Lusin wavelets corresponding to $q_j^\mp$. Fix $\alpha > 0$ and let $\N = \I_1\sqcup \I_2$, where $d(q_j^-) \geq \alpha > 0$ for every $j \in \I_2$. Let $\J \subset \N$ be any finite set. If $\F := \{\psi^+_j\}_{j = 1}^\infty \cup \{\psi^-_j\}_{j = 1}^\infty \cup \{w^{-j}\}_{j = 1}^k$ is a frame for $L_2(\gamma)$, then so is $\{\psi^+_j\}_{j \in \I_1} \cup \{\psi^-_j\}_{\in \N \setminus \J} \cup \{w^{-j}\}_{j = 1}^k$.
\end{lemma}
\begin{proof}
    Define $U : L_2(\gamma) \to \ell^2(\Z)$ by 
    \begin{equation*}
         (Uf)_j =\begin{cases}
            (f,\psi^-_{-j-k+1}) \quad & \text{if } j \in (-\infty,-k], \\
            (f,w^{j-1}) & \text{if } j \in [-k+1,0], \\
            (f,\psi^+_j) \quad & \text{if } j \in [1,\infty).
        \end{cases}
    \end{equation*}  
    % $(Uf)_j = (f,\psi^-_{-j-k+1})$ if $j \leq -k$, $(Uf)_j = (f,w^{j-1})$ if $-k < j \leq 0$, and $(Uf)_j = (f,\psi^+_j)$ if $j \geq 1$. 
    Note that since $\F$ is a frame, $U$ is bounded, injective, and has closed range (in fact, these properties of $U$ are equivalent to $\F$ being a frame; see Section~\ref{sec:bvp_theory} for more on this). Decompose $U = U_1 + U_2$, where 
    \begin{align*}
        (U_1f)_j = \chi_{\{j \leq 0\}\sqcup \I_1}(j)(Uf)_j \quad \text{and} \quad (U_2f)_j = \chi_{\I_2}(j)(Uf)_j.
    \end{align*}
    We claim that $U_1$ has closed range and that $\dim N(U_1) < \infty$. Let $b_j^+$ be the normalizing constant for $\psi^+_j$, so that $|b_j^+|^2 \simeq d(q_j^-)^{2k-1}m(B_j)$, where $B_j = B(q_j^-,\varepsilon d(q_j^-))$. Let $Nf = \overline{i\nu f}$ and note that $|(Nf, \psi_j^+)| \simeq |b_j^+F^{(k)}(q_j^-)|$, where $F$ is the Cauchy integral of $f$. Given $f \in L_2(\gamma)$, using the mean-value property for holomorphic functions, we estimate,
    \begin{equation*}
    \begin{aligned}
        \|U_2Nf\|^2 & \simeq \sum_{j \in \I_2} |b_j^+F^{(k)}(q_j^-)|^2 \\
        & \simeq \sum_{j \in \I_2} d(q_j^-)^{2k-1}m(B_j)|F^{(k)}(q_j^-)|^2 \\
        & = \sum_{j \in \I_2} d(q_j^-)^{2k - 1}m(B_j)^{-1}\bigg|\int_{B_j}F^{(k)}(z)\,dA(z)\bigg|^2 \\
        & \leq\sum_{j \in \I_2} d(q_j^-)^{2k - 1}\int_{B_j}|F^{(k)}(z)|^2\,dA(z) \\
        & \lesssim \int_{\Omega^- \setminus \Omega^-_{\alpha/2}}|F^{(k)}(z)|^2d(z)^{2k - 1}\,dA(z).
    \end{aligned}
    \end{equation*}
    The last step uses the finiteness of the covering constant for $\{q_j^-\}_{j = 1}^\infty$, and that $z \in B_j$ implies $d(z) \geq \frac{1}{2}d(q_j^-)$. Since $K : L_2(\gamma) \to L_2(\C\setminus \gamma, d(z)^{2k-1})$ defined by $Kf(z) = \chi_{\Omega^- \setminus \Omega^-_{\alpha/2}}(z)F^{(k)}(z)$ has the kernel estimate $\lesssim \chi_{\Omega^- \setminus \Omega^-_{\alpha/2}}(z)|z|^{-2k-2}$, it is a Hilbert-Schmidt operator and hence is compact. As $\|U_2Nf\| \lesssim \|Kf\|$, this implies that $U_2N$ is compact. But compactness of operators is stable under composition with bounded operators, hence $U_2 = U_2N N^{-1}$ is compact. Then, since $U$ is injective with closed range, $U_1 = U - U_2$ has closed range and $\textup{dim}\,N(U_1) <\infty$ by Fredholm perturbation theory. 

    We claim that there exists a finite index set $\I_2' \subset \I_2$ such that $U_1 + U_2'$ is injective, where $(U_2'f)_j := \chi_{\I_2'}(j)(f,\psi^+_j)$. The following technique is found in Christensen~\cite[Lem.~23.2.1]{Christensen_2016}. Let $A$ and $B$ denote the lower and upper frame bound, respectively, of the frame $\F$. Let $\beta = \frac{1}{2}(\frac{A}{2B})^{1/2}$. Since $N(U_1)$ is finite-dimensional, the set $\{f \in N(U_1) \:; \|f\| = 1\}$ is compact. Thus, we can cover it by finitely many balls $B(f_i, \beta)$, $i = 1,...,n$, where $\|f_i\| = 1$. Since 
    \begin{equation*}
        A\|f_i\|^2 \leq \sum_{j = 1}^\infty |(f_i,\psi^+_j)|^2 + \sum_{j = 1}^\infty |(f_i,\psi^-_j)|^2 + \sum_{j = 1}^k |(f_i,w^{-j})|^2
    \end{equation*}
     for each $i$, there exists a finite set $\I_2' \subset \I_2$ such that
    \begin{equation*}
        \frac{A}{2}\|f_i\|^2 \leq \sum_{j \in \I_1 \sqcup \I'_2} |(f_i,\psi^+_j)|^2 + \sum_{j = 1}^\infty |(f_i,\psi^-_j)|^2 + \sum_{j = 1}^k |(f_i,w^{-j})|^2
    \end{equation*}
    for each $i$. Given $f \in N(U_1)$ with $\|f\| = 1$, $\|f-f_i\| < \beta$ for some $i$. By the reverse triangle inequality, we have
    \begin{align*}
        \|(U_1+U_2')f\|  \geq \|(U_1+U_2')f_i\| - \|(U_1+U_2')(f-f_i)\| \\
        \geq \sqrt{\frac{A}{2}} - \sqrt{B}\beta = \frac{1}{2}\sqrt{\frac{A}{2}}.
    \end{align*}
    Thus, $\|f\| \lesssim \|(U_1 + U_2')f\|$ for every $f \in N(U_1)$. Since $N(U_1 + U_2') \subset N(U_1)$, this implies that $U_1 + U_2'$ is injective. 

    Now, $U_2'$ is a finite rank operator, hence $U_1 + U_2'$ has closed range. Thus, by the bounded inverse theorem, $\|f\| \simeq \|(U_1 + U_2')f\|$ for every $f \in L_2(\gamma)$, that is, $\{\psi^+_j\}_{j \in \I_1\cup\I_2'} \cup \{\psi^-_j\}_{j =1}^\infty \cup \{w^{-j}\}_{j = 1}^k$ is a frame for $L_2(\gamma)$. Since $\psi_j^+ \in L^+_2(\gamma)$ while $\psi_j^-, w^{-j} \in L^-_2(\gamma)$, this implies in particular that $\Span\{\psi^+_j\}_{j \in \I_1\cup\I_2'}$ is dense in $L^+_2(\gamma)$, and that $\Span(\{\psi^-_j\}_{j =1}^\infty \cup \{w^{-j}\}_{j = 1}^k)$ is dense in $L_2^-(\gamma)$. As $\I_2'$ and $\J$ are finite sets, Proposition~\ref{prop:no_lusin_basis} implies that $\Span \{\psi^+_j\}_{j \in \I_1}$ is dense in $L^+_2(\gamma)$, and that $\Span(\{\psi^-_j\}_{\N \setminus \J} \cup \{w^{-j}\}_{j = 1}^k)$ is dense in $L^-_2(\gamma)$; hence $\Span(\{\psi^+_j\}_{j \in \I_1} \cup \{\psi^-_j\}_{\N \setminus \J} \cup \{w^{-j}\}_{j = 1}^k)$ is dense in $L_2(\gamma)$. According to Heil and Walnut~\cite[Thm.~2.1.6]{HeilWalnut_1989}, removal of a vector from a frame either leaves a frame or an incomplete set. By applying this result $|\I_2'| + |\J|$ times, we conclude that $\{\psi^+_j\}_{j \in \I_1} \cup \{\psi^-_j\}_{\N \setminus \J} \cup \{w^{-j}\}_{j = 1}^k$ is a frame for $L_2(\gamma)$.
\end{proof}

By applying Lemma~\ref{lem:removal_of_lusin_wavelets_at_a_positive_distance_from_gamma}, we now show how the case $\delta < \infty$ follows immediately from the case $\delta = \infty$.
\begin{proof}[Proof of Theorem~\ref{thm:lusin_frame}, case $\delta < \infty$.] 
    Let $\varepsilon \leq \varepsilon_0$, where $\varepsilon_0$ was specified in the proof of Theorem~\ref{thm:lusin_frame}, case $\delta = \infty$.
    If $\{q_j^+\}_{j=1}^\infty$ is not a Whitney set for $\Omega^+$, let $\O^+ = \Omega^+ \setminus \Omega^+_\delta$. According to Lemma~\ref{lem:existence_of_whitney_sets}, there exists an $\varepsilon$-fine Whitney set $\{p^+_j\}_{j = 1}^{J} \subset \O^+$ for $\O^+$ with respect to $\gamma$. Then $\{q_j^+\}_{j = 1}^\infty \cup \{p^+_j\}_{j = 1}^{J} \subset \Omega^+$ is an $\varepsilon$-fine Whitney set for $\Omega^+$ with respect to $\gamma$. Note that since $\O^+$ is compactly contained in $\Omega^+$, Condition~\ref{item:whitney_set_cond_2} of Definition~\ref{def:whitney_set} implies that $J < \infty$. We apply the same procedure for $\{q_j^-\}_{j=1}^\infty$ to obtain an $\varepsilon$-fine Whitney set $\{q_j^-\}_{j=1}^\infty \cup \{p^-_j\}_{j = 1}^{\infty} \subset \Omega^-$ for $\Omega^-$ with respect to $\gamma$. In this case, the additional source points $p_j^-$ are no longer finite in number since $\Omega^-\setminus\Omega^-_\delta$ is unbounded. By Theorem~\ref{thm:lusin_frame}, case $\delta = \infty$, $\mathcal{F} \cup \{\phi_j^+\}_{j = 1}^{\infty} \cup \{\phi_j^-\}_{j = 1}^{J}$ is a frame for $L_2(\gamma)$, where $\mathcal{F} =\{\psi_j^+\}_{j = 1}^\infty\cup\{\psi_j^-\}_{j = 1}^\infty \cup \{w^{-j}\}_{j=1}^{k}$, and $\phi_j^\pm$ are the Lusin wavelets corresponding to the source points $p^\mp_j$. Since $d(p_j^-) \geq \delta$ for every $j$, and $J < \infty$, Lemma~\ref{lem:removal_of_lusin_wavelets_at_a_positive_distance_from_gamma} shows that $\{\psi_j^+\}_{j = 1}^\infty\cup\{\psi_j^-\}_{j = 1}^\infty \cup \{w^{-j}\}_{j=1}^{k}$ is a frame for $L_2(\gamma)$, which concludes the proof of Theorem~\ref{thm:lusin_frame}.
\end{proof}

We end this section with a remark related to Proposition~\ref{prop:no_lusin_basis}. According to this proposition, a Lusin frame for $L_2(\gamma)$ is never a Schauder basis for $L_2(\gamma)$. It is, however, true that any finite collection of Lusin wavelets is linearly independent. To see this, suppose $\{\psi_1, ... , \psi_J\}$ is a collection of Lusin wavelets with distinct source points $p_1,...,p_J \in \C \setminus \gamma$. Assume that we can write $\psi_i = \sum_{j \neq i} \alpha_j\psi_j$ for some $1 \leq i \leq J$, with equality holding on $\gamma$. By analytic continuation, the equality, in fact, holds on $\C \setminus\{p_1,...,p_J\}$. By letting $z \to p_i$, the left-hand side blows up, while the right-hand side stays bounded, leading to a contradiction.

\section{Solving boundary value problems with frames}
\label{sec:bvp_theory}
In this section, we describe the type of boundary value problems we aim to solve, and show how the Lusin frames of Theorem~\ref{thm:lusin_frame} can be used to solve these problems. We start by introducing standard operators related to frames; see Christensen~\cite[Eqs.~(5.3)-(5.5)]{Christensen_2016}. Although these operators are not used in our numerical algorithm in Section~\ref{sec:numerical_results}, they provide a convenient vocabulary for this section. Suppose 
$\{f_j\}_{j = 1}^\infty$ is a frame for a Hilbert space $H$ (see Definition~\ref{def:frame}). We define the \textit{analyzing operator} of the frame by
\begin{align*}
    U: H \to \ell^2, \quad Uf =  ((f,f_j)_{H})_{j = 1}^\infty.
\end{align*}
By Definition~\ref{def:frame}, $U$ is well-defined and bounded. Because $\{f_j\}_{j = 1}^\infty$ is a frame, $\sum_{j=1}^\infty c_j f_j$ is norm convergent for every $(c_j)_{j = 1}^\infty \in \ell^2$. The adjoint of $U$, called the \textit{synthesizing operator}, and the \textit{frame operator} $S$, are given by
\begin{align*}
    U^* &: \ell^2 \to H, \quad U^* (c_j)_{j = 1}^\infty = \sum_{j=1}^\infty c_j f_j,\\
    S&:H \to H, \quad Sf = U^*Uf = \sum_{j=1}^\infty (f, f_j)_H f_j.
\end{align*}
Note that for any bounded linear operator $T : H \to K$ between Hilbert spaces, in particular for $U : H \to \ell^2$, the following are equivalent:
\begin{enumerate}[font = \rmfamily, label = (\roman*)]
    \item $\|T f\|_K \gtrsim \|f\|_H$ for all $f \in H$,
    \item $T$ is injective with closed range,
    \item $T^*$ is surjective,
    \item $T^*T$ is invertible.
\end{enumerate}
Thus, by the lower frame bound in Definition~\ref{def:frame}, the corresponding properties hold for the operators $U$, $U^*$, and $S = U^*U$. In particular, by $\rm (iii)$, every $f \in H$ has a (possibly non-unique) frame expansion $f = \sum_{j=1}^\infty c_j f_j$. In fact, $\rm (iv)$ gives the following representation
\begin{align}\label{eq:canonical_frame_expansion}
    f = SS^{-1}f = \sum_{j=1}^\infty (S^{-1}f,f_j)_H f_j = \sum_{j=1}^\infty (f,S^{-1}f_j)_Hf_j.
\end{align}
This is the unique frame expansion that minimizes the norm of the coefficients, since $c_j =(f,S^{-1}f_j)_H = (S^{-1}f,f_j)_H$  gives $(c_j)_{j = 1}^\infty \in R(U) = N(U^*)^\perp$. The collection of vectors \\$\{S^{-1}f_j\}_{j = 1}^\infty$ is also a frame for $H$ (which follows from Lemma~\ref{lem:frames_under_closed_range_mappings} below), called the \textit{dual frame of $\{f_j\}_{j = 1}^\infty$}. In general, there is no easy way to compute this dual frame, and hence, the coefficients $(f,S^{-1}f_j)_H$ may be difficult to obtain. 

Unlike bases, frames are always preserved under bounded, closed-range linear mappings.
\begin{lemma}\label{lem:frames_under_closed_range_mappings}
Let $H$ and $K$ be Hilbert spaces over either $\R$ or $\C$, as indicated in \textup{(i)} and  \textup{(ii)}, respectively. Suppose that $\{f_j\}_{j = 1}^\infty$ is a frame for $H$.
    \begin{enumerate}[font=\upshape, label = (\roman*)]
        \item If $H$ and $K$ are both real or both complex, and if $T : H \to K$ is linear, bounded, and has closed range, then $\{Tf_j\}_{j = 1}^\infty$ is a frame for $R(T)$.
        \item If $H$ is complex, $K$ is real, and if $T: H \to K$ is $\R$-linear, bounded, and has closed range, then $\{Tf_j\}_{j = 1}^\infty \cup \{Tif_j\}_{j = 1}^\infty$ is a frame for $R(T)$.
    \end{enumerate} 
\end{lemma}
\begin{proof}
    \begin{enumerate}[label = (\roman*)]
    \item Since $R(T) = N(T^*)^\perp$ and $T^*: N(T^*)^\perp \to H$ is injective with closed range, $\|T^* g \|_H \simeq \|g\|_K$ for all $g \in R(T)$. Therefore,
    \begin{align*}
        \sum_{j=1}^\infty |(g,Tf_j)_K|^2 = \sum_{j=1}^\infty |(T^*g,f_j)_H|^2 \simeq \|T^* g\|_H^2 \simeq \|g\|^2_K
    \end{align*}
    for every $g \in R(T)$.
    \item This follows from (i) after noting that $\{f_j\}_{j = 1}^\infty\cup\{if_j\}_{j = 1}^\infty$ is a frame for $H$, considered as a real Hilbert space equipped with the inner product $(\cdot,\cdot)_\R := \Re(\cdot,\cdot)_H$. \qedhere
    \end{enumerate}
\end{proof}
From Theorem~\ref{thm:lusin_frame}, we obtain Lusin frames for $L_2^\pm(\gamma)$.
\begin{proposition}\label{prop:frames_for_hardy_spaces}
    Let $\psi_j^\pm$ be defined as in Theorem~\ref{thm:lusin_frame}. Then $\{\psi_j^+\}_{j = 1}^\infty$ is a frame for $L_2^+(\gamma)$, and $\{\psi_j^-\}_{j = 1}^\infty\cup\{w^{-j}\}_{j = 1}^k$ is a frame for $L_2^-(\gamma)$. 
\end{proposition}
\begin{proof}
    Let $P^\pm$ be the complementary projections from Theorem~\ref{thm:hardy_splitting}. Since each $\psi_j^+$ is analytic in $\Omega^+$, $\psi_j^+ \in L_2^+(\gamma)$ for every $j \ge 1$. Similarly, $\psi_j^-$ and $w^{-j}$ $(1 \leq j \leq k)$ are analytic in $\Omega^-$ with decay at infinity, hence $\psi_j^-, w^{-j} \in L_2^-(\gamma)$. Theorem~\ref{thm:lusin_frame} and Lemma~\ref{lem:frames_under_closed_range_mappings}(i) then shows that 
    \begin{equation*}
        \{\psi_j^+\}_{j = 1}^\infty \cup \{0\} = \{P^+ \psi_j^+\}_{j = 1}^\infty \cup \{P^+\psi_j^-\}_{j = 1}^\infty\cup \{P^+w^{-j}\}_{j = 1}^k,
    \end{equation*}
    and hence $\{\psi_j^+\}_{j = 1}^\infty$ is a frame for $R(P^+) = L_2^+(\gamma)$. By instead applying Lemma~\ref{lem:frames_under_closed_range_mappings}(i) with $T = P^-$, we see that $\{\psi_j^-\}_{j = 1}^\infty\cup\{w^{-j}\}_{j = 1}^k$ is a frame for $L_2^-(\gamma)$.
\end{proof}

We consider the following boundary value problems for harmonic and analytic functions. Let $\widetilde T : L_2(\gamma;\R^2) \to L_2(\gamma; \R)$ be an $\R$-linear bounded operator with closed range, and suppose that $g \in R\big(\widetilde T\big)$. We want to find $u : \Omega^+ \to \R$ such that
\begin{equation}\label{eq:real_BVP}
    \begin{aligned}
        \begin{cases}
            \Delta u = 0  \quad & \text{in } \Omega^+,\\
            \widetilde T(\nabla u|_\gamma) = g & \text{on } \gamma.
        \end{cases}
    \end{aligned}
\end{equation}
An equivalent complex, first-order formulation of this problem is as follows. Given a harmonic function $u$ in $\Omega^+$, we obtain an analytic function $f$ in $\Omega^+$ by considering $f = (\partial_x - i \partial_y)u$. Thus, if $u$ solves~\eqref{eq:real_BVP}, and $Tf = \widetilde T(\Re f,- \Im f)$, then $f$ solves
\begin{equation}\label{eq:complex_BVP}
    \begin{aligned}
        \begin{cases}
            (\partial_x + i \partial_y) f = 0 \quad & \text{in } \Omega^+, \\
            T(f|_\gamma) = g & \text{on } \gamma.
        \end{cases}
    \end{aligned}
\end{equation}
Conversely, if $f$ solves~\eqref{eq:complex_BVP}, then $u = \Re \widetilde f$ solves~\eqref{eq:real_BVP}, where $\widetilde f$ is a complex primitive of $f$ in $\Omega^+$ (which exists since $\Omega^+$ is simply connected). Since those analytic functions $f$ in $\Omega^+$ which have a trace $f|_\gamma \in L_2(\gamma)$ in the sense that
\begin{align*}
    \lim_{t \downarrow 0}\int_\gamma \big|f(e^{-t}w) - f|_\gamma(w)\big|^2\,|dw| = 0
\end{align*}
are precisely the Cauchy integrals of functions in $L_2^+(\gamma)$, we may rephrase~\eqref{eq:complex_BVP} as: 
\begin{equation}\label{eq:hardy_space_formulation_of_BVP}
    \text{find $f \in L_2^+(\gamma)$ such that $Tf = g$.}
\end{equation}
Thus, we can assume that $T$ is defined on $L_2^+(\gamma)$, and we fix $T : L_2^+(\gamma) \to L_2(\gamma;\R)$ such that $T$ is $\R$-linear, bounded, and has closed range.
\begin{example}[Neumann problem]\label{ex:interior_Neumann_problem}
    Let $\nu$ be the almost everywhere defined outward unit normal of $\Omega^+$, and choose 
    \begin{align*}
        Tf = \Re(\nu f).
    \end{align*}
    Let $f \in L_2^+(\gamma)$ and extend $f$ to $F$ in $\Omega^+$ through its Cauchy integral. Let $\widetilde F$ be a complex primitive of $F$ in $\Omega^+$ and define $u = \Re \widetilde F$. If $z \in \Omega^+$ and $w \in \gamma$, then, using the Cauchy-Riemann equations,
    \begin{align*}
        \nu(w) \cdot \nabla u(z) & = \nu_1(w) \partial_x (\Re \widetilde F)(z) -\nu_2(w)\partial_x (\Im \widetilde F)(z) \\
        & = \Re \big((\nu_1 + i \nu_2)(w) \partial_x (\Re \widetilde F + i \Im \widetilde F)(z)  \big) = \Re(\nu(w)F(z)).
    \end{align*}
    Letting $z \to w$ we obtain $\partial_\nu u = \Re(\nu f)$ on $\gamma$. Thus, we see that with $Tf = \Re(\nu f)$, \eqref{eq:hardy_space_formulation_of_BVP} is the complex version of the interior Neumann problem. As was shown first by Jerison and Kenig~\cite{JerisonKenig_1981}, this problem and the problem in Example~\ref{ex:interior_Dirichlet_problem} below are well-posed on Lipschitz domains if the data is taken in $L_2(\gamma;\R)\cap \{ \int_\gamma f\,d\sigma = 0\}$.
\end{example}
\begin{example}[Dirichlet regularity problem]\label{ex:interior_Dirichlet_problem}
    Let $\tau = i \nu$ be the almost everywhere defined positively oriented unit tangential vector of $\gamma$, and choose
    \begin{align*}
        Tf = \Re(\tau f).
    \end{align*}
    With the same notation as in Example~\ref{ex:interior_Neumann_problem}, we have 
    \begin{align*}
        \tau(w)\cdot \nabla u(z) = \Re(\tau(w)F(z)), \quad \quad z \in \Omega^+, w \in \gamma,
    \end{align*}
    and hence $\partial_\tau u = \Re(\tau f)$ on $\gamma$. Thus, with this choice of $T$, \eqref{eq:hardy_space_formulation_of_BVP} is the complex version of the interior Dirichlet regularity problem. Compared to the ordinary Dirichlet problem with boundary condition $u |_\gamma = g$, the Dirichlet regularity problem requires control over the tangential derivative of $u$; the boundary data $u|_\gamma$ is then recovered (up to a constant) by integrating $\partial_\tau u $ over $\gamma$. This problem is also well-posed when the data is taken in $L_2(\gamma;\R)\cap \{ \int_\gamma f\,d\sigma = 0\}$; see Jersion and Kenig~\cite{JerisonKenig_1981}.
\end{example}

We now describe our method of fundamental solutions for solving~\eqref{eq:hardy_space_formulation_of_BVP} using Lusin frames. By Proposition~\ref{prop:frames_for_hardy_spaces}, $\{\psi_j^+\}_{j = 1}^\infty$ is a frame for $L_2^+(\gamma)$; let $U^*$ be the corresponding synthesizing operator. Then every $f \in L_2^+(\gamma)$ has a frame expansion $f = U^*(c_j)_{j=1}^\infty = \sum_{j = 1}^\infty c_j \psi_j^+$. Writing $c_j = d_j + id_{-j}$ with $d_j \in \R$, and imposing the boundary condition $Tf = g$, then gives
\begin{align}\label{eq:bdry_cond_and_frame_expansion}
    g = T \sum_{j = 1}^\infty (d_j + id_{-j})\psi_j^+ =  \sum_{j = 1}^\infty d_jT\psi_j^+ + \sum_{j = 1}^\infty d_{-j} T i \psi_j^+ = \sum_{j \in \Z_0} d_{j} \varphi^+_j,
\end{align}
where $\Z_0 := \Z \setminus \{0\}$ and
\begin{equation}
\label{eq:real_wavelets_def}
    \varphi^+_j = \begin{cases}
        T\psi_j^+ \quad & \text{if } j > 0, \\
        Ti\psi_{-j}^+ & \text{if } j < 0.
    \end{cases}
\end{equation} 
By Lemma~\ref{lem:frames_under_closed_range_mappings}(ii), $\{\varphi^+_j\}_{j\in \Z_0}$ is a frame for $R(T)$, and hence~\eqref{eq:bdry_cond_and_frame_expansion} is simply the equation \\$U_T^*(d_j)_{j \in \Z_0} = g$, $(d_j)_{j \in \Z_0} \in \ell^2(\Z_0;\R)$, where $U^*_T$ is the synthesizing operator for $ \{\varphi^+_j\}_{j \in \Z_0}$. As $U_T^*$ is onto $R(T)$,~\eqref{eq:bdry_cond_and_frame_expansion} can be solved for $(d_j)_{j \in \Z_0}$ whenever $g \in R(T)$. This then yields a solution $f = U^*(d_j + id_{-j})_{j = 1}^\infty$ to~\eqref{eq:hardy_space_formulation_of_BVP}.

By Proposition~\ref{prop:no_lusin_basis}, $\{\psi_j^+\}_{j = 1}^\infty$ is not a basis for $L_2^+(\gamma)$; hence the real Lusin frame $\{\varphi^+_j\}_{j \in \Z_0}$ is not a basis for $R(T)$, and the coefficients $(d_j)_{j \in \Z_0}$ are therefore not unique. As we saw in~\eqref{eq:canonical_frame_expansion}, one possible choice for the coefficients is
\begin{equation}
    \label{eq:dual_frame_coefficients}
    d_j = (S_T^{-1}g, \varphi^+_j) = (g, S_T^{-1}\varphi^+_j),
\end{equation}
where $S_T = U_T^* U_T$ is the frame operator for $\{\varphi^+_j\}_{j \in \Z_0}$. Since this approach uses the invertible operator $S_T$ to compute the coefficients, in contrast to directly trying to solve the underdetermined system 
\begin{equation}\label{eq:underdetermined_system}
    U_T^* (d_j)_{j \in \Z_0} = g,
\end{equation}
one might expect to end up with an invertible matrix when implementing this method numerically. However, it turns out that truncating and discretizing $S_T$ directly leads to a highly ill-conditioned and non-invertible system. Therefore, we shall directly truncate and discretize~\eqref{eq:underdetermined_system} in Section~\ref{sec:numerical_results}.

\section{Truncation and numerical results}\label{sec:numerical_results}
In this section, we implement the algorithm outlined in Section~\ref{sec:bvp_theory} and present numerical results. 
We begin with a brief overview of the parameters used in this section. 
For precise definitions, we refer the reader to the opening paragraph of Section~\ref{sec:sfe}.
The fineness parameter $\varepsilon$ from Definition~\ref{def:whitney_set} governs the sample density of the source points, with $\varepsilon = 0$ representing a continuous frame and $\varepsilon > 0$ a discrete frame.
The parameter $\delta$ is roughly the maximal distance between a source point and the boundary $\gamma$.
The frame is truncated to a finite collection of Lusin wavelets $\psi_j$ with source points in a finite number of layers enumerated by $N_0, \dots, N$, where each source point $q_{j,l}$ in layer $l$ is given by 
\[q_{j,l} =(1+(1+\varepsilon)^{-l}) r_\gamma(\theta_{j,l})e^{i\theta_{j,l}},\]
where $\theta_{j,l} = -\pi + 2\pi j/n_l$ is the angle of $q_{j,l}$ and $n_l = \lceil 2\pi \varepsilon^{-1}(1+(1+\varepsilon)^l)\rceil$ is the number of source points in layer $l$. 
Thus $\cup_{l=N_0}^\infty \{q_{j,l}\}_{j=1}^{n_l} = \{q_j\}_{j=1}^\infty$ is an $\varepsilon'$-fine Whitney set with $\varepsilon' \simeq \varepsilon$ (see Example~\ref{ex:explicit_whitney_partition}).
In our implementation below, we always take $\varepsilon = 0.3$.
We fix $N_0 = 0$, which implicitly determines $\delta$, and let $N$ vary.
The total number of wavelets is then $s_N = \sum_{l=0}^N n_l$. 

With this notation fixed, we describe our algorithm: the \emph{Whitney method of fundamental solutions} (WMFS). 
% To limit the scope of our numerical experiments, we only consider the interior Neumann problem, as formulated in~\eqref{eq:hardy_space_formulation_of_BVP} for the operator $Tf = \Re (\nu f)$; see Example~\ref{ex:interior_Neumann_problem}.
We begin by considering the interior Neumann problem, as formulated in~\eqref{eq:hardy_space_formulation_of_BVP}, for the operator $Tf = \Re(\nu f)$; see Example~\ref{ex:interior_Neumann_problem}. In Examples~\ref{ex:smooth_data_u} and~\ref{ex:smooth_data_dirichlet}, we show how to modify the algorithm to solve the harmonic Neumann and Dirichlet problems as well.
The idea is to represent the boundary data $g = Tf$ in the truncated Lusin frame (see Theorem~\ref{thm:lusin_frame}) as
\begin{equation}
    \label{eq:g_finite_representation}
     \sum_{1 \le |j|\le s_N}d_j \varphi_j^+ = g,%= F^*_{T,N}d,
\end{equation}
where $\varphi_j^+$ is defined in~\eqref{eq:real_wavelets_def}.
To compute the coefficients $(d_j)_{1\le|j|\le s_N}$, we pair~\eqref{eq:g_finite_representation} with boundary elements $\chi_m$, which yields the finite system of equations
\begin{equation}
    \label{eq:wmfs_system_of_eqs}
    \sum_{1 \le |j|\le s_N}d_j (\varphi_j^+, \chi_m) = (g, \chi_m), \quad \quad 1 \le m \le M.
\end{equation}
Here, $\chi_m = \chi_{\gamma_m}$ are indicator functions for the curve segments $\gamma_m \subset \gamma$ between the points $(w_m)_{m=1}^M \subset \gamma$, ordered such that $\Arg w_m$ is increasing.
More precisely,
\[\gamma_m = \begin{cases}
    \{r_\gamma(\theta)e^{i\theta} \in \gamma \:; \Arg w_{m-1} \le \theta < \Arg w_m\}, & 2\le m \le M, \\
    \{r_\gamma(\theta)e^{i\theta} \in \gamma \:;  -\pi \le \theta < \Arg w_1 \text{ or } \Arg w_M \le \theta \le \pi\}, & m=1.
\end{cases}\]
A simple choice of the points $(w_m)_{m=1}^M$ is to space them uniformly  (with respect to $\theta$) on $\gamma$. 
This works well when the solution $f$ has an analytic continuation outside of the domain $\Omega^+$; however, to allow for less regular data (as in Example~\ref{ex:nonsmooth_data}), we instead pick $(w_m)_{m=1}^M$ adapted to the source points $(q_j)_{j=1}^{s_N}$, so that each source point corresponds to $M_0 \in \Z_+$ $w$-points. 
Explicitly, writing $q_j= r_j r_\gamma(\theta_j)e^{i \theta_j}$, we take
\[\{w_m\} = \bigcup_{j=1}^{s_N}\{r_\gamma(\theta)e^{i\theta}\:; \theta \in \Theta_j^{M_0}\},\]
where
\[
    \Theta_j^{M_0} =  \biggl\{\theta_j \pm \frac{(r_j-1)}{M_0/2}, \theta_j \pm \frac{2(r_j-1)}{M_0/2}, \dots,  \theta_j \pm \frac{(r_j-1)M_0/2}{M_0/2}\biggr\}
\]
if $M_0$ is even, and
\[
    \Theta_j^{M_0} =  \biggl\{\theta_j, \theta_j \pm \frac{(r_j-1)}{(M_0-1)/2}, \dots, \theta_j \pm \frac{(r_j-1)(M_0-1)/2}{(M_0-1)/2}\biggr\}
\]
if $M_0$ is odd.
Using this setup, the matrix in~\eqref{eq:wmfs_system_of_eqs} is $M_0 s_N \times 2 s_N$.

One reason that we use indicator functions $\chi_m$ is that the inner products $(\varphi_j^+, \chi_m)$ in~\eqref{eq:wmfs_system_of_eqs} may be computed analytically (for the Neumann problem) by
\begin{multline}
\label{eq:real_wav_analytic_ip_plus}
    (\varphi_j^+, \chi_m) 
    = \int_{w_{m-1}}^{w_m} \varphi_j^+(w) \,|dw| 
    =\int_{w_{m-1}}^{w_m} \Re \bigl(\nu(w)\psi_j^+(w)\bigr) \,|dw| \\
    = b_j \Re \biggl(-i\int_{w_{m-1}}^{w_m} (w-q_j)^{-k-1}  \,dw\biggr) 
    = \frac{b_j}{k} \Im \bigl((w_{m-1}-q_j)^{-k}-(w_m-q_j)^{-k}\bigr)
\end{multline}
and similarly 
\begin{equation}
\label{eq:real_wav_analytic_ip_minus}
    (\varphi_{-j}^+, \chi_m) = \frac{b_j}{k} \Re \bigl((w_{m-1}-q_j)^{-k}-(w_m-q_j)^{-k}\bigr),
\end{equation}
for $j>0$. 
Here, $b_j$ are constants picked such that $\|\varphi_j^+\| \simeq 1$ for all $1 \le |j| \le s_N$.

The integrals $(g, \chi_m)$ on the right-hand side of~\eqref{eq:wmfs_system_of_eqs} are calculated numerically using 10-point Gauss-Legendre quadrature. 
As the coefficients $(d_j)_{1 \le |j|\le s_N}$ in~\eqref{eq:wmfs_system_of_eqs} may not be unique, we pick the coefficients with the minimal $\ell_2$ norm, which in MATLAB corresponds to the function \texttt{lsqminnorm}. 
This is a canonical choice, as the norm-minimizing coefficients in~\eqref{eq:g_finite_representation} for $M= s_N=\infty$ are those given by inner products with the dual frame of $(\varphi_j^+)_{j \in \Z_0}$; see~\eqref{eq:canonical_frame_expansion} and the remark after the equation.
A technical detail is that we set the tolerance of \texttt{lsqminnorm} to zero to disable the default low-rank approximation of the matrix in~\eqref{eq:wmfs_system_of_eqs}. 
Using the default settings leads to larger errors of several orders of magnitude in Example~\ref{ex:nonsmooth_data}.

In summary, the WMFS consists of two main steps. 
First, we solve for the coefficients $(d_j)_{1 \le |j|\le s_N}$ in~\eqref{eq:wmfs_system_of_eqs}, where the boundary integrals $(g, \chi_m) = \int_{w_{m-1}}^{w_m}g(w)\,|dw|$ are computed using Gauss-Legendre quadrature, and $(\varphi_j^+, \chi_m)$ are computed analytically using the formulas in~\eqref{eq:real_wav_analytic_ip_plus} and~\eqref{eq:real_wav_analytic_ip_minus}.
Second, we compute the approximate solution in the domain by evaluating 
\begin{equation}
\label{eq:step_2_wmfs}
    f(z) = \sum_{j=1}^{s_N}(d_j+id_{-j})\psi_j(z), \quad \quad z \in \Omega^+.
\end{equation}

Before moving on to numerical results, we make a few remarks about the algorithm.
\begin{remark}
Instead of pairing~\eqref{eq:g_finite_representation} with indicator functions as in~\eqref{eq:wmfs_system_of_eqs}, one may require~\eqref{eq:g_finite_representation} to hold on a set of so-called collocation points $(w_m)_{m=1}^M\subset \gamma$.
This is easier to implement and, in many cases, gives results comparable to the method described above using indicator functions.
However, we found that when source points are placed close to the boundary, using collocation points may result in less robust computations, with large fluctuations in the error as parameters $\varepsilon$ and $M_0$ are perturbed.
In contrast, the results for boundary elements~\eqref{eq:wmfs_system_of_eqs} are less sensitive to changes in parameters, especially for non-smooth data as in Example~\ref{ex:nonsmooth_data} below. % Is this still true for LSQMINNORM?    
\end{remark}
\begin{remark}\label{rmk:CC-method}
    A different variant of~\eqref{eq:wmfs_system_of_eqs} is to pair~\eqref{eq:g_finite_representation} with another (typically smaller) set of frame functions $(\varphi_m^+)_{1 \le |m| \le M}$, so that~\eqref{eq:wmfs_system_of_eqs} reads
\begin{equation}
    \label{eq:first_order_cc}
     \sum_{1\le |j|\le s_N}d_j (\varphi_j^+, \varphi_m^+) = (g, \varphi_m^+), \quad \quad 1 \le |m| \le M.
\end{equation}
In this case, given that $M \ll s_N$, one obtains an underdetermined system of equations which is equivalent to the \emph{Casazza-Christensen} (CC) method, a general method for determining coefficients of frame expansions; see Christensen~\cite[Thm. 23.2.3]{Christensen_2016}. 
The idea of the CC method is to approximate the inverse of the frame operator $S$ from Section~\ref{sec:bvp_theory} by
\[S^{-1} \approx (P_{M}S_{s_N})^{-1}P_M =: S_{\text{CC}}^{-1}\]
for $M \ll s_N$, where $P_{M}$ is the orthogonal projection mapping $L_2(\gamma ; \R)$ to $\Span \{\varphi_m^+\}_{1 \le |m| \le M}$ and
\[S_{s_N} : \Span \{\varphi_m^+\}_{1 \le |m| \le M} \to \Span \{\varphi_j^+\}_{1 \le |j| \le s_N}, \quad  S_{s_N}g = \sum_{1 \le |j| \le s_N} (g, \varphi_j^+) \varphi_j^+\]
is a truncated frame operator for $\{\varphi_j^+\}_{j \in \Z_0}$.
% As $M \to \infty$, 
The coefficients $(d_j)_{1 \le |j| \le s_N}$ in~\eqref{eq:g_finite_representation} may then be computed from the formula
\begin{equation}
    \label{eq:second_order_cc_coeffs}
    d_j = (S_{CC}^{-1}g, \varphi_j^+), \quad \quad 1 \le |j| \le s_N.
\end{equation}
To see that this gives the same coefficients as in~\eqref{eq:first_order_cc}, note that $h := S_{\text{CC}}^{-1}g$ is the unique function in $\Span\{\varphi^+_m\}_{1\le|m|\le M}$ satisfying $P_M S_{s_N}h = P_M g$, or equivalently
\[(S_{s_N}h, \varphi_m^+) = (g, \varphi_m^+), \quad \quad 1 \le |m| \le M.\]
After expanding the sum $S_{s_N}h$ and applying the relation in~\eqref{eq:second_order_cc_coeffs}, \eqref{eq:first_order_cc} follows.

While this is a well-conditioned method in the sense that the condition number of $S_{\text{CC}}^{-1}$ is bounded (if $M  \ll  s_N$), see Jonsson and Timlin~\cite[Thm.~5.4.4, Fig~6.5]{JonssonTimlin_2024}, numerical experiments indicate that the convergence of this method is slower than that of the WMFS, though still exponential (in $L_2(\Omega^+)$ norm) for polynomial solutions.
Further, the error of the CC method is several orders of magnitude higher for points near the boundary than for points not too close to the boundary, even for solutions with an analytic continuation outside of the domain.
In addition, the $L_2(\Omega^+)$ error of the CC method plateaus at a higher point than for the WMFS, with the CC method never reaching errors below approximately $10^{-12}$ while the WMFS attains a minimum error of around $10^{-14}$; see Jonsson and Timlin~\cite[Figs.~6.1-6.4]{JonssonTimlin_2024}.
\end{remark}
\begin{remark}
    As mentioned in Section~\ref{sec:intro}, we do not present measures of numerical stability such as condition numbers or coefficient norms in this section, as a detailed analysis of the numerical stability of the WMFS is beyond the scope of this paper. 
    However, the size of coefficient norms is closely related to the accuracy of the solutions, as evaluating $f_a$ always carries with it an error of order $(\text{coefficient norm}) \times 10^{-16}$ due to round-off errors; see Barnett and Betcke~\cite{Barnett_2008}. 
    In Figures~\ref{fig:Ca} and~\ref{fig:Cb}, we observe differences in accuracy of several orders of magnitude, and indeed these differences are reflected in the sizes of the coefficient norms. 
    See also Remark~\ref{rmk:error_decay_rate} for a consequence of the same phenomenon. 
    Even though the coefficient norms are large in a few cases (which can be rectified by adaptive placement of source points), we consistently find that the coefficient norm does not grow with the number of source points. 
\end{remark}
In the Examples~\ref{ex:smooth_data} and~\ref{ex:nonsmooth_data} below, we evaluate the accuracy of the WMFS in the following way.
First, we select an analytic function $ f : \Omega^+ \to \C $ defined by an explicit formula. 
Next, we compute the Neumann trace $ g = \Re(\nu f) $ analytically. Then, using \( g \), we numerically estimate \( f \) by solving for the coefficients \( (d_j)_{1 \le |j| \le s_N} \) in~\eqref{eq:wmfs_system_of_eqs} and subsequently evaluating~\eqref{eq:step_2_wmfs}.
The error of the WMFS is then the absolute value of the difference between $f$ and the approximation~\eqref{eq:step_2_wmfs} of $f$.
To estimate the $L_\infty(\Omega^+)$ error, we compute the maximum pointwise error in the intersection of the domain with a $1000 \times 1000$ square grid.
In Example~\ref{ex:smooth_unknown_data}, we have no formula for the exact solution.
Therefore, we instead estimate the error of the WMFS by measuring the difference between solutions and an "exact" solution computed with a significantly higher number of source points.

In Examples~\ref{ex:smooth_data_u} and~\ref{ex:smooth_data_dirichlet}, we solve for a harmonic potential $u$ rather than a holomorphic function $f$. In these examples, we choose an explicit holomorphic function $f$, and then compute the Neumann data $g = \Re(\nu f)$ and the Dirichlet data $g = \Re(f)$. From $f$, we analytically compute the exact solutions $u$. The error of the WMFS is then estimated as in Examples~\ref{ex:smooth_data} and~\ref{ex:nonsmooth_data} by computing the absolute difference between $u$ and the WMFS approximations~\eqref{eq:u_formula_wmfs} and~\eqref{eq:dir_u_wmfs_representation} of $u$.

\begin{example}[Smooth data and domain]
\label{ex:smooth_data}
Consider the domain parameterized by $r_\gamma(\theta) = 3+\cos(4\theta)$, with boundary data $g = \Re(\nu f_1)$, where $f_1(z) = e^{z/3-iz/10}\sin (z/3)$ is an entire function.
In Figure~\ref{fig:Aa}, we display by a color plot the pointwise error of the WMFS for $s_N = 148$ source points distributed over three layers, with $M_0=5$, so that the matrix in~\eqref{eq:wmfs_system_of_eqs} is $740 \times 296$.
As we see in Figure~\ref{fig:Ab}, increasing the number of source points from this point does not further decrease the error.
In contrast to the BIE method discussed in Section~\ref{sec:intro}, the error of the WMFS in Figure~\ref{fig:Aa} is mostly uniform in the domain, all the way up to the boundary. 
In particular, this means that all $L_p(\Omega^+)$ errors for $1 \le p \le \infty$ are, up to a few orders of magnitude, comparable in size. 
Therefore, we only present the $L_\infty(\Omega^+)$ error in Figure~\ref{fig:Ab} and all the following figures.

\begin{figure}
     \centering
     \begin{subfigure}[t]{0.4\textwidth}
         \centering
         \includegraphics[width=\textwidth]{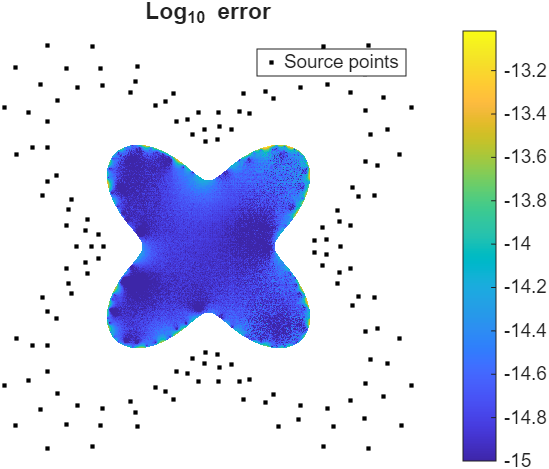}
         \caption{Distribution of source points and pointwise error as a color plot for the WMFS with 148 source points in three layers. Here, $M_0=5$. %, so the number of boundary elements is $M=186 \times 5 = 930$.
         }
         \label{fig:Aa}
     \end{subfigure}
     \hfill
     \begin{subfigure}[t]{0.4\textwidth}
         \centering
         \includegraphics[width=\textwidth]{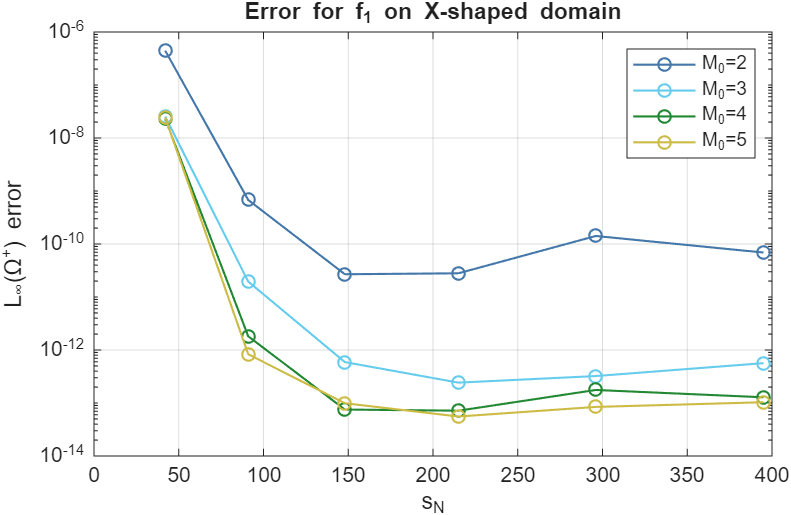}
         \caption{$L_\infty(\Omega^+)$ error of the WMFS to the number of source points $s_N$ for different values of $M_0$.}
         \label{fig:Ab}
     \end{subfigure}
        \caption{Pointwise and $L_\infty(\Omega^+)$ error of the WMFS for the solution $f_1(z) = e^{z/3-iz/10}\sin (z/3)$.}
        \label{fig:x_entire_data}
\end{figure}

An additional conclusion from Figure~\ref{fig:Ab} is that solving an overdetermined system in~\eqref{eq:wmfs_system_of_eqs}, which corresponds to $M_0 \ge 3$, is necessary for the error to reach levels close to machine epsilon $\approx 2 \times 10^{-16}$. 
For such choices of $M_0$, the WMFS has an exponential rate of convergence in the number of source points, at least when the solution is an entire function.
In the semi-logarithmic plot shown in Figure~\ref{fig:Ab}, such exponential convergence corresponds to linear error curves.
\end{example}

\begin{example}[Harmonic Neumann problem]\label{ex:smooth_data_u}
In Example~\ref{ex:smooth_data}, we used the WMFS to compute the conjugate gradient $f$ of the potential $u$ solving the Neumann problem~\eqref{eq:neumann_problem}.
In applications, however, one might be more interested in computing $u$.
As $u = \Re \widetilde f$, where $\widetilde f$ is a complex primitive of $f$, one can explicitly compute $u$ by modifying~\eqref{eq:step_2_wmfs} in the following way:
given the coefficients $(d_j)_{1 \le |j| \le s_N}$, we find a complex primitive of $f$ by
\[\widetilde f(z) = -\frac{1}{k}\sum_{j=1}^{s_N}b_j(d_j + i d_{-j})(z-q_j)^{-k}, \quad \quad z \in \Omega^+,\]
from which we obtain the formula
\begin{equation}
    \label{eq:u_formula_wmfs}
    u(z) = \Re \biggl(-\frac{1}{k}\sum_{j=1}^{s_N}b_j(d_j + i d_{-j})(z-q_j)^{-k} \biggr), \quad \quad z \in \Omega^+.
\end{equation}
In Figure~\ref{fig:ABa}, we display the error of the WMFS for the same domain and boundary data as in Example~\ref{ex:smooth_data}, but for the corresponding potential $u_1$ instead of $f_1$.
As the computed potential may differ from $u_1$ by a constant, we modify~\eqref{eq:u_formula_wmfs} by adding a constant chosen so that the average error in a $200 \times 200$ grid in the center of the domain in Figure~\ref{fig:ABa} vanishes.
The minimal $L_\infty(\Omega^+)$ error for the potential $u_1$ is attained with 148 source points and $M_0 = 5$, at which point the error is roughly $10^{-14}$.
    \begin{figure}
     \centering
     \begin{subfigure}[t]{0.4\textwidth}
         \centering
         \includegraphics[width=\textwidth]{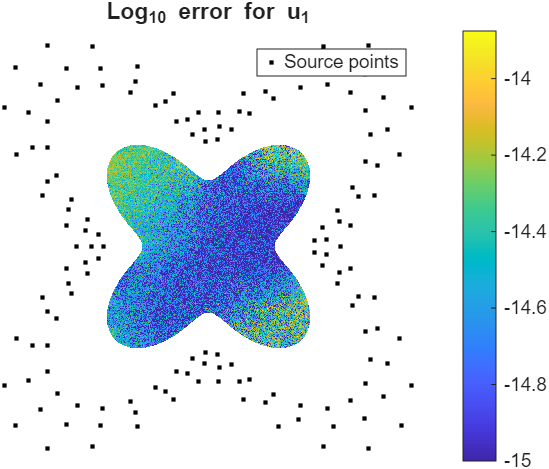}
         \caption{Error of the WMFS for the harmonic Neumann problem. 
         }
         \label{fig:ABa}
     \end{subfigure}
     \hfill
     \begin{subfigure}[t]{0.4\textwidth}
         \centering
         \includegraphics[width=\textwidth]{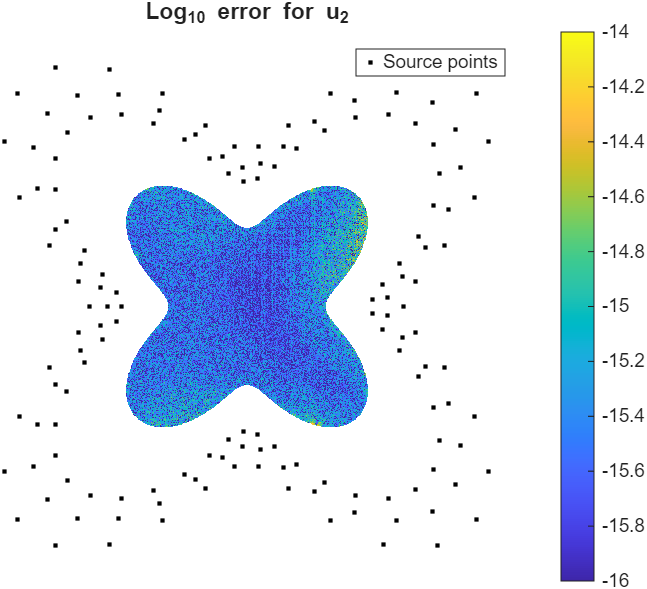}
         \caption{Error of the WMFS for the harmonic Dirichlet problem.}
         \label{fig:ABb}
     \end{subfigure}
        \caption{Pointwise errors of the WMFS for (\textsc{a}) Neumann data $g = \Re(\nu f_1)$, and (\textsc{b}) Dirichlet data $g = \Re(f_1)$, where $f_1(z) = e^{z/3-iz/10}\sin(z/3)$. In both cases, $s_N = 148$ and $M_0 = 5$.}
        \label{fig:x_entire_data_u}
\end{figure}
\end{example}

\begin{example}[Harmonic Dirichlet problem]\label{ex:smooth_data_dirichlet}
As in Example~\ref{ex:smooth_data_u}, we can modify the numerical algorithm to instead solve the harmonic Dirichlet problem and find the potential $u : \Omega^+ \to \R$ with prescribed boundary data $g = u|_\gamma$.
The algorithm proceeds as follows.
Given a harmonic function $u$, there exists a conjugate function $v$ such that $\widetilde f = u + iv$ is holomorphic.
Consequently, $\widetilde f$ may be approximated by Lusin wavelets 
\[\widetilde f(z) \approx \widetilde f_a(z) =\sum_{j=1}^{s_N} (d_j + i d_{-j}) \frac{b_j}{(z-q_j)^{k+1}}, \quad \quad z \in \Omega^+,\]
which, after taking real parts, yields the approximation $u_a$ of $u$ given by
\begin{equation}
    \label{eq:dir_u_wmfs_representation}
    u_a(z) = \sum_{1 \le |j| \le s_N}d_j \widetilde \varphi_j^+(z), \quad \quad z \in \Omega^+,
\end{equation}
where 
\[\widetilde \varphi_j^+(z) = \begin{cases}
     b_j\Re( (z-q_j)^{-(k+1)}) \quad & \text{if }j>0, \\
     b_{-j}\Re ( i(z-q_{-j})^{-(k+1)}) & \text{if }j<0.
\end{cases}\]
We compute the coefficients $(d_j)_{1 \le |j| \le s_N}$ by solving the system
\begin{equation}
    \label{eq:dir_g_wmfs_representation}
    g(w_m) = \sum_{1\le |j| \le s_N}^{s_n} d_j \widetilde \varphi_j^+(w_m), \quad \quad 1 \le m \le M,
\end{equation}
obtained by evaluating~\eqref{eq:dir_u_wmfs_representation} in a set of collocation points $\{w_m\}_1^M \subset \gamma$.
The approximate solution $u_a$ is then obtained by~\eqref{eq:dir_u_wmfs_representation}.
While it is possible to solve for $(d_j)_{1 \le |j| \le s_N}$ using boundary elements as in~\eqref{eq:wmfs_system_of_eqs}, this is not as efficient as in the other numerical examples, since we do not have an analytic formula for the inner products $(\widetilde \varphi_j^+, \chi_m)$.
Because of this, we use collocation points instead.

Figure~\ref{fig:ABb} displays the pointwise error of the WMFS for the same domain as in Example~\ref{ex:smooth_data}, with boundary data
\[g(w) = \Re(f_1(w)) = \Re(e^{z/3-iz/10}\sin(z/3)).\]
We use 148 source points and $M_0 = 5$, which means that 740 collocation points are used. 
This yields a maximal error of roughly $7 \times 10^{-15}$.
Further increasing the number of source points, or the value of $M_0$, does not reduce the error beyond this point.
\end{example}

\begin{example}[Non-smooth data and domain] \label{ex:nonsmooth_data}
Consider the square domain $\Omega^+ = \{x+iy\:;-1 < x, y < 1\}$. 
In Figure~\ref{fig:square_smooth_vs_nonsmooth}, we see the $L_\infty(\Omega^+)$ error of the WMFS, where the solution is either an entire function $f_1(z) = e^{z/3-iz/10}\sin (z/3)$ or a non-smooth function $f_2(z) = \sqrt{z-p}$, where $\sqrt{\cdot}$ denotes the principal branch of the square root. 
Here, the non-differentiable point $p = -1-i/10$ of $f_2$ is a point on $\gamma$.
To ensure that the inner products $(g, \chi_m)$ are computed accurately, we place points $w_m$ at the corners of $\gamma$, so that the segments $\gamma_m$ partitioning $\gamma$ are always smooth, as well as one point at $p$.
Just like for the smooth domain, we see in Figure~\ref{fig:Ba} that the WMFS converges exponentially when the solution is an entire function, even though $\gamma$ is not smooth.
For the non-smooth solution in Figure~\ref{fig:Bb}, we still observe convergence, however, at a significantly lower speed than for the entire solution.
The linear curve in the log-log plot in Figure~\ref{fig:Bb} for $M_0=5$ corresponds to a polynomial convergence rate.
Again, overdetermined systems, that is, where $M_0 \ge 3$, result in lower minimum errors for both the smooth and non-smooth solutions.
\begin{figure}
     \centering
     \begin{subfigure}[t]{0.4\textwidth}
         \centering
         \includegraphics[width=\textwidth]{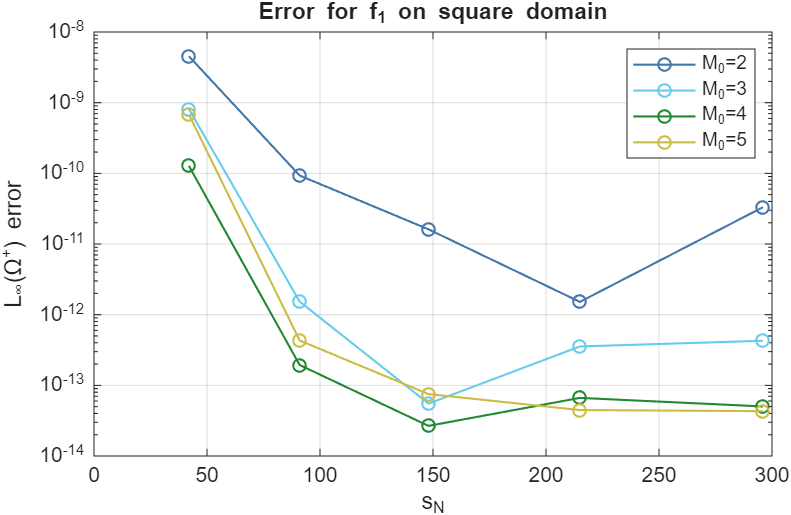}
         \caption{$L_\infty(\Omega^+)$ error for the solution $f_1(z) = e^{z/3-iz/10}\sin (z/3)$.}
         \label{fig:Ba}
     \end{subfigure}
     \hfill
     \begin{subfigure}[t]{0.4\textwidth}
         \centering
         \includegraphics[width=\textwidth]{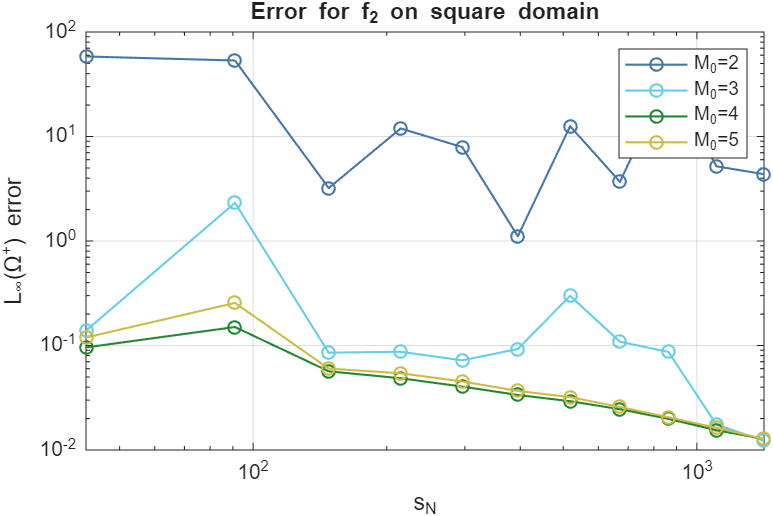}
         \caption{$L_\infty(\Omega^+)$ error for the solution $f_2(z) = \sqrt{z+1+i/10}$.}
         \label{fig:Bb}
     \end{subfigure}
        \caption{Comparison of $L_\infty(\Omega^+)$ errors of the WMFS for (\textsc{a}) an entire solution and (\textsc{b}) a non-smooth solution on the square domain, for different values of $M_0$ and varying number of source points $s_N$. Note that (\textsc{a}) is a semi-logarithmic plot while (\textsc{b}) is a log-log plot.}
        \label{fig:square_smooth_vs_nonsmooth}
\end{figure}

To improve the convergence in the non-smooth case, we add Lusin wavelets localized near the singularity $p$ of $f_2$, as this is where the absolute error in Figure~\ref{fig:Ca} is the largest.
Geometrically, such wavelets correspond to source points in a cone whose apex is at $p$ (see Figure~\ref{fig:Cb}).
The $L_\infty(\Omega^+)$ error for this method is displayed in Figure~\ref{fig:Da}.
Using this scheme, the $L_\infty(\Omega^+)$ error decreases until it reaches a minimum at around $s_N=700$ source points, after which the error is approximately $10^{-12}$. %, although the error is slightly higher close to the singularity of $f_2$.
As the error curves in the semi-logarithmic plot in Figure~\ref{fig:Da} are linear, we further conclude that the rate of convergence is exponential.

\begin{figure}
     \centering
     \begin{subfigure}[t]{0.4\textwidth} 
         \centering
         \includegraphics[width=\textwidth]{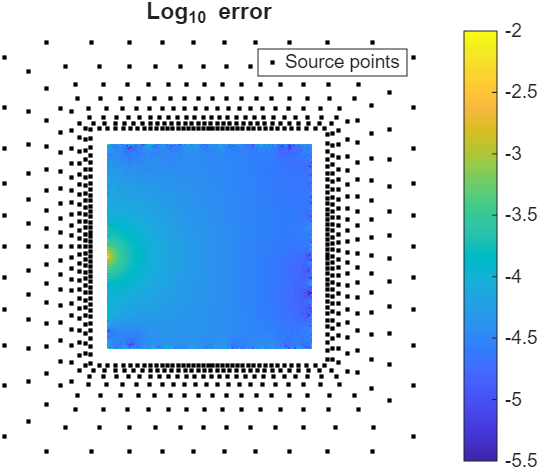}
         \caption{Error plot using a Whitney distribution of $s_N = 671$ source points in eight layers, with $M_0=5$.}
         \label{fig:Ca}
     \end{subfigure}
     \hspace{0.05\textwidth}  
     \begin{subfigure}[t]{0.4\textwidth}
         \centering
         \includegraphics[width=\textwidth]{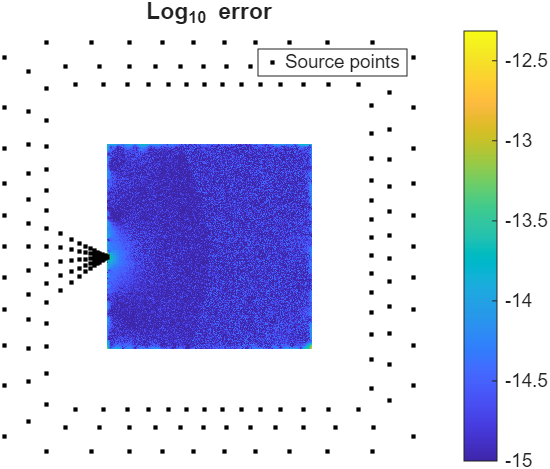}
         \caption{Error plot using a distribution of $s_N = 448$ source points, with 148 placed in three outer layers and the rest arranged in a cone whose apex is at the singularity of the solution. Here, $M_0=5$.}
         \label{fig:Cb}
     \end{subfigure}
        \caption{Comparison of absolute error of the WMFS as color plots for the solution $f_2(z) = \sqrt{z+1+i/10}$, with source points according to (\textsc{a}) a standard truncated Whitney distribution or (\textsc{b}) with further source points added in a cone whose apex is at $-1-i/10$. Note the different color scales in (\textsc{a}) and (\textsc{b}).} 
        \label{fig:square_whitney_vs_cone}
\end{figure}

Note that adding Lusin wavelets with source points as in Figure~\ref{fig:Cb} uses that frame expansions are unconditionally convergent.
Depending on the function $f$, and by extension, its boundary data $g$, it may be appropriate to use a particular exhaustive sequence when computing the frame expansion of $f$.
In the case of non-smooth boundary data, adding more source points close to the singularity appears beneficial, as shown in Figure~\ref{fig:Cb}.

To finish this example, we compare these results with the classical method of fundamental solutions, where one places the source points $(q_j)_{j=1}^{s_N}$ of the Lusin wavelets $(\psi_j)_{j=1}^{s_N}$ uniformly on a curve in $\Omega^-$. 
Explicitly, we take
\begin{equation}
    \label{eq:mfs_source_points}
    q_j = r_{\text{MFS}}r_\gamma({2 \pi j}/{s_N})e^{i \frac{2 \pi}{s_N}j}, \quad \quad 1 \le j \le s_N,
\end{equation}
where the parameter $r_{\text{MFS}}$ specifies the (approximate) distance of the source points to the boundary.
An implementation of this method for the solution $f_2$ is displayed in Figure~\ref{fig:Db}, where we see that the corresponding error never reaches below $10^{-3}$ in Figure~\ref{fig:Db}, regardless of the choice of $r_{\text{MFS}}$.

\begin{figure}
     \centering
     \begin{subfigure}[t]{0.4\textwidth}
         \centering
         \includegraphics[width=\textwidth]{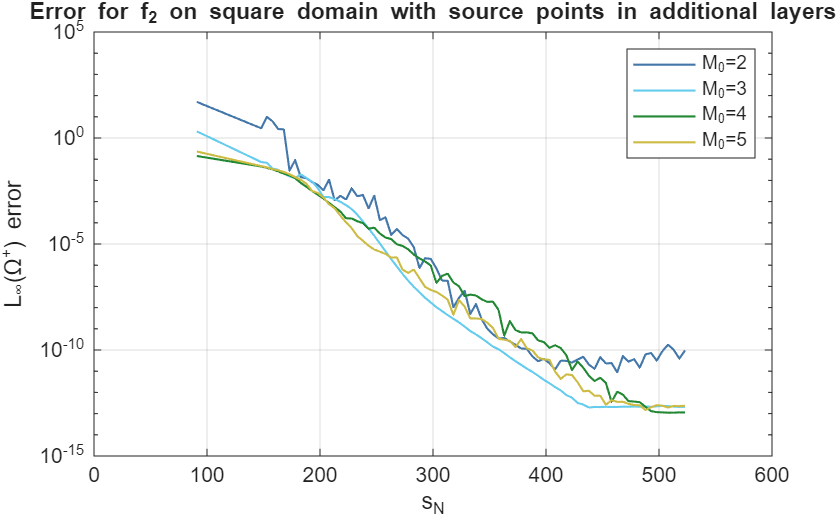}
         \caption{$L_\infty(\Omega^+)$ error for the WMFS when additional source points are placed in a cone as in Figure~\ref{fig:Cb}. The first 148 source points are in three layers of a truncated Whitney set.}
         \label{fig:Da}
     \end{subfigure}
     \hfill
     \begin{subfigure}[t]{0.4\textwidth}
         \centering
         \includegraphics[width=\textwidth]{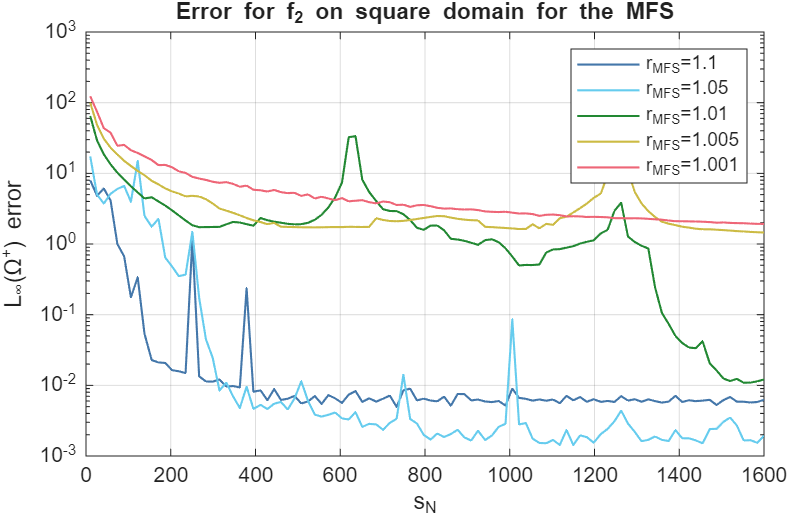}
         \caption{$L_\infty(\Omega^+)$ error for the MFS for various choices of $r_{\text{MFS}}$.}
         \label{fig:Db}
     \end{subfigure}
        \caption{Comparison of $L_\infty(\Omega^+)$ errors for a non-smooth solution $f_2$ on the square domain, when the source points are either (\textsc{a}) placed in a Whitney configuration plus a cone as in Figure~\ref{fig:Cb}, or (\textsc{b}) according to classical MFS configurations along a curve as in~\eqref{eq:mfs_source_points}.}
        \label{fig:square_conewmfs_vs_mfs}
\end{figure}

\end{example}

\begin{example}
\label{ex:smooth_unknown_data}
    Consider again the square domain, and let 
    \begin{equation}
        \label{eq:unknown_bdry_data}
        g_3(x+iy) = \Re \bigl(\nu(x+iy)(xy^2 \sin(y) - i \sin(x^3) \cos(x y)) \bigr) - c,
    \end{equation}
    where $c \in \R$ is a constant chosen so that $\int_\gamma g(w)\,|dw| = 0$.
    For this boundary data, we have no explicit formula for the solution $f_3$ satisfying $\Re(\nu f_3) = g_3$ on $\gamma$. 
    To gauge the accuracy of the WMFS in this case, we estimate the error of the solutions by comparing them to a reference solution $\widetilde f_3$, computed using a high number of source points.
    More precisely, the "exact" solution $\widetilde f_3$ is calculated with 863 source points in nine outer layers and with 600 source points in four cones with the corners of the square domain as their apexes, as in Figure~\ref{fig:Ea}.
    So, in total, $s_N = 1463$ for $\widetilde f_3$.
    Like in Example~\ref{ex:nonsmooth_data}, we allocate additional source points in cones to increase the rate of convergence. 
    In Figure~\ref{fig:Eb}, we see that the estimated error decreases exponentially since the error curve is linear in the semi-logarithmic plot.
    An estimated $L_\infty(\Omega^+)$ error of approximately $7 \times 10^{-13}$ is achieved at $s_N = 556$.
    The estimated error of this solution is displayed as a color plot in Figure~\ref{fig:Ea}.

    \begin{remark}\label{rmk:error_decay_rate}
        Although placing additional source points in cones with the corners of the domain as their apexes accelerates convergence, it is not necessary for convergence.
        If source points are simply placed in layers of a truncated Whitney partition, the estimated $L_\infty(\Omega^+)$ error still decreases, however only at a polynomial rate, reaching a level of $10^{-4}$ at $s_N = 1814$.
    \end{remark}
    \begin{remark}        
    In~\eqref{eq:unknown_bdry_data}, the boundary data is of the form $g = \Re(\nu h)$, where $h$ is a smooth but non-harmonic function.
    This type of boundary data is common in applications.
    However, if one instead takes $g$ to be a smooth function, we still observe exponential convergence as in Figure~\ref{fig:Eb}, with a minimum estimated $L_\infty(\Omega^+)$ error at around $10^{-12}$.
    \end{remark}

\begin{figure}
     \centering
     \begin{subfigure}[t]{0.4\textwidth}
         \centering
         \includegraphics[width=\textwidth]{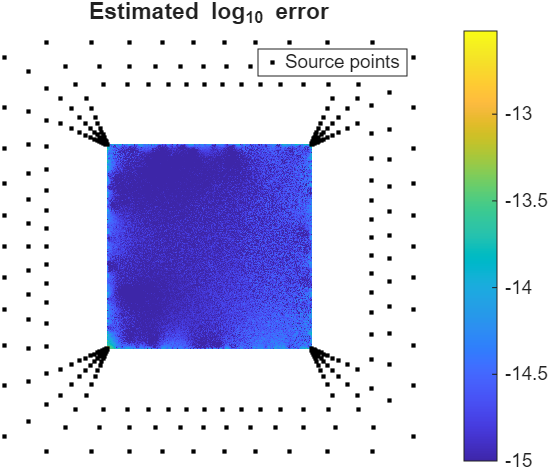}
         \caption{Distribution of source points and estimated pointwise error as a color plot for the WMFS with 556 source points, where 148 are in three outer layers and 408 are in four cones. Here, $M_0=5$. %, so the number of boundary elements is $M=186 \times 5 = 930$.
         }
         \label{fig:Ea}
     \end{subfigure}
     \hfill
     \begin{subfigure}[t]{0.4\textwidth}
         \centering
         \includegraphics[width=\textwidth]{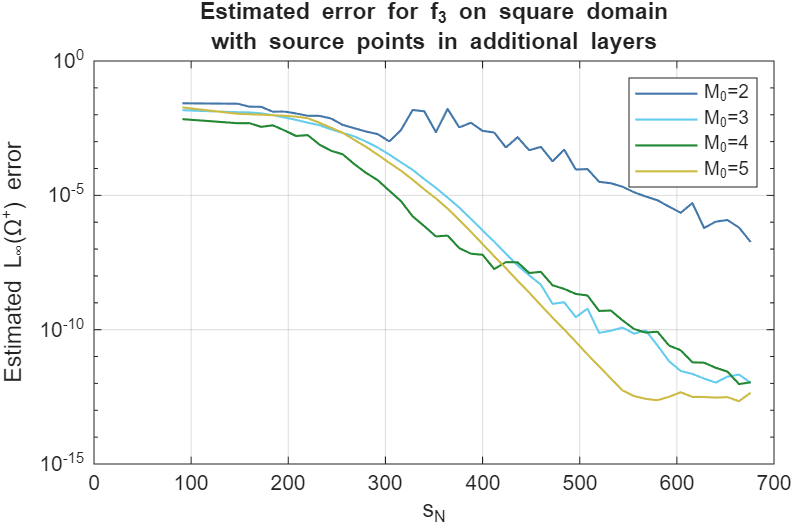}
         \caption{$L_\infty(\Omega^+)$ error of the WMFS to the number of source points $s_N$ for different values of $M_0$.}
         \label{fig:Eb}
     \end{subfigure}
        \caption{Estimated pointwise and $L_\infty(\Omega^+)$ error of the WMFS for the boundary data~\eqref{eq:unknown_bdry_data}.}
        \label{fig:no_holo_extension}
\end{figure}
\end{example}

To summarize, the WMFS is a simple and flexible algorithm that addresses the main downside of the classical MFS, numerical instability, without sacrificing the rapid rate of convergence, nor the uniform error up to the boundary.
Further, the WMFS can be applied even when the solution $f$ cannot be analytically continued to some neighborhood of $\overline{\Omega^+}$.
To improve the convergence rate in this case, one can use source points from a subset of the Whitney set from Example~\ref{ex:explicit_whitney_partition} adapted to $f$, as in Figure~\ref{fig:Cb}.

\newpage
\begin{appendices}
\section{Hardy splitting}
\label{appendix:hardy_splitting}
In this appendix, we prove Theorem~\ref{thm:hardy_splitting} from known estimates of the operator $BD$. According to Proposition~\ref{prop:bd_facts}, the spectrum of $BD$ is contained in the bisector $S_\omega$. Fix an angle $\mu \in (\omega, \pi/2)$ and define $S_\mu$ analogously to $S_\omega$, so that $S_\mu$ is the closed bisector with corresponding angle $\mu$. By $S_\mu^o$, we mean the interior of $S_\mu$. We define 
\begin{gather*}
    H(S^o_\mu) = \{\text{holomorphic functions in } S^o_\mu\}, \\
    H_\infty(S^o_\mu) = \{h \in H(S^o_\mu) \:; \|h\|_{\infty} < \infty\},
\end{gather*}
and we let $\H = R(BD)$. Note that $BD |_\H : \H \to \H$ is an injective operator since we split off the constants $N(BD)$; see Proposition~\ref{prop:bd_facts}. The main property of $BD$ that we need to prove Theorem~\ref{thm:hardy_splitting} is that $BD$ has a bounded holomorphic functional calculus:
\begin{theorem}\label{thm:functional_calculus}
    Let $\mathcal{L}(\H)$ denote the algebra of bounded operators on $\H$. There exists a bounded homomorphism 
    \begin{align*}
        H_\infty(S^o_\mu) \longrightarrow \mathcal{L}(\H), \quad h \longmapsto h(BD). 
    \end{align*}
\end{theorem}
Theorem~\ref{thm:functional_calculus} is a corollary of Theorem~\ref{thm:abstract_sfe} below. See, for example, Rosén, Keith, and McIntosh~\cite[Prop.~4.8, part (iv) of proof]{AKM_2006} for the abstract Schur estimate used. For $h \in H_\infty(S_\mu^o)$ with decay $|h(\lambda)| \lesssim \min(|\lambda|^s,|\lambda|^{-s})$ for some $s > 0$, $h(BD)$ is defined by 
\begin{align*}
    h(BD) = \frac{1}{2\pi i} \int_\Gamma \frac{h(w)}{w-BD}\,dw,
\end{align*}
where $\Gamma = \{re^{\pm i\theta}; r  \geq 0\} \cup \{re^{i(\pi\pm\theta)}; r  \geq 0\}$, $\omega < \theta < \mu$, parametrized counterclockwise around the two components. This definition is then extended to all of $H_\infty(S_\mu^0)$ using Lemma~\ref{lem:convergence_lemma} below. In fact, $BD$ has a functional calculus for a wider class of unbounded holomorphic symbols. More precisely, there is a homomorphism $h \mapsto h(BD)$ from the class of $h \in H(S^0_\mu)$ that satisfies a growth estimate $|h(\lambda)| \leq C(|\lambda|^{-s}+ |\lambda|^s)$ for some $s <\infty$, to the set of densely defined closed unbounded operators on $\H$. This functional calculus is compatible with the forming of polynomials of $BD$, that is, $(\lambda^k)(BD) = (BD)^k$, and extends the $H_\infty(S_\mu^0)$ functional calculus of Theorem~\ref{thm:functional_calculus}. This unbounded functional calculus is available for any injective and bisectorial operator satisfying resolvent estimates as in~\eqref{eq:resolvent_estimates}. We emphasize, however, that not every such operator satisfies the estimate $\|h(BD)\| \lesssim \|h\|_\infty$, and that this typically requires $D$ to be a differential operator and $B$ to be a multiplication operator. Our operator $BD$ enjoys the following convergence lemma.
\begin{lemma}\label{lem:convergence_lemma}
    If $\{h_\alpha\}$ is a sequence in $H_\infty(S^0_\mu)$ such that $\{h_\alpha\}$ is uniformly bounded in $S^0_\mu$ and $h_\alpha \to h \in H_\infty(S^0_\mu)$ pointwise, then $h_\alpha(BD) \to h(BD)$ in the strong operator topology.
\end{lemma}
See Albrecht, Duong, and McIntosh \cite[Sec.~D~and~Thm.~D]{ADM_1996} for a general exposition of holomorphic functional calculus for sectorial operators (which is straightforward to generalize to bisectorial operators) and a proof of Lemma~\ref{lem:convergence_lemma}. See also Auscher and Axelsson~\cite[Sec.~6.1]{AuscherRosén_2011} for short proofs of these results.

We use Theorem~\ref{thm:functional_calculus} to define the spaces and operators that we need. Let $\chi^+(\lambda)$ and $\chi^-(\lambda)$ be the indicator functions for $\Re \lambda > 0$ and $\Re \lambda < 0$, respectively. Then $\chi^\pm  \in H_\infty(S^0_\mu)$, so we can define the spectral projections
\begin{equation*}
    E^\pm = \chi^\pm(BD).
\end{equation*}
Since $h \mapsto h(BD)$ is a homomorphism, it is immediate from the properties of $\chi^\pm$ that $E^\pm$ are complementary projections, giving rise to the direct topological sum $\H = E^+\H \oplus E^-\H$. By Proposition~\ref{prop:bd_facts}, $L_2(\T)$ then splits topologically as
\begin{align}\label{eq:L2(T)_splitting}
    L_2(\T) = L_2^{+,b}(\T) \oplus L_2^{-,b}(\T),
\end{align}
where $L_2^{+,b}(\T) := E^+\H \oplus N(BD)$ and $L_2^{-,b}(\T) := E^-\H$. 
We will see that~\eqref{eq:L2(T)_splitting}, after a change of variables, corresponds precisely to the Hardy space splitting of Theorem~\ref{thm:hardy_splitting}. 

Henceforth, we extend any operator $h(BD)$ with $h \in H_\infty(S^o_\mu)$ to $L_2(\T)$ by implicitly precomposing with the projection $f \mapsto f_\H$ corresponding to the splitting $\H \oplus N(BD)$. Let $E_0$ be the projection $f \mapsto f_{N(BD)}$, and define
\begin{alignat*}{3}
    &C_t^+ && = \big(e^{-t\lambda}\chi^+(\lambda)\big)(BD) + E_0, \quad  && t > 0, \\
    &C_t^- && = \big(e^{-t\lambda}\chi^-(\lambda)\big)(BD) ,&&  t < 0.
\end{alignat*}
Given a function $f \in L_2(\T)$, we let
\begin{align*}
    f_t^+ = C_t^+ f \quad (t > 0) \quad \text{and} \quad f_t^- = C_t^- f\quad (t < 0 ).
\end{align*}
We can view $f_t^\pm$ either as extensions of $f$ to $\R_{\pm} \times \T$, where $\R_+ = (0,\infty)$ and $\R_- = (-\infty,0)$, or simply as functions in $L_2(\T)$ for $t$ fixed. As we show below, the relevance of $f_t^\pm$ is that they correspond to the Cauchy integrals of $f \in L_2(\gamma)$, after changing variables. 

By using the functional calculus for $BD$, we now derive some facts about $f_t^\pm$. Let $h_t(\lambda) = e^{-t\lambda}\chi^+(\lambda)$, so that $f_t^+ = h_t(BD)f + E_0f$. Clearly $|h_t| \leq 1$, $h_t \to \chi^+$ as $t \to 0^+$, and $h_t \to 0$ as $t \to \infty$. Thus, by Lemma~\ref{lem:convergence_lemma},
\begin{align*}
    f_t^+ \xrightarrow[t \to 0^+]{} E^+f + E_0f = f^+ \quad \text{and} \quad f_t^+ \xrightarrow[t \to \infty]{} E_0f
\end{align*}
in the topology of $L_2(\T)$. Similarly, $f_t^- \to f^-$ as $t \to 0^-$, and $f_t^- \to 0$ as $t \to -\infty$. By considering the symbols
\begin{align*}
     \frac{e^{-\lambda(t + \alpha) } - e^{-\lambda t}}{\alpha}\chi^\pm(\lambda),
\end{align*}
for $t \in \R_\pm$ fixed, we also find from Lemma~\ref{lem:convergence_lemma} and the fact that $(\lambda)(BD) = BD$, that
\begin{align*}\label{eq:t_derivative_of_f_t}
    \frac{f^\pm_{t + \alpha} - f^\pm_t}{\alpha}  \xrightarrow[\alpha \to 0]{L_2(\T)}  \big(-\lambda e^{-t\lambda}\chi^\pm(\lambda) \big)(BD)  f = - BDf_t^\pm.
\end{align*}
From this last observation it follows that $f^\pm(t,\theta) = f^\pm_t(\theta)$ satisfies the $(t,\theta)$-version~\eqref{eq:t_theta_cauchy_riemann} of the Cauchy-Riemann equations, in $\R_\pm \times \T$. For example, if $\varphi \in C_0^\infty(\R_+ \times \T)$, then
\begin{equation}\label{eq:integration_by_parts_for_f_t}
\begin{aligned}
    \int_0^\infty \int_0^{2\pi} f^+_t(\theta) \partial_t\varphi(t,\theta)\,d\theta dt & = \lim_{\alpha \to 0} \int_0^\infty \int_0^{2\pi} f^+_t(\theta) \frac{\varphi(t+\alpha,\theta) - \varphi(t,\theta)}{\alpha}\,d\theta dt \\
    & = \lim_{\alpha \to 0} -\int_0^\infty \int_0^{2\pi} \frac{f^+_{t-\alpha}(\theta) - f^+_t(\theta)}{-\alpha} \varphi(t,\theta)\,d\theta dt \\
    & = \int_0^\infty \int_0^{2\pi} BD f^+_t(\theta) \varphi(t,\theta)\,d\theta dt;
\end{aligned}
\end{equation}
hence $\partial_t f^+_t = - BD f^+_t$. Let us summarize our findings in a lemma.
\begin{lemma}\label{lem:properties_of_f_t}
    If $f \in L_2(\T)$ and we decompose $f = f^+ + f^-$ according~\eqref{eq:L2(T)_splitting}, then
    \begin{align*}
        f_t^\pm \xrightarrow[t \to 0^\pm]{} f^\pm, \quad f_t^+ \xrightarrow[t \to \infty]{} E_0f, \quad \text{and} \quad f_t^- \xrightarrow[t\to -\infty]{} 0,  
    \end{align*}
    in the topology of $L_2(\T)$. Moreover, in the sense of distributions, $(\partial_t + BD)f_t^\pm = 0$ in $\R_\pm \times \T$.
\end{lemma}

Let $R_\gamma : L_2(\gamma) \to L_2(\T)$ be the parametrization operator $R_\gamma f(\theta) = f(r_\gamma(\theta)e^{i\theta})$. For ${f\in L_2(\gamma)}$, we define $F^\pm$ in $\Omega_0^+ :=\Omega^+\setminus\{0\}$ and $\Omega^-$, respectively, by
\begin{align*}\label{eq:z_version_of_f_t}
    F^\pm(r_\gamma(\theta)e^{-t}e^{i\theta}) = C_t^\pm R_\gamma f(\theta).
\end{align*}
\begin{lemma}\label{lem:analyticity_of_f_t}
    The functions $F^\pm$ are analytic in $\Omega^+_0$ and $\Omega^-$, respectively, and $F^+$ has an analytic extension to $\Omega^+$.
\end{lemma}
\begin{proof}
        We start by showing that $\overline{\partial}_z F^+ = 0$ in $\Omega^+_0$, where we set $\overline{\partial}_z = \partial_x + i \partial_y$. To simplify the calculations slightly, we show the equivalent statement that $\overline{\partial}_z (z^{-1} F^+) = 0$ in $\Omega^+_0$. By elliptic regularity, it is sufficient to show that this equation holds in the sense of distributions. Thus, let $\varphi \in C_0^\infty(\Omega^+_0)$. We want to show that 
    \begin{align*}
        \int_{\Omega^+}z^{-1} F^+(z)\overline{\partial}_z\varphi(z)\,dA(z) = 0.
    \end{align*}    
    Let $g = R_\gamma f$, so that $F^\pm(r_\gamma(\theta)e^{-t}e^{i\theta}) = g_t^+(\theta)$. We will make the Lipschitz change of variables $z = r_\gamma(\theta)e^{-t}e^{i\theta}$, and it follows from the chain rule that 
    \begin{align*}
        \overline{\partial}_z\varphi(z) = \frac{ie^{-t}}{(r_\gamma(\theta)e^{-t})^2}\Big[\big(r_\gamma(\theta)e^{i\theta}\big)'\partial_t + r_\gamma(\theta)e^{i\theta} \partial_\theta \Big].
    \end{align*}
    Recalling that $b(\theta) = i \frac{r_\gamma(\theta)e^{i\theta}}{(r_\gamma(\theta)e^{i\theta})'}$, and using $dA = (r_\gamma(\theta)e^{-t})^2\,d\theta dt$, we get that
    \begin{equation}\label{eq:change_of_variables}
    \begin{aligned}
        \int_{\Omega^+}z^{-1} F^+(z)\overline{\partial}_z\varphi(z)\,dA(z) =  - \int_0^\infty \int_0^{2\pi} g_t^+(\theta) \big[b(\theta)^{-1}\partial_t + D\big]\varphi(t,\theta) \,d\theta dt.
    \end{aligned}
    \end{equation}    
    Using~\eqref{eq:integration_by_parts_for_f_t}, which is seen to hold also with $\varphi \in C_0^\infty(\R_+ \times \T)$ replaced by $b(\theta)^{-1}\varphi(t,\theta)$ where $b^{-1} \in L_\infty(\T)$ and $\varphi \in \text{Lip}_0(\R_+ \times \T)$, we get
    \begin{multline*}
        \int_0^\infty \int_0^{2\pi} g_t^+(\theta) b(\theta)^{-1}\partial_t\varphi(t,\theta) \,d\theta dt  = \int_0^\infty \int_0^{2\pi} \big(BDg_t^+(\theta)\big) b(\theta)^{-1}\varphi(t,\theta) \,d\theta dt \\
         = \int_0^\infty \int_0^{2\pi} D g_t^+(\theta) \varphi(t,\theta) \,d\theta dt 
         = -\int_0^\infty \int_0^{2\pi} g_t^+(\theta)D\varphi(t,\theta) \,d\theta dt,
    \end{multline*}
    where the last step follows from integration by parts. Thus, by~\eqref{eq:change_of_variables}, we have that $\overline{\partial}_z (z^{-1} F^+) = 0$ in $\Omega^+_0$. The proof that $F^-$ is analytic in $\Omega^-$ is analogous. 

    To see that $F^+$ can be analytically extended through the origin, let $\varphi \in C_0^\infty(\Omega^+)$, and let $\eta$ be a smooth function in $\C$ such that $\eta \equiv 0$ on $\overline{B(0,1/2)}$, $\eta \equiv 1$ on $\C \setminus B(0,1)$, and $0 \leq \eta \leq 1$. Put $\eta_\varepsilon(z) = \eta(z / \varepsilon)$. By the product rule and the analyticity of $F^+$ in $\Omega^+_0$, we have
    \begin{align}\label{eq:analytically_through_the_origin_integration_by_part_trick}
        \int_{\Omega^+} \eta_\varepsilon(z) F^+(z) \overline{\partial}_z \varphi(z)\,dA(z) = - \int_{\Omega^+}F^+(z)\varphi(z) \overline{\partial}_z\eta_\varepsilon(z) \,dA(z).
    \end{align}   
    Theorem~\ref{thm:functional_calculus} combined with $\sup_{(t,\lambda)\in (0,\infty)\times\C}|e^{-t\lambda}\chi^+(\lambda)| = 1$ implies that $\|g_t^+\|_{L_2(\T)} \lesssim \|g\|_{L_2(\T)}$. The same change of variables as before therefore gives    
    \begin{align*}
        \|F^+\|^2_{L_2(\Omega^+)} \lesssim \int_0^\infty \|g_t^+\|^2_{L_2(\T)} e^{-2t}\,dt \lesssim \|g\|_{L_2(\T)} \int_0^\infty e^{-2t}\,dt < \infty.
    \end{align*}
    Thus, $F^+ \in L_2(\Omega^+)$, so the left-hand side of~\eqref{eq:analytically_through_the_origin_integration_by_part_trick} tends to $\int_{\Omega^+} F^+(z)\overline{\partial}_z\varphi(z)dA$ as $\varepsilon \to 0$. Let $r_0 = \min_{\theta \in [0,2\pi]} r_\gamma(\theta)$ and $\Omega^+_\varepsilon = \{r_\gamma(\theta)e^{-t}e^{i\theta} : t \geq \ln(r_0/\varepsilon),\, \theta \in [0,2\pi]\} $. Using $\|\overline{\partial}_z\eta_{\varepsilon}\|_{\infty} \lesssim 1/ \varepsilon$, $\text{supp}\, \overline{\partial}_z\eta_\varepsilon \subset \Omega^+_\varepsilon$, and $\|g_t^+\|_{L_1(\T)} \lesssim \|g\|_{L_2(\T)}$, we estimate the right-hand side of~\eqref{eq:analytically_through_the_origin_integration_by_part_trick} as
        \begin{multline*}
        \bigg|\int_{\Omega^+} F^+(z)\varphi(z) \overline{\partial}_z\eta_\varepsilon(z) \,dA(z)\bigg|  \lesssim \frac{1}{\varepsilon}\int_{\Omega^+_\varepsilon} |F^+(z)|\,dA(z)  \\
          \lesssim \frac{1}{\varepsilon}\int_{\ln(r_0/\varepsilon)}^\infty  \|g_t^+\|_{L_1(\T)}e^{-2t}\,dt 
         \\
         \lesssim \frac{\|g\|_{L_2(\T)}}{\varepsilon}\int_{\ln(r_0/\varepsilon)}^\infty  e^{-2t}\,dt 
         = \frac{\|g\|_{L_2(\T)}}{2r_0^2}\varepsilon \xrightarrow[\varepsilon \to 0]{} 0.
    \end{multline*}
    Therefore, $\int_{\Omega^+}F^+(z)\overline{\partial}_z\varphi(z)\,dA(z) = 0$, showing that $F^+$ has a removable singularity at $z =~0$. 
\end{proof}

We are now in a position to prove Theorem~\ref{thm:hardy_splitting}. For $t \in \R_\pm$, define $F_t^\pm(w) = F^\pm (e^{-t}w)$ for $w \in \gamma$. Recall that every operator $h(BD)$ with $h \in H_\infty(S_\mu^o)$ is extended to act on $L_2(\T)$ by precomposing with the projection $f \mapsto f_\mathcal{H}$ corresponding the splitting $L_2(\T) = \mathcal{H} \oplus N(BD)$.
\begin{proof}[Proof of Theorem~\ref{thm:hardy_splitting}] The bounded complementary projections $E^+ + E_0$ and $E^-$ corresponding to the splitting~\eqref{eq:L2(T)_splitting}, induce, through similarity, bounded complementary projections $P^\pm$ on $L_2(\gamma)$ defined by
\begin{align*}
    P^+ = R_\gamma^{-1}(E^+ + E_0)R_\gamma \quad \text{and} \quad P^- = R_\gamma^{-1}E^-R_\gamma.
\end{align*}
Given $f \in L_2(\gamma)$, we define $f^\pm = P^\pm f$. Observe that $F^\pm_t = R_\gamma^{-1} C_t^\pm R_\gamma f$, so that due to Lemma~\ref{lem:properties_of_f_t}, we have the following convergences in $L_2(\gamma)$:
 \begin{align}\label{eq:convergence_of_F_t^+}
     F_t^+ \xrightarrow[t \to 0 ^+]{} R_\gamma^{-1}(E^+ + E_0)R_\gamma f = f^+ \quad \text{and} \quad F_t^- \xrightarrow[t \to 0^-]{} R_\gamma^{-1}E^- R_\gamma f = f^-.
 \end{align}
Thus, to show Theorem~\ref{thm:hardy_splitting}, it is enough to prove that $\pm F_t = F^\pm_t$ for all $f \in L_2(\gamma)$ and $t \in \R_\pm$. To this end, define 
\begin{gather*}
     V^\pm  = \big\{f_\varepsilon^\pm \in L_2(\T) \:; f \in L_2(\T),\,  \pm\varepsilon > 0 \big\} \subset L_2^{\pm,b}(\T), 
     \\ V = V^+ \oplus V^-, \quad W = R_\gamma^{-1}V.
\end{gather*}
By Lemma~\ref{lem:properties_of_f_t}, $V^\pm$ is dense in $L_2^{\pm,b}(\T)$. Thus, due to~\eqref{eq:L2(T)_splitting}, $V$ is dense in $L_2(\T)$; hence $W$ is dense in $L_2(\gamma)$.

Let $f \in W$, so that $R_\gamma f \in V$. It follows from the functional calculus of $BD$ that $C_t^+R_\gamma f = C_{t+\varepsilon}^+ g$ for some $g \in L_2(\T)$ and $\varepsilon > 0$. Thus, by Lemma~\ref{lem:analyticity_of_f_t}, $F^+$ is analytic in $\{r_\gamma(\theta)e^{-t}e^{i\theta} : t >-\varepsilon,\, \theta \in \T\}$, a neighborhood of $\overline{\Omega^+}$. In particular, $ F^+|_\gamma$ is well-defined, and by~\eqref{eq:convergence_of_F_t^+}, $F^+|_\gamma = \lim_{t \to 0^+}F^+_t = f^+$. Similarly, $F^-$ is analytic in a neighborhood of $\overline{\Omega^-}$, and $ F^-|_\gamma = f^-$. Lemma~\ref{lem:properties_of_f_t} also implies $\lim_{t \to -\infty} F_t^- = 0$. By Cauchy's integral theorem, we then have
    \begin{align*}
        F^+(z) = \frac{1}{2\pi i} \int_\gamma \frac{f^+(z)}{w-z}\,dw = \frac{1}{2\pi i} \int_\gamma \frac{f^+(z) + f^-(z)}{w-z}\,dw = F(z), \quad z \in \Omega^+.
    \end{align*}
Similarly, $ F^-(z) = - F(z)$ for $z \in \Omega^-$ (where the minus sign is due to the reversed orientation of $\gamma$ when considered as the boundary of $\Omega^-$). In particular, for fixed $t \in \R_\pm$, the $L_2$ bounded operators $f \mapsto F_t^\pm$ and $f \mapsto \pm F_t$ agree on the dense subset $W$, and are therefore equal everywhere. We conclude that for each $f \in L_2(\gamma)$, we have $\lim_{t \to 0^\pm}\pm F_t = f^\pm$ strongly in $L_2(\gamma)$.
\end{proof}

\section{Square function estimates}
\label{appendix:sfe}
In this appendix, we prove Theorem~\ref{thm:square_function_estimates}. As before, $\mu$ is a fixed angle in $(\omega, \pi/2)$.
\begin{theorem}\label{thm:abstract_sfe}
Fix $h \in H_\infty(S_\mu^o)$ such that $h \not \equiv 0$ on $S_{\mu}^o \cap \{\lambda : \textup{Re}\, \lambda > 0\}$ and $h \not \equiv 0$ on $S_{\mu}^o \cap \{\lambda : \textup{Re}\, \lambda < 0\}$, and suppose that $h$ satisfies the decay condition
\[|h(\lambda)| \lesssim |\lambda|^\alpha(1+|\lambda|)^{-2\alpha}\]
for some $\alpha > 0$. 
Then
\[\|f\|_{L_2(\T)}^2 \simeq \int_0^\infty \|h(t BD)f\|_{L_2(\T)}^2 \frac{dt}{t}\]
for all $f \in \H$.
\end{theorem}
The non-periodic result $(\T \to \R)$ is closely related to the seminal work of Coifman, McIntosh, and Meyer \cite{CMM_1982}. 
% See \cite[Conseq.~3.2]{AKM_2006} for references. 
An elementary proof of the (non-periodic) square function estimate is found in Coifman, Jones, and Semmes \cite[Sec.~1~and~2]{CoifmanJonesSemmes1989}. The periodic result~\ref{thm:abstract_sfe} is a straightforward tweak. A precise reference is~\cite[Thm.~7.1]{AuscherRosén_2012}, of which Theorem~\ref{thm:abstract_sfe} is the special case $n = 1$, $\sigma = 0$.
\begin{corollary}
    Let $k \geq 1$. For all $f\in R_\gamma^{-1}\H$ we have
    \label{cor:dirty_square_function_estimates}
    \begin{equation*}
    % \label{eq:dirty_sfe}
        \|f\|_{L_2(\gamma)}^2 \simeq \int_{\C\setminus\gamma}|z^{-1}(z\partial_z)^kF(z)|^2\,\widetilde d(z)^{2k-1}dA(z),
    \end{equation*}
    where $\widetilde d(z)= \bigl|\log(r_\gamma(\Arg z)/|z|)\bigr|$. 
\end{corollary}
\begin{proof}
This follows by applying Theorem~\ref{thm:abstract_sfe} with the symbol
\begin{equation*}
    h(\lambda) = \lambda^k\big(e^{-\lambda}\chi^+(\lambda) + e^{\lambda}\chi^-(\lambda)\big),
\end{equation*}
using the fact that $BD f_t = -\partial_t f_t$, and then doing the change of variables $z = r_\gamma(\theta)e^{-t}e^{i\theta}$.
\end{proof}

We now show how Corollary~\ref{cor:dirty_square_function_estimates} implies Theorem~\ref{thm:square_function_estimates}.
Note that by the product rule, $(z\partial_z)^k = \sum_{j=1}^k a_{j,k}z^j\partial_z^j$, where $a_{j,k}$ are positive integers and $a_{k,k} = 1$. 
So for $z \in \Omega^+$, we have that
\begin{align*}
(z\partial_z)^kF(z) 
= \chi_\delta(z)z^kF^{(k)}(z) + z^k(1-\chi_\delta(z))F^{(k)}(z) + \sum_{j=1}^{k-1} a_{j,k}z^j F^{(j)}(z) ,
\end{align*}
where $\chi_\delta(z) = \chi_{\Omega_\delta}(z)$ and $\Omega_\delta = \Omega^+_\delta \cup \Omega^-_\delta$.\\
\begin{lemma}
    \label{lemma:compact_remainder_terms}
    Let $k \ge 1$ be an integer and $\delta >0$. 
    Further, let $F$ denote the Cauchy integral of $f \in L_2(\gamma)$.
    Then the operators
    \begin{align*}
        K_j : L_2(\gamma) &\longrightarrow L_2(\C \setminus \gamma, \widetilde d(z)^{2k-1}dA(z)) \\
        K_jf(z) &= a_{j,k}z^{j-1} F^{(j)}(z),
    \end{align*}
    $1 \le j \le k-1$, and
    \begin{align*}
        K_k : L_2(\gamma) &\longrightarrow L_2(\C \setminus \gamma, \widetilde d(z)^{2k-1}dA(z)) \\
        K_kf(z) &= z^{k-1}(1-\chi_\delta(z))F^{(k)}(z)
    \end{align*}
     are Hilbert-Schmidt.
\end{lemma}
\begin{proof}
We have that $(K_j)_{j = 1}^k$ are integral operators with kernel estimates
\begin{align*}
    |\widetilde K_j(z,w)| &\simeq |z|^{j-1}|w-z|^{-j-1}, \quad \quad 1 \le j \le k-1, \\
    |\widetilde K_k(z,w)| &\simeq (1-\chi_\delta(z))|z|^{k-1}|w-z|^{-k-1}.
\end{align*}
Now, we have for $1 \le j \le k$, $w \in \gamma$, and $z \in \Omega^+ \setminus \Omega^+_\delta$, that
\[|\widetilde K_{j}(z,w)|^2  \simeq  |z|^{2j-2} |w-z|^{-2j-2}  \lesssim 1,\]
uniformly in $z$ and $w$. Likewise, for $z \in \Omega^- \setminus \Omega^-_\delta$, we have that
\[|\widetilde K_{j}(z,w)|^2 \widetilde d(z)^{2k-1} \simeq  |z|^{2j-2} |w-z|^{-2j-2} \widetilde d(z)^{2k-1} \simeq |z|^{-4}\widetilde d(z)^{2k-1} \lesssim |z|^{-3},\]
uniformly in $z$ and $w$. 
Finally, for $1 \le j \le k-1$, we have that
\begin{multline*}
    \int_{\Omega_\delta} |z|^{2j-2} \int_\gamma |w-z|^{-2j-2}\,|dw|\, \widetilde d(z)^{2k-1}dA(z) \\
    \simeq \int_{\Omega_\delta} d(z)^{-2j-1} \,\widetilde d(z)^{2k-1}dA(z) 
     \simeq \int_{\Omega_\delta} d(z)^{2k-2j-2} dA(z) 
    < \infty.
\end{multline*}
Hence
\[    \int_\gamma \int_{\C\setminus \gamma}  |\widetilde K_{j}(z,w)|^2\, \widetilde d(z)^{2k-1} dA(z) |dw| < \infty, \quad j = 1,...,k,\]
so $K_j$, $j = 1,...,k$, are Hilbert-Schmidt operators.
\end{proof}
\begin{proof}[Proof of Theorem~\ref{thm:square_function_estimates}]  
Let $r_0 = \min_{\theta \in [0,2\pi]} r_\gamma(\theta)$ and suppose first $\delta < r_0$. Define $\widetilde C_k : L_2(\gamma) \to L_2(\C\setminus \gamma, \widetilde d(z)^{2k-1}dA(z))$, $\widetilde C_kf(z) = z^{-1}(z\partial_z)^k F(z)$. Then Corollary~\ref{cor:dirty_square_function_estimates} implies that  $\dim \ker (\widetilde C_k) = 1$ and that $\widetilde C_k$ is bounded and has closed range. 
To see that the range of $\widetilde C_k$ is closed, note that for any Cauchy sequence $g_n \in \ran (\widetilde C_k)$, we may take $f_n \in \H$ so that $\widetilde C_kf_n = g_n$. 
But then, by Corollary~\ref{cor:dirty_square_function_estimates}, it follows that
\[\|f_m - f_n\| \simeq \|\widetilde C_k(f_m - f_n)\| = \|g_m-g_n\| \to 0, \quad m,n \to \infty,\]
so $f_n \to f$ for some $f \in \H$, and consequently $g_n \to \widetilde C_kf \in \ran (\widetilde C_k)$. 

Then $C_k:=\widetilde C_k - \sum_{j=1}^kK_j$ is a semi-Fredholm operator by Lemma~\ref{lemma:compact_remainder_terms} and perturbation theory for Fredholm operators.
That is,
\[C_kf(z) = \chi_\delta(z)z^{k-1}F^{(k)}(z)\]
is a closed-range operator with a finite-dimensional nullspace. 
It is immediate from Cauchy's integral theorem that $C_kf=0$ if $f$ is a polynomial of degree less than $k$.
On the other hand, if $C_kf = 0$ for some $f \in L_2(\gamma)$, then the Cauchy integral $F$ of $f$ is a polynomial $p^\pm$ on $\Omega^\pm$ of degree $\leq k-1$. 
Since $|p^-(z)| \lesssim \int_\gamma \frac{|f(w)|}{|w-z|^{k+1}}\,|dw| \to 0$ as $|z| \to \infty$, it follows that $p^- = 0$. 
By Theorem~\ref{thm:hardy_splitting}, $f$ has a unique Hardy decomposition $f = f^+ + f^-$, so $f = p^+ - p^- = p^+$.
We conclude that
\[\ker (C_k) = \{p \in L_2(\gamma)\:;p \text{ is a polynomial of degree }\le k-1\}.\]
Now, since $C_k$ has closed range, we have by the bounded inverse theorem that
\[\|f\| \simeq \|C_kf\| \]
for all $f \in (\ker (C_k))^\perp$. 
Consequently,
\[\|f\|^2 \simeq \|C_kf\|^2 + \sum_{j=0}^{k-1}|F^{(j)}(0)|^2\]
for all $f \in L_2(\gamma)$. 
% , since $(w^j)_0^{k-1}$ is a basis of $\ker (C_k)$. 
Since $\delta < r_0$, we have $|z| \simeq 1$ and $\widetilde d(z) \simeq d(z)$ on $\Omega^+_\delta$ and $\Omega^-_\delta$, hence
\begin{align*}
    \|C_kf\|^2 
    % &= \int_{\Omega^+_\delta}|z^{k-1}F^{(k)}(z)|^2\,\widetilde d(z)^{2k-1}dA(z) + \int_{\Omega^-_\delta}|z^{k-1}F^{(k)}(z)|^2\,\widetilde d(z)^{2k-1}dA(z) \\
    \simeq \int_{\Omega^+_\delta}|F^{(k)}(z)|^2\,d(z)^{2k-1}dA(z) + \int_{\Omega^-_\delta}|F^{(k)}(z)|^2\, d(z)^{2k-1}dA(z).
\end{align*}
This shows Theorem~\ref{thm:square_function_estimates} when $\delta < r_0$. 

Suppose now $r_0 \leq \delta \leq \infty$. Then for every $f \in L_2(\gamma)$,
\begin{multline*}
    \|f\|^2  \simeq \int_{\Omega^+_{r_0/2}}|F^{(k)}(z)|^2\,d(z)^{2k-1}dA(z) + \int_{\Omega^-_{r_0/2}}|F^{(k)}(z)|^2\, d(z)^{2k-1}dA(z) \\
    + \sum_{j=0}^{k-1}|F^{(j)}(0)|^2
     \leq \int_{\Omega^+_\delta}|F^{(k)}(z)|^2\,d(z)^{2k-1}dA(z) \\
     + \int_{\Omega^-_\delta}|F^{(k)}(z)|^2\, d(z)^{2k-1}dA(z) + \sum_{j=0}^{k-1}|F^{(j)}(0)|^2,
\end{multline*}
showing that the estimate $\lesssim$ holds. Let $U^\pm_\delta = \Omega_\delta^\pm \setminus \Omega_{r_0/2}^\pm$. We claim that the mappings $f \mapsto \chi_{U^\pm_\delta}F^{(k)}$ from $L_2(\gamma)$ to $L_2(\C\setminus\gamma, d(z)^{2k-1})$ are bounded. For any $\delta$, $f \mapsto \chi_{U^+_\delta}F^{(k)}$ is bounded since $d(z)$ and the kernel of $F^{(k)}$ are bounded on $U^+_\delta$. If $\delta < \infty$, the same is true for $f \mapsto \chi_{U^-_\delta}F^{(k)}$. In the case of $\delta = \infty$, we have
\begin{align*}
    \int_{U_\delta^-}|F^{(k)}(z)|^2d(z)^{2k-1}\,dA(z) & \leq \|f\|^2\int_{U_\delta^-}\int_\gamma \frac{1}{|w-z|^{2k+2}}\,|dw| d(z)^{2k-1}\,dA(z) \\
    & \lesssim \|f\|^2\int_{\Omega^-}|z|^{-3}\,dA(z) \lesssim \|f\|^2.
\end{align*}
Thus, $f \mapsto \chi_{U^-_\delta}F^{(k)}$ is bounded for any $\delta$. Writing $\Omega^\pm_\delta = \Omega_{r_0/2}^\pm \cup U^\pm_\delta$, it follows that

\begin{multline*}
     \int_{\Omega^+_\delta}|F^{(k)}(z)|^2\,d(z)^{2k-1}dA(z) + \int_{\Omega^-_\delta}|F^{(k)}(z)|^2\, d(z)^{2k-1}dA(z) + \sum_{j=0}^{k-1}|F^{(j)}(0)|^2 \\
    \lesssim  \int_{\Omega^+_{r_0/2}}|F^{(k)}(z)|^2\,d(z)^{2k-1}dA(z) + \int_{\Omega^-_{r_0/2}}|F^{(k)}(z)|^2\, d(z)^{2k-1}dA(z) \\
    + \sum_{j=0}^{k-1}|F^{(j)}(0)|^2 + \|f\|^2 \lesssim  \|f\|^2.
\end{multline*}
We conclude that Theorem~\ref{thm:square_function_estimates} holds for each fixed $\delta \in (0,\infty]$.  
\end{proof}

\section{Whitney sets}
\label{appendix:whitney_partition}
The goal of this appendix is prove some details regarding the existence of Whitney sets, and to show that the explicit set $\cup_{l = 0}^\infty \{q_{j,l}\}_{j = 1}^{n_l}$ of Example~\ref{ex:explicit_whitney_partition} is a Whitney set for $\Omega^-_{1/\mathcal{L}}$ with respect to $\gamma$. The latter objective is accomplished by constructing a certain Whitney set for $\mathbb{D}^-_1 = \{z : 1<|z|<2\}$ with respect to the unit circle $\T$, and then mapping this set to $\Omega^-$ using the bi-Lipschitz map $re^{i\theta} \mapsto rr_\gamma(\theta)e^{i\theta}$. The next lemma gives a sufficient condition for Condition~\ref{item:whitney_set_cond_2} of Definition~\ref{def:whitney_set} to be satisfied.
\begin{lemma}\label{lemma:whitney_covering_lemma}
    Let $A \subset \R^n$ be closed, and $\{x_j\}_{j = 1}^\infty \subset A^c$. Fix $\varepsilon \leq 1/4$ and $\sigma > 1$, and let $B_j = B(x_j,\varepsilon d(x_j,A))$. Suppose $\{E_j\}_{j = 1}^\infty$ is a collection of measurable sets with disjoint interiors such that $E_j \subset B_j$, $m(\partial E_j) = 0$, and $m(E_j) \geq \frac{1}{\sigma}m(B_j)$. Then $\sum_{j = 1}^\infty \chi_{B_j}(x) \leq 4^n\sigma$ for all $x \in \R^n$.
\end{lemma}
\begin{proof}
    Let $d(x) = d(x,A)$. We need to show that for a given $x \in \cup_{j = 1}^\infty B_j$, $x$ is contained in at most $4^n\sigma$ balls. After relabeling, we can assume that $x \in \cap_{j = 1}^J B_{j}$, for some $1 \leq J \leq \infty$. For each $j$, we have $|x_j-x| \le \varepsilon d(x_j) \le \varepsilon (|x_j - x| + d(x))$. Thus, 
    \begin{align*}
        |x_j-x| \le \frac{\varepsilon}{1-\varepsilon }d(x) \leq \frac{4}{3}\varepsilon d(x),
    \end{align*}
    which implies in particular $\frac{2}{3}d(x) \leq d(x_j) \leq \frac{4}{3}d(x)$. It follows that $E_{j} \subset B(x,\frac{8}{3}\varepsilon d(x))$ for every $j$. We therefore have
    \begin{align}\label{eq:cov_lem_2}
         \sum_{j = 1}^J m(E_{j}) = m\Bigg(\bigcup_{j = 1}^J E_{j}\Bigg) \leq \alpha(n)\bigg(\frac{8}{3}\varepsilon d(x)\bigg)^n,  
    \end{align}
    where $\alpha(n)$ is the volume of the unit ball. By assumption, we have that $m(E_{j}) \geq \frac{1}{\sigma}\alpha(n)(\varepsilon d(x_j))^n$, which combined with $d(x_j) \geq \frac{2}{3}d(x)$ and \eqref{eq:cov_lem_2} gives
    \begin{equation*}  
         J\alpha(n)\bigg(\frac{2}{3} \varepsilon d(x) \bigg)^n \frac{1}{\sigma}\leq\alpha(n)\bigg(\frac{8}{3}\varepsilon d(x)\bigg)^n \quad \iff \quad J \leq 4^n\sigma. \qedhere
    \end{equation*}
\end{proof}
Next follows two lemmata, which are used in Section~\ref{sec:discrete_frames}. \begin{lemma}\label{lem:existence_of_whitney_sets}
    Let $A \subset \R^n$ be closed, and let $\O \subset A^c$. Let $\varepsilon \leq 1/8$. There exists an $\varepsilon$-fine Whitney set $C \subset \O$ for $\O$ with respect to $A$ such that $\sum_{x \in C}\chi_{B(x,2\varepsilon d(x))}(y) \leq 3(40)^n$ for all $y \in \R^n$.
\end{lemma}
\begin{proof}
    Let $d(x) = d(x,A)$, $V_0 = \{x \in A^c : d(x) \leq 1\}$, and for $k \geq 1$ let
    \begin{align*}
        V_k = \{x \in A^c ; \beta^{k-1} < d(x) \leq \beta^{k}\}, \quad \beta = \frac{1}{1-2\varepsilon}.
    \end{align*}
    Let $k \geq 0$ and $l \geq k+2$. Let $x \in V_l$ and $x' \in B(x, 2\varepsilon d(x))$. Then
    \begin{align*}
        d(x') \geq d(x) - |x'-x| > (1-2\varepsilon)d(x) > \beta^{l-2} \geq \beta^{k},
    \end{align*}
    showing that $B(x,2\varepsilon d(x)) \cap V_k = \emptyset$. Let $k \geq 2$ and $0 \leq l \leq k-2$, then
    \begin{align*}
        d(x') \leq d(x) + |x'-x| < (1+ 2\varepsilon)d(x) \leq (1+2\varepsilon) \beta^{l} < \beta^{k-1},
    \end{align*}
    showing again that $B(x, 2\varepsilon d(x)) \cap V_k = \emptyset$. We conclude that $B(x, 2\varepsilon d(x)) \cap V_k = \emptyset$ whenever $x \in V_l$ and $l \notin [k-1,k+1]$. 

    Let $\F_k = \{B(x,\frac{1}{5}\varepsilon d(x)): x\in \O \cap V_k \}$, $k \geq 0$. Since $\sup\{\textrm{diam}\,B : B \in \F_k\} < \infty$, Vitali's covering theorem (see Evans and Gariepy~\cite[Thm~1.24]{EvansGariepy_2015}) shows that there exists a countable set $C_k \subset \O \cap V_k$ such that the collection $\{B(x,\frac{1}{5}\varepsilon d(x))\}_{x \in C_k}$ is disjoint and $\O \cap V_k \subset \cup_{x \in C_k} B(x,\varepsilon d(x))$. Let $C = \cup_k C_k$. Then $\O = \cup_{k = 0}^\infty \O \cap V_k \subset \cup_{x \in C} B(x,\varepsilon d(x))$.
    
    By Lemma~\ref{lemma:whitney_covering_lemma}, $\sum_{x \in C_k} \chi_{B(x,2\varepsilon d(x))} \leq (40)^n$. Let $y \in A^c$. Then $y \in V_k$ for some $k \geq 0$. Thus, 
    \begin{multline*}
        \sum_{x \in C}\chi_{B(x,2\varepsilon d(x))}(y) = \sum_{l = 0}^\infty\sum_{x \in C_l} \chi_{B(x,2\varepsilon d(x))}(y) \\
        = \sum_{l = k-1}^{k+1}\sum_{x \in C_l} \chi_{B(x,2\varepsilon d(x))}(y) \leq 3(40)^n
    \end{multline*}
    whenever $k \geq 1$. If $k = 0$, the sum is bounded by $2(40)^n$. This proves the claimed bound on the sum, and that $C$ is an $\varepsilon$-fine Whitney set for $\O$ with respect to $A$.
\end{proof}

\begin{lemma}\label{lem:extracting_sub_whitney_sets}
    Let $\varepsilon \leq 1/20$. Let $A \subset \R^n$ be closed, and let $\O \subset A^c$, and suppose $\{x_j\}_{j = 1}^\infty$ is an $\varepsilon$-fine Whitney set for $\O$ with respect to $A$. There exists $\J \subset \N$ such that $\{x_j\}_{j \in \J}$ is a $5\varepsilon$-fine Whitney set for $\O$ with respect to $A$, with covering constant $N \leq 3(20)^n$
\end{lemma}
\begin{proof}
    Let $d(x) = d(x,A)$. Define the sets $V_k$ as in Lemma~\ref{lem:existence_of_whitney_sets}, but with $2\varepsilon$ replaced by $5\varepsilon$. Let also $\I_k = \{j  \in \N : x_j\in V_k\}$, $k \geq 0$. Let $B_j = B(x_j,\varepsilon d(x_j))$ and $B_j' = B(x_j,5\varepsilon d(x_j))$. As in the proof of Lemma~\ref{lem:existence_of_whitney_sets}, we see that $B_j' \cap V_k = \emptyset$ whenever $j \in \I_l$ and $l \notin [k-1,k+1]$. Since $\sup\{\textrm{diam}\,B_j : j \in \I_k\} < \infty$, Vitali's covering theorem shows that there exists $\J_k \subset \I_k$ such that the collection $\{B_j\}_{j \in \J_k}$ is disjoint, and $\cup_{j \in \I_k}B_j \subset \cup_{j \in \J_k}B_j'$. Let $\J = \cup_k \J_k$. By Lemma~\ref{lemma:whitney_covering_lemma}, $\sum_{j \in \J_k} \chi_{B_j'}(y) \leq (20)^n$ for every $y \in \R^n$, and as in the proof of Lemma~\ref{lem:existence_of_whitney_sets}, we see that $\sum_{j \in \J}\chi_{B_j'}(y) \leq 3(20)^n$ for every $y \in \R^n$. Since $\O \subset \cup_{k = 0}^\infty \cup_{j \in \I_k} B_j \subset \cup_{j \in \J}B_j'$, it follows that $\{x_j\}_{j \in \J}$ is an $5\varepsilon$-fine Whitney set for $\O$ with respect to $A$, with covering constant $N \leq 3(20)^n$.
\end{proof}
The following lemma is used to show that the set $\cup_{l = 0}^\infty \{q_{j,l}\}_{j = 1}^{n_l}$ of Example~\ref{ex:explicit_whitney_partition} is a Whitney set for $\Omega^-_{1/\mathcal{L}}$ with respect to $\gamma$.
\begin{lemma}\label{lem:whitney_set_under_bilip_map}
    Let $\rho : \R^n \to \R^n$ be a bi-Lipschitz map, and define $\mathcal{L} = \max(\textup{Lip}(\rho),\textup{Lip}(\rho^{-1}))$. Let $A \subset \R^n$ be closed, and suppose $\varepsilon$, $\sigma$, $\{x_j\}_{j = 1}^\infty$, $\{B_j\}_{j = 1}^\infty$, and $\{E_j\}_{j = 1}^\infty$ are given as in Lemma~\ref{lemma:whitney_covering_lemma}. If $\mathcal{L}^2\varepsilon \leq 1/4$, then $\{\rho(x_j)\}_{j = 1}^\infty$ is a Whitney set for $\rho(\cup_{j = 1}^\infty B_j)$ with respect to $\rho(A)$, with parameters $(\mathcal{L}^2\varepsilon, N)$, where $N \leq 4^n \mathcal{L}^{3n+1}\sigma$.    
\end{lemma}
\begin{proof}
    Let $d(x) = d(x,A)$, $d'(y) = d(y,\rho(A))$, $y_j = \rho(x_j)$, and set $B_j' = B(y_j, \mathcal{L}^2\varepsilon d'(y_j))$. Suppose $y \in \rho(\cup_{j = 1}^\infty B_j)$, with $y = \rho(x)$, $x \in B_j$. Then
    \begin{align}\label{eq:new_balls_contain_U}
        |y-y_j| \leq \mathcal{L}|x-x_j| \leq \mathcal{L}\varepsilon d(x_j) \leq \mathcal{L}^2\varepsilon d'(y_j).
    \end{align}
    Thus, $y \in B'_j$, showing that $\rho(\cup_{j = 1}^\infty B_j) \subset \cup_j B_j'$. Condition~\ref{def:whitney_set_cond_1} of Definition~\ref{def:whitney_set} is therefore satisfied. 

    Next, since $\rho$ is bi-Lipschitz, $\rho(E_j)$ is measurable, $\rho(E_j^o) = \rho(E_j)^o$ and $\rho(\partial E_j) = \partial \rho(E_j)$. Thus $\{\rho(E_j)\}_{j = 1}^\infty$ is a collection of measurable sets with disjoint interiors satisfying $m(\partial\rho(E_j)) = m(\rho(\partial E_j)) \leq \mathcal{L}m(\partial E_j) = 0$. If $y = \rho(x) \in \rho(E_j)$, then since $x \in E_j \subset B_j$,~\eqref{eq:new_balls_contain_U} shows that $\rho(E_j) \subset B_j'$. Moreover,
    \begin{multline*}
        m(\rho(E_j))  \geq \frac{1}{\mathcal{L}}m(E_j) \geq \frac{1}{\mathcal{L}\sigma} \alpha(n) (\varepsilon d(x_j))^n  \\
        \geq \frac{1}{\mathcal{L}\sigma} \alpha(n) \bigg(\varepsilon \frac{1}{\mathcal{L}}d'(y_j)\bigg)^n 
         = \frac{1}{\mathcal{L}^{3n+1}\sigma}m(B_j').
    \end{multline*}
    Lemma~\ref{lemma:whitney_covering_lemma} thus implies that $\sum_{j = 1}^\infty\chi_{B_j'}(x) \leq 4^n \mathcal{L}^{3n+1}\sigma$, showing that Condition~\ref{item:whitney_set_cond_2} of Definition~\ref{def:whitney_set} is also satisfied, with the claimed estimate for the covering constant $N$.
\end{proof}
\begin{figure}
     \centering
         \includegraphics[trim={0 0 14cm 0},clip,width=0.5\textwidth]{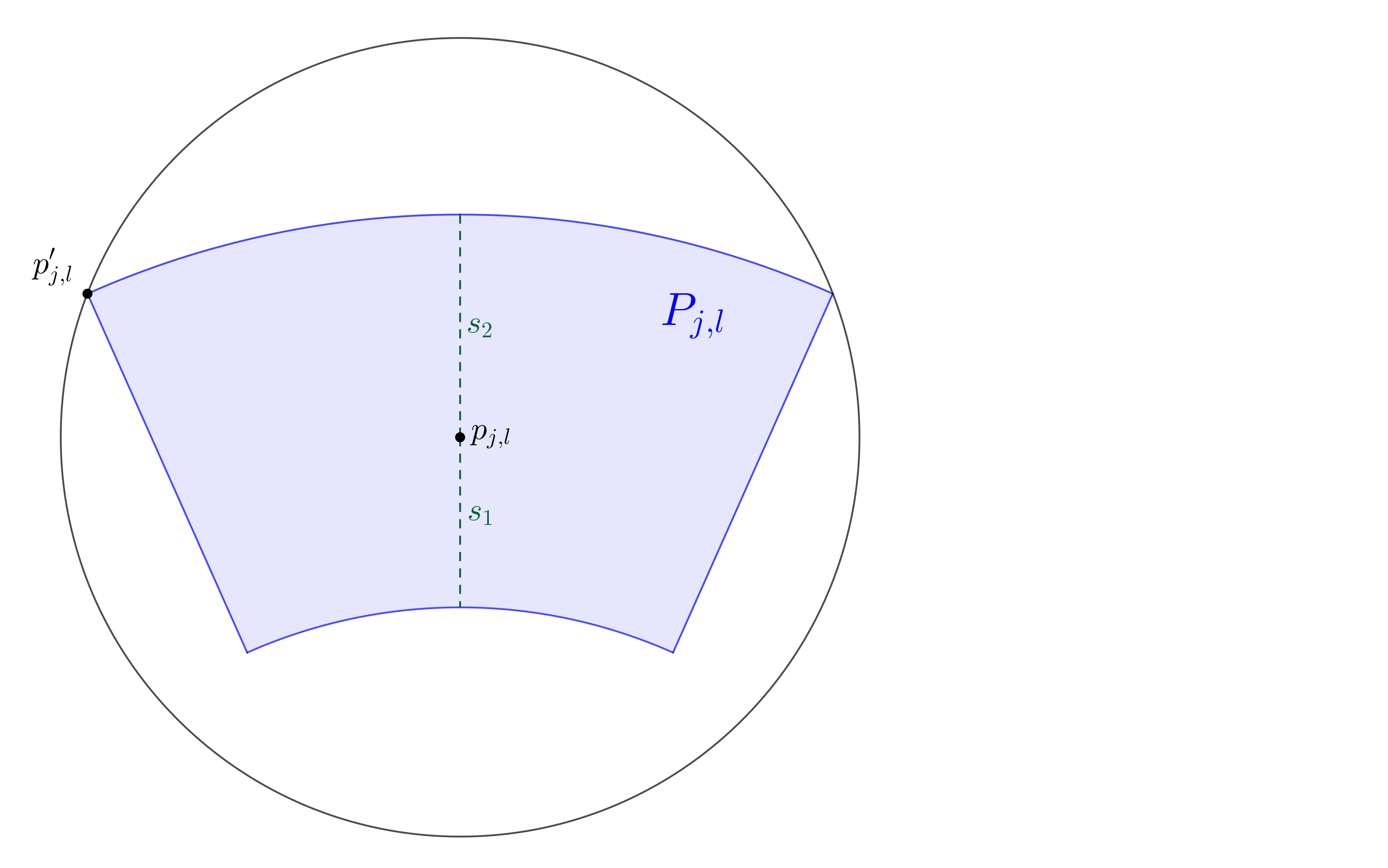}
         \caption{The annular sectors $P_{j,l}$ are shown together with their approximate midpoints $p_{j,l}$. The upper-left corner $p_{j,l}'$ of $P_{j,l}$ is the point in $P_{j,l}$ (together with the upper-right corner) that is farthest away from $p_{j,l}$. The numbers $s_1$ and $s_2$ are the lengths of the corresponding dashed lines.}
         \label{fig:annular_sec}
\end{figure}
We now turn attention back to Example~\ref{ex:explicit_whitney_partition}. Let $\varepsilon$, $\varepsilon'$, $\rho$, $\mathcal{L}$, $n_l$, $r_l$, and $\theta_{j,l}$, be as in that example.
Let $R_{-1} = (4+\varepsilon)/2$, and for $l \geq 0$ and $1 \leq j \leq n_l$, let
\begin{gather*}
    p_{j,l} = r_l e^{i\theta_{j,l}}, \quad R_l = \frac{r_l + r_{l+1}}{2}, \quad \varphi_{j,l} = -\pi + \frac{2\pi}{n_l}(j - 1/2), \\
    P_{j,l} = \{re^{i\varphi} : R_l\leq r \leq R_{l-1}, \varphi_{j,l} \leq \varphi \leq \varphi_{j+1,l} \}.
\end{gather*}
The sets $P_{j,l}$ are the "annular sectors" shown in Figure~\ref{fig:annular_sec}, and the points $p_{j,l}$ are approximately "midpoints" of $P_{j,l}$. Our plan is for the $P_{j,l}$:s to play the role of the sets $E_j$ in Lemma~\ref{lem:whitney_set_under_bilip_map}. As such, we need to verify the hypotheses on $P_{j,l}$ in Lemma~\ref{lem:whitney_set_under_bilip_map}. Because the geometry of these sets is very simple, only elementary computations are required for this, so we do not spell out all the details.

Clearly, $\cup_{l = 0}^\infty \{P_{j,l}\}_{j = 1}^{n_l}$ is a collection of measurable sets with disjoint interiors satisfying $m(\partial P_{j,l}) = 0$. For each $j$ and $l$, let $B_{j,l}=B(p_{j,l},\frac{3}{2}\varepsilon d(p_{j,l},\T))$. We claim that $P_{j,l} \subset B_{j,l}$, and that $m(P_{j,l}) \gtrsim m(B_{j,l})$ uniformly in $j$, $l$, and $\varepsilon$. Fix $j$, $l$, and let $s_1 = r_l - R_{l}$ and $s_2 = R_{l-1} - r_l$. Using that $s_2 > s_1$ and $2\pi/n_l < \pi$, one verifies, by parameterizing each boundary segment of $P_{j,l}$, that
\begin{equation*}
    \max_{z \in P_{j,l}}|z-p_{j,l}| = |p_{j,l}' - p_{j,l}| = \bigg[(R_{l-1}-r_l)^2 + 2r_lR_{l-1}\bigg(1-\cos \frac{\pi}{n_l}\bigg)\bigg]^{1/2},
\end{equation*}
where $p_{j,l}' = R_{l-1}e^{i\varphi_{j+1,l}}$; see Figure~\ref{fig:annular_sec}. By estimating $1 - \cos (\pi /n_l) < (\pi/n_l)^2/2$, and bounding $n_l$ below by $2\pi \varepsilon^{-1}(1+\varepsilon)^l$, one then obtains $|p_{j,l}' - p_{j,l}| < \frac{3}{2}\varepsilon d(p_{j,l},\T)$. Thus, $P_{j,l}\subset B_{j,l}$. Furthermore, by bounding $n_l$ above by $4\pi \varepsilon^{-1}(1+\varepsilon)^l$, straightforward calculations show that
\begin{align*}
    m(P_{j,l}) = \frac{1}{n_l}[m(B(0,R_{l-1})) - m(B(0,R_l))] \geq \frac{1}{3}(\varepsilon d(p_{j,l},\T))^2 = \frac{4}{27 \pi}m(B_{j,l}).
\end{align*}
Thus, $P_{j,l}$ satisfies the hypotheses in Lemma~\ref{lem:whitney_set_under_bilip_map}. Now observe that $R_0 > 2$ implies 
\begin{align*}
    \mathbb{D}^-_1 \subset \cup_{l = 0}^\infty \cup_{j = 1}^{n_l}P_{j,l} \subset \cup_{l = 0}^\infty \cup_{j = 1}^{n_l}B_{j,l}.
\end{align*}
Thus, since $\varepsilon' = 3\mathcal{L}^2\varepsilon /2 \leq 1/4$ by assumption, we may apply  Lemma~\ref{lem:whitney_set_under_bilip_map} with 
\begin{gather*}
    \{x_j\}_{j = 1}^\infty = \cup_{l = 0}^\infty \{p_{j,l}\}_{j = 1}^{n_l}, \quad \{E_j\}_{j = 1}^\infty = \cup_{l = 0}^\infty \{P_{j,l}\}_{j = 1}^{n_l}, \\
    \varepsilon = \frac{3}{2}\varepsilon, \quad \text{and} \quad\sigma = \frac{4}{27\pi},
\end{gather*}
to obtain an $\varepsilon'$-fine Whitney set $\cup_{l = 0}^\infty \{\rho(p_{j,l})\}_{j = 1}^{n_l}$ for $\rho(\mathbb{D}^-_1)$ with respect to $\rho(\T)$. Since $q_{j,l} = \rho(p_{j,l})$, $\gamma = \rho(\T)$, and $\Omega^-_{1/\mathcal{L}} \subset \rho(\mathbb{D}^-_1)$, we conclude that $\cup_{l = 0}^\infty \{q_{j,l}\}_{j = 1}^{n_l}$ is an $\varepsilon'$-fine Whitney set for $\Omega^-_{1/\mathcal{L}}$ with respect to $\gamma$. This proves the statement made in Example~\ref{ex:explicit_whitney_partition}.
\end{appendices}
\section*{Data availability}
Data will be made available on request.
\bibliographystyle{plain}
\bibliography{ref}

\end{document}